\documentclass[pdflatex,sn-mathphys-num]{sn-jnl}

\usepackage{graphicx}%
\usepackage{multirow}%
\usepackage{amsmath,amssymb,amsfonts}%
\usepackage{amsthm}%
\usepackage{mathrsfs}%
\usepackage[title]{appendix}%
\usepackage{xcolor}%
\usepackage{textcomp}%
\usepackage{manyfoot}%
\usepackage{booktabs}%
\usepackage{algorithm}%
\usepackage{algorithmicx}%
\usepackage{algpseudocode}%
\usepackage{listings}%

\usepackage[utf8]{inputenc}

\usepackage[normalem]{ulem}   

\theoremstyle{thmstyleone}%
\newtheorem{theorem}{Theorem}[section]

\newtheorem{lemma}[theorem]{Lemma}
\newtheorem{corollary}[theorem]{Corollary} 

\theoremstyle{thmstyletwo}%
\newtheorem{rmk}{Remark}%

\theoremstyle{thmstylethree}%
\newtheorem{cona}{Conditions}
\newtheorem{conb}{Conditions}[cona]
\newenvironment{con}[1]{
	\renewcommand\theconb{#1}
	\conb
}{\endconb}

\newcommand{\R}{{\mathbb R}}
\newcommand{\Z}{{\mathbb Z}}

\renewcommand{\div}{\text{$\nabla\cdot$}}
\newcommand{\grad}{\text{$\nabla$}}
\newcommand{\bs}[1]{{\boldsymbol  #1}}
\newcommand{\bb}[1]{{\mathbf  #1}}

\newcommand{\Af}{A}
\newcommand{\N}{N}
\newcommand{\HH}{{\mathcal H}}

\newtheorem{exa}{Example}[section]

\usepackage{hyperref}
\hypersetup{
	colorlinks=true,
	linkcolor=blue,
	citecolor=cyan,
	filecolor=magenta,      
	urlcolor=cyan,
}

\DeclareMathOperator{\supp}{supp}

\allowdisplaybreaks  

\begin{document}

\title[Uniqueness for TV-regularized inverse problems in divergence form]{Uniqueness for TV-regularized  inverse problems\\
  with source in divergence form}

\author[1]{\fnm{Laurent} \sur{Baratchart}}\email{Laurent.Baratchart@inria.fr}
\equalcont{These authors contributed equally to this work.}

\author[2]{\fnm{Douglas P.} \sur{Hardin}}\email{doug.hardin@vanderbilt.edu}
\equalcont{These authors contributed equally to this work.}

\author*[3]{\fnm{Cristóbal} \sur{Villalobos Guillén}}\email{cristobal.villalobos.guillen@univie.ac.at}
\equalcont{These authors contributed equally to this work.}

\affil[1]{\orgdiv{Projet FACTAS}, \orgname{INRIA}, \orgaddress{\street{2004 route des Lucioles, BP 93}, \city{Sophia-Antipolis}, \postcode{06902 Cedex}, \country{FRANCE}}}

\affil[2]{\orgdiv{Department of Mathematics}, \orgname{Vanderbilt University}, \orgaddress{\city{Nashville}, \postcode{37240}, \state{TN}, \country{USA}}}

\affil*[3]{\orgdiv{Faculty of Mathematics}, \orgname{University of Vienna}, \orgaddress{\street{Oskar-Morgenstern-Platz 1}, \city{Wien}, \postcode{1090}, \country{Austria}}}

\abstract{
Inverse  source problems in divergence form consist in finding a vector field
with prescribed support S, whose divergence is the Laplacian of some observed potential.
In this paper, we assume
the unknown vector field is a vector-valued measure, and we study the corresponding
least-squares inversion problems, regularized by penalizing the total variation, without
discretizing the criterion nor the unknown. We prove that this problem has a unique minimizer in the case where S is a slender set; i.e., it has zero Lebesgue measure and each
connected component of its complement has infinite Lebesgue measure. The proof relies on a refinement of
 Smirnov's decomposition of divergence-free measures \cite{Smi94}  which is of independent interest.
}

\keywords{inverse problems, second order elliptic PDE, divergence-free measures, total variation regularization}

\maketitle

%
%
%

\section{Introduction}
Roughly speaking, inverse source problems with source in divergence form
consist in reconstructing 
an $\R^N$-valued function or distribution knowing the potential of its divergence (or the field of that potential).  The potential here is understood  with respect to some uniformly elliptic differential operator on $\R^N$ that may be the Euclidean Laplacian or a more general operator in divergence form, possibly with non-smooth coefficients.
From the point of view of geometric analysis,  the problem  is to
recover the unknown function or distribution from the gradient part
of its  Helmholtz
decomposition, relative to some Riemannian metric.

From a physical point of view such problems appear in various
contexts including
 Geomagnetism,  Paleomagnetism,  as well as
Medical Imaging from Electro-Encephalography (EEG) and Electro-Cardiography (ECG) (see, e.g. \cite{Blakely, Parker, backus96,
  Fuller,WLF,Sarvas,KreKuPo}). In these examples, the unknown is typically a current or  magnetization distribution, and the problem is severely ill-posed
as the solution is highly non-unique   while data can only be collected
on a limited portion of space, away from the support.
This makes prior assumptions and regularization techniques
an essential aspect of every consistent computational approach.  
In this paper, we assume that
the unknown distribution is a  vector-valued \emph{measure}
and we study, from a mathematical point of view,  the corresponding infinite-dimensional least squares inversion
problem, regularized by  penalizing  the measure-theoretic total variation
(possibly with the adjunction of a weight).

Specifically, let $\mathcal{M}(S)^N$ denote $\R^N$-valued Borel measures supported on a given compact set $S\subset\R^N$, and $\HH$  be a separable Hilbert space where observational data lies. We consider a linear {\em forward operator} $\Af:\mathcal{M}(S)^N \to \HH$ that is sequentially weak-star to strong continuous  and
whose null space consists 
of divergence-free measures, with
adjoint operator $A^*$ mapping $\mathcal{H}$ into
Lipschitz gradients, see Conditions \ref{conU} in Section~\ref{Sect2.1}; that $A^*$ is valued in gradients follows from the null-space assumption on $A$ if for instance $S$ is quasiconvex and that $A^*$ is valued in Lipschitz functions follows for instance if $\Af$ is continuous in the Kantorovich-Rubinstein norm \cite{Hanin92}.

For  $f\in  \HH$ and   $\lambda>0$,   we study  minimizers of the functional
\begin{equation}
	\label{eq|main}
	\mathcal{F}_{f,\lambda}(\bs{\mu})=\mathcal{F}(\bs{\mu}):=
	\|f-\Af\bs{\mu}\|_{\HH}^2+\lambda \|\sigma\bs{\mu}\|_{TV}
\end{equation}
where $\bs{\mu}$ ranges over  $\mathcal{M}(S)^N$,  $\|\bs{\mu}\|_{TV}$  denotes the (Euclidean) total variation of  $\bs{\mu}$, and $\sigma$ is a positive Lipschitz continuous weight function on $S$ reflecting ancillary information on
the support and magnitude of $\bs{\mu}$. In the language of inverse problems the number $\lambda$ in \eqref{eq|main} is
termed the \emph{regularization parameter}, and its choice reflects the
sensitivity  of the residuals $\|f-\Af\bs{\mu}\|_{\HH}^2$
to small changes of a candidate minimizer $\bs{\mu}$.
The importance of solving a regularized formulation like
\eqref{eq|main} to approach inverse problems
in spaces of measures (possibly with inaccurate data $f$) has been much
emphasized in recent years; see for example  \cite{ModRegReview2018}.

We remark
that $\mathcal{F}$ is not strictly convex when $A$ has a nontrivial null space,
which happens whenever $S$ supports nontrivial divergence-free measures.  The main result in this paper  (Theorem~\ref{thm|uniqueness}) asserts that under the previous assumptions, the minimizer of $\mathcal{F}$ is unique in spite of this
lack of strict convexity.  The proof makes use of another result
of independent interest (Theorem~\ref{prop|smi_tra}) which goes as follows.
If
$\bs{\mu}\in\mathcal{M}(S)^N$ then the Radon-Nikodym derivative
$d\bs{\mu}/d|\bs{\mu}|$ exists, where $|\bs{\mu}|$
denotes the total variation measure; the representation $d\bs{\mu}=(d\bs{\mu}/d|\bs{\mu}|)d|\bs{\mu}|$ is the \emph{polar decomposition} of $\bs{\mu}$.
  Now, if $\bs{\mu}$ 
is divergence free, then the unit-norm vector field $d\bs{\mu}/d|\bs{\mu}|$,
which is defined $|\bs{\mu}|$-almost everywhere,  admits Lipschitz continuous integral curves  through $|\bs{\mu}|$-almost every point (such curves are generally non-unique).
The demonstration relies on Smirnov's decomposition of divergence-free
measures \cite{Smi94}.

The relevance of our main result to  the inverse problem is that uniqueness of a minimizer is key to optimize \eqref{eq|main}, for it turns subsequentially convergent algorithms into sequentially convergent ones\footnote{
For example the draft at https://hal.science/hal-05008293
sets up weak-star subsequentially convergent
discretization schemes based on the LASSO to solve \eqref{eq|main}, and
by  Theorem \ref{thm|uniqueness} they are
in fact weak-star sequentially convergent.}.
In the particular case where $S$ is contained in a plane and the forward operator $A$ is the $x_3$-derivative of the Newton potential of $\mathrm{div}\bs{\mu}$,
uniqueness of a minimizer had already been proven in \cite{BVH} when $\sigma=1$.

In connection with the previous discussion, it is worth pointing out
a property which is more specific to our case: if $\bs{\mu}$ has purely 1-unrectifiable support
and  $f_n$ tends to $A\bs{\mu}$ while $\lambda_n$ tends to $0$
with  $\|A\bs{\mu}-f_n\|_{\mathcal{H}}/\lambda_n^{1/2}\to0$, then
`the minimizers $\bs{\mu}_n$  of $\mathcal{F}_{f_n,\lambda_n}$
converge weak-star to $\bs{\mu}$ and the total variation measures 
$|\bs{\mu}_n|$ to $|\bs{\mu}|$.
This is established in \cite[Theorem 4.3]{BVHNS} (when $\sigma\equiv1$) for inverse source problems in divergence form of the Laplacian when the source is
supported on a slender set (see further below), and the proof carries over at once to our more general setting. In view of \cite[Equation (37)]{BVHNS},
an equivalent nonasymptotic formulation is: if $\mathrm{d}$ denotes any metric defining the weak-star topology
on the closed ball centered at $\bs{0}$ with radius $1+\|\bs{\mu}\|_{TV}$ in
$\mathcal{M}(S)^N$, then to each $\varepsilon>0$ there are $\lambda_0>0$ and
$\eta\in(0,1)$ such that $\max\{\mathrm{d}(\bs{\mu},\bs{\mu}_{f,\lambda}),\mathrm{d}(|\bs{\mu}|,|\bs{\mu}_{f,\lambda}|)\}<\varepsilon$ as soon as $\lambda<\lambda_0$
    and $\|f-A\bs{\mu}\|<\eta\lambda^{1/2}$. In other words, 
minimizing \eqref{eq|main} while applying the Morozov discrepancy rule
is a weak-star consistent regularization method  with respect to both the
unknown measure and its total variation, at least  when the unknown measure has purely 1-unrectifiable support. The latter condition may be viewed as a definition of sparsity in the present,  infinite-dimensional context (note that it
includes  measures supported in countably many points), and the sparse recovery result just mentioned
motivates our study of criterion \eqref{eq|main}.

In the electromagnetic applications mentioned
earlier,  the forward operator $A$ arises from an elliptic partial differential
equation and the  observational data is  the restriction of
the solution (or its gradient)  to a portion of the space. To address this case,
one needs to devise assumptions on $S$ and the coefficients of the equation that
imply Conditions~\ref{conU}, so as to be able to use  
Theorem~\ref{thm|uniqueness}. This piece of elliptic
regularity theory is perhaps non-standard, as the authors could not find
in the literature a discussion of Equation \eqref{eqelf} for  non-smooth coefficients
when $\bs{\mu}$ is a $\R^N$-valued \emph{measure}.  Appendix \ref{A} contains an existence and uniqueness result 
with some corollaries 
for very weak solutions  on $\R^N$ to this equation  that vanish at infinity (Theorem \ref{soldim1}),  which may be interesting in their own.  They team up with
known unique continuation results in Appendix \ref{UC}
to produce assumptions  on the geometry and the
observational data,  as well as smoothness requirements on Equation \eqref{eqelf}
that imply Conditions \ref{conU} and are in particular
met in two scenarios that can accommodate inverse EEG and scanning magnetic microscopy problems respectively \cite{DLMT,LiWe,FuLiVoTru};
see Conditions \ref{conE} and Scenarios (S.1), (S.2) in Section \ref{CondEUsec}.
Conditions \ref{conE} require that the set $S$ supporting the
unknown source is \emph{slender}; the notion of slenderness appears in \cite{BVHNS} which deals with inverse source problems for
the Laplacian, and it is instrumental to make sure that
the null space of $A$ is comprised only of divergence-free measures.
In the present setting where $S$ is compact, slenderness amounts to
$S$ having Lebesgue measure zero and connected complement.

As shown in Theorem \ref{soldim1} under suitable regularity assumptions,
the solution to
Equation \eqref{eqelf}  can be expressed  as a Bochner integral of the
gradient of the fundamental solution against $d\bs{\mu}$. This allows one to recast our
forward operator as a Bochner  integral operator with  kernel, and this has
become a
common framework in recent, \emph{off-the-grid} approaches to sparse
reconstruction of scalar measures; see \cite{DuPe,CaDuPe,LaBlAu,Chizat}. Compared to these references, that typically
aim at recovering a finite linear combination of Dirac masses, the
kernel in our case is less smooth
(Lipschitz rather than $C^{[N/2]+1}$ as an $L^2$-valued function) and 
the source condition is generally not satisfied (already for the Laplacian, see \cite[Section 4]{BVHNS}), but this is compensated in part by the  
additional structure of our $\R^N$-valued kernel (it is a gradient) which
eventually allows us to prove uniqueness of a minimizer in \eqref{eq|main}.

The paper is organized as follows. After fixing notation in Section \ref{preln},
we discuss inverse problems with source in divergence form in Section \ref{Sect2.1}, and state  our main result as well as Conditions \ref{conU}
which are sufficient for the former to apply. Section \ref{secDF} introduces
(one version of) Smirnov's decomposition of divergence-free measures, and shows the
existence of integral curves to the unit field arising from their polar decomposition. In Section \ref{sec|Uni} we prove our main result: the uniqueness of a minimizer to \eqref{eq|main}, and in Section \ref{CondEUsec} we set up
sufficient conditions for the solution
of an elliptic PDE of the form \eqref{eqelf} to generate (by restriction
of that function or its gradient to a
subset of the ambient space) an operator $A$ satisfying Conditions \ref{conU};
two scenarios are emphasized in which these conditions are satisfied.
Finally, Appendix \ref{A} contains existence, uniqueness and regularity results for (very weak) solutions to Equation \eqref{eqelf} vanishing at infinity,
while Appendix \ref{UC} adapts known unique continuation results to the present setting and Appendix \ref{Rgrad} deals with quasi-gradients, showing in particular  that they are gradients on a quasi-convex compact subset of $\R^N$.

\section{Preliminaries and notation}
\label{preln}
In    Euclidean space $\R^N$, we write $\bs{x}=(x_1,\cdots,x_N)^t$ to display the components of $\bs{x}$
(a  superscript ``$t$'' means ``transpose'') and
$|\bs{x}|$, $\langle \bs{x},\bs{y}\rangle$ for norm of $\bs{x}$ and
the scalar product of $\bs{x}$ and $\bs{y}$, respectively. 
As a convention, we will typically write elements of $\R^N$ as well as $\R^N$-valued functions and measures with bold symbols to facilitate the distinction with their scalar counterparts.
 We let $B(\bs{z},r)$ be the open ball centered at  $\bs{z}$ with radius $r$
and   $\mathbb{S}(\bs{z},r)$  its  boundary sphere.
We put  $\mathrm{dist}\,(E, F)$ for the Euclidean distance between 
sets $E,F\subset\R^N$, or simply  $\mathrm{dist}\,(\bs{z}, F)$  when
$E=\{\bs{z}\}$. The symbol $\partial E$ denotes
the boundary of $E\subset\R^N$. We write $f{|_E}$ for  the restriction of a function $f$ to a subset $E$ of its domain, and $\chi_F$ for the characteristic function of a set $F$.

A set in $\R^N$ is {\em rectifiably-connected} if any two points thereof can be joined by a rectifiable arc in the set. 
Every $E\subset\R^\N$ can be partitioned into {\em rectifiably connected components} formed by the equivalence relation $z_1\equiv z_2$ if they can be connected by a rectifiable curve in $E$. A rectifiably connected set is 
{\em quasi-convex} if there is a constant $C$ such that
the arc joining any two points $z_1$, $z_2$ within the set may be chosen of length less than  $C|z_1-z_2|$. We note that the closure of a quasi-convex set is rectifiably-connected, therefore the rectifiably connected components of $E$ are
relatively closed if only they are quasi-convex.

Given $E\subset\R^N$ and $F\subset\R^m$, let  $C(E,F)$  denote the continuous functions on $E$ valued in $F$; we simply write $C(E)$ for $C(E,\R)$. 
When $E$ is locally compact,  we  put  $C_0(E,\R^m)$ (or simply $C(E)$ if $m=1$) for
those $\bs{f}\in C(E,\R^m)$ such that, to each $\varepsilon>0$, there is a compact $K\subset E$ with $|\bs{f}(x)|<\varepsilon$ if $x\in E\setminus K$ (hence $C_0(E,\R^m)=C(E,\R^m)$ when $E$ is compact).
We topologize $C_0(E,\R^m)$  with the sup of the Euclidean norm, which makes it
a  Banach space, and we note that it is the closure in $C(E,\R^m)$
of $C_c(E,\R^m)$, the space of compactly supported continuous functions on $E$.


For $E\subset\R^N$ a Borel measurable set, we let
$\mathcal{M}(E)$ denote the finite signed Borel measures on $\R^\N$  carried by $E$. The space
$\mathcal{M}(E)^m$ identifies with  $\R^m$-valued measures on $E$, and 
for $\bs{\nu}\in  \mathcal{M}(E)^m$ its  {\em total variation measure} $|\bs{\nu}|\in\mathcal{M}(E)$  is defined  by
\begin{equation}\label{totVarMeasDef}
  |\bs{\nu}|(B):=\sup_{\mathcal{P}} \sum_{P\in \mathcal{P}} |\bs{\nu}(P)|,
  \qquad B\ \mathrm{a \ Borel \ subset\ of\ } \R^N,
\end{equation}
where the supremum is taken over all finite Borel partitions $\mathcal{P}$ of $B$.   

By a theorem of F. Riesz, if $E$ is locally compact then  the Banach space $\mathcal{M}(E)^m$  equipped  with the {\em total variation norm} $\|\bs{\nu}\|_{TV}:=|\bs{\nu}|(\R^N)$ is dual to $C_0(E,\R^m)$
under the pairing $\langle\bs{\nu}, \bs{\varphi}\rangle =\sum_{i=1}^m\int\varphi_id\nu_i$, where $\nu_i$ and $\varphi_i$ are the components of $\bs{\nu}$ and $\bs{\varphi}$ respectively.
As $|\bs{\nu}|$ is a Radon measure the Radon-Nikodym derivative $\bb{u}_{\bs{\nu}}:=d\bs{\nu}/d{|\bs{\nu}|}$ exists,  $|\bb{u}_{\bs{\nu}}|=1$ a.e.  with respect to $|\bs{\nu}|$ and $\bs{\nu}$ is absolutely continuous with respect to $|\bs{\nu}|$, so that $d\bs{\nu}=\bb{u}_{\bs{\nu}}d|\bs{\nu}|$; the latter is
the \emph{polar decomposition} of $\bs{\nu}$. 
When  $\sigma\in C(E)$, we
let $\sigma\bs{\nu}$ denote the measure such that 
$d(\sigma\bs{\nu})=\sigma\bb{u}_\bs{\nu}d|\bs{\nu}|$.
By definition the support of
$\bs{\nu}$ is the support of $|\bs{\nu}|$, denoted as $\supp\bs{\nu}$ (or  $\supp|\bs{\nu}|$).

For locally compact $E$, we  also endow $\mathcal{M}(E)^m$ 
with the weak-star topology making it a locally convex vector space,
the dual of which is  $C_0(E,\R^m)$; by the separability of the latter and the
Banach-Alaoglu theorem, the  weak-star
topology induced on norm-closed balls is  compact  metrizable \cite[Theorem 3.10 \& 3.15 \& 3.16]{Rudin}.
When $E_1\subset E_2$ are Borel  subsets of $\R^N$
and $\bs{\nu}\in \mathcal{M}(E_2)^m$, we write $\bs{\nu}\mathcal{b} E_1$ for
the restriction of $\bs{\nu}$ to $E_1$.


We
write $C^\alpha(E,\R^m)$ for the space of  \emph{Hölder-continuous}
$\R^m$-valued functions on $E$ with  exponent $\alpha\in(0,1)$.
When $E$ is compact we endow $C^\alpha(E,\R^m)$  with  the norm
\begin{equation}
  \label{defHoln}
  \|\bs{f}\|_{C^\alpha(E,\R^m)}:=\|\bs{f}\|_{C(E,\R^m)}+\sup_{\bs{x},{\bs y}\in E}\frac{|\bs{f}(\bs{x})-\bs{f}({\bs y})|}{|\bs{x}-{\bs y}|^\alpha}.
\end{equation}
A function $\bs{f}:E\to\R^N$ is \emph{locally}
  H\"older continuous if
$\bs{f}|_K$ is H\"older continuous on each compact  $K\subset E$,
and we topologize   $C^\alpha_{loc}(E,\R^m)$ with the seminorms  $\|\cdot\|_{C^\alpha(K,\R^m)}$.
Likewise, $\mathrm{Lip}(E,\R^m)$ indicates the \emph{Lipschitz} functions
$E\to\R^m$; {\it i.e.} those $\bs{f}$ such that $|\bs{f}(\bs{x})-\bs{f}(\bs{y})|\leq c|\bs{x}-\bs{y}|$, and the smallest possible  $c$ is the \emph{Lipschitz constant} of $\bs{f}$ which
defines a seminorm on  $\mathrm{Lip}(E,\R^m)$. The space $\mathrm{BLip}(E,\R^m)$
of bounded Lipschitz functions receives the norm
$\|\cdot\|_{\mathrm{BLip}(E,\R^m)}$ obtained by putting $\alpha=1$ in
\eqref{defHoln}; clearly  $\mathrm{BLip}(E,\R^m)=\mathrm{Lip}(E,\R^m)$
when $E$ is bounded. The space
$\mathrm{Lip}_{loc}(E,\R^m)$
of \emph{locally} Lipschitz functions; {\it i.e.}
whose restriction to each compact $K\subset E$  is Lipschitz receives the seminorms $\|\cdot\|_{\mathrm{BLip}(K,\R^m)}$.
Lipschitz maps can naturally
be defined between any two metric spaces, and we shall have an occasion to deal with mappings  from subsets of $\R^N$ into a Hilbert space.
The modulus of continuity of 
$\bs{f}:V\to\R^m$ is the function  $\omega_{\bs{f}}:\R^+\to\R^+$ defined by $\omega_{\bs{f}}(t):=\sup_{|\bs{x}-\bs{y}|\leq t}|\bs{f}(\bs{x})-\bs{f}(\bs{y})|$, and $\bs{f}$ is \emph{Dini-continuous} if
$\int_0^\delta(\omega_f(t)/t) dm(t)<\infty$ for some $\delta>0$.
Clearly, functions in $C^\alpha(V,\R^m)$ or $\mathrm{Lip}(V,\R^m)$ are Dini-continuous.

For $V\subset\R^N$ an open set, the space $C^1(V,\R^m)$
    is comprised of functions with continuous derivatives,
    topologized with the seminorms $\max\{\|\bs{f}\|_{C(K,\R^m)}, \|D \bs{f}\|_{C(K,\R^{m\times N})}\}$ when $K$ ranges over compact subsets of $V$, with $D \bs{f} $ to mean the Jacobian matrix  of $\bs{f}$.
   A function in $C^1(V,\R^m)$ is called \emph{Dini-smooth} if its derivatives are Dini-continuous. The spaces  $C^{1,\alpha}(V,\R^m)$ and  $C^{1,1}(V,\R^m)$ 
  consist of $C^1$-functions on $V$
  whose first order derivatives are respectively Hölder and Lipschitz-continuous, endowed with the seminorms  $\max\{\|\bs{f}\|_{C(V,\R^m)}, \|D \bs{f}\|_{C^\alpha(V,\R^{m\times N})}\}$ and 
  $\max\{\|\bs{f}\|_{C(V,\R^m)}, \|D f\|_{\mathrm{Lip}(V,\R^{m\times N})}\}$; such functions are called H\"older-smooth and Lipschitz-smooth.
   We put $C_c^\infty(V,\R^m)$ for  the space of
$C^\infty$-smooth $\R^m$-valued 
functions with compact support in $V$.

In all of the previous notation, we  omit $\R^m$ if $m=1$.

  For an arbitrary set $E\subset\R^N$ we put $C^1(E)$ 
  for the space of restrictions to $E$ of functions in $C^1(V)$
  where $V$ ranges over all open sets containing $E$.
  We denote by $\mathcal{G}(E)$ the space of continuous gradients on $E$; i.e., the restrictions $\nabla\varphi{|_E}$ where $\varphi\in C^{1}(E)$ and $\nabla\varphi$ denotes its gradient. The closure in $C(E,\R^N)$ of $\mathcal{G}(E)$
  is the space $\mathcal{Q}(E)$ of \emph{quasi-gradients} on $E$; see \cite[Section 3.1.1]{SmirnovKhavin1998} and Appendix \ref{Rgrad}.

An open subset of $\R^N$ is called Lipschitz (resp.
  Hölder-smooth) if its boundary is compact and
   locally isometric to the graph of a Lipschitz
   (resp. 
   Hölder-smooth) function,
   defined on an open subset of $\R^{N-1}$.
\section{Inverse source problems in divergence form and main result} \label{Sect2.1}
Our primary example of an  inverse  source problem with source in divergence form is the inverse  magnetization problem (e.g., see \cite{Parker}, \cite{IMP}, \cite{WLF}, \cite{IMP2}, \cite{Multipole}).   It consists in recovering  a {\em magnetization distribution} $\bs{\mu}\in\mathcal{M}(S)^3$, supported on a compact set $S\subset \R^3$, from measurements   of the     {\em magnetic field} $\bb{b}(\bs{\mu})$ (or just a component thereof) taken on a set $Q\subset\R^3$ disjoint from $S$.  We model the measurement space as a separable Hilbert space $\mathcal{H}$; e.g.  $\mathcal{H}=L^2(Q,\R^3)$ or $ L^2(Q)$ are natural choices. Here, the
magnetic field is minus the gradient of the  {\em scalar magnetic potential}  $\Phi(\bs{\mu})$  generated by  $\bs{\mu}$
(see \cite{Jac1975}):
\begin{equation}\label{PhiDef}
\Phi(\bs{\mu})(\bs{x}):=\int \mathbf{G}(\bs{x},{\bs y})\cdot d\bs{\mu}({\bs y}) , \qquad  \mathrm{a.e.} \ \bs{x}\in\R^3\ \mathrm{and}\ \mathrm{every}\  \bs{x}\notin \text{supp } \bs{\mu},
\end{equation} where $$\mathbf{G}(\bs{x},{\bs y}):=\frac{1}{4\pi}  \grad_{\bs y}\frac{1}{|\bs{x}-{\bs y}|}=\frac{1}{4\pi}\frac{\bs{x}-{\bs y}}{|\bs{x}-{\bs y}|^3}$$
and $\nabla_{\bs y}$ indicates that the gradient is computed with respect to ${\bs y}$. That \eqref{PhiDef} is well-defined for a.e. $\bs{x}$ follows from Fubini's theorem, see Appendix~\ref{A} Equation \eqref{majpF} for a more general estimate.
In particular, when $\bs{x}\not\in \text{supp } \bs{\mu}$, one has
\begin{equation}
\label{eq|b(mu)}
\begin{split}
  \bb{b}(\bs{\mu})(\bs{x})&:=-\grad \Phi(\bs{\mu})(\bs{x})=-\frac{1}{4\pi}
                             \grad_{\bs{x}} \int \grad_{\bs y}\frac{1}{|\bs{x}-{\bs y}|}\cdot d\bs{\mu}({\bs y})\\&= \int D_{\bs y} \mathbf{G}(\bs{x},{\bs y})   d\bs{\mu}({\bs y})= -\langle \mathbf{G}(\bs{x},\cdot),\div \bs{\mu}\rangle,\\
\end{split}
\end{equation}
where  $$ D_{\bs y}\mathbf{G}(\bs{x},{\bs y})=\frac{1}{4\pi}\frac{1}{|\bs{x}-{\bs y}|^5}\left(|\bs{x}-{\bs y}|^2I_{3\times 3}-3(\bs{x}-{\bs y})(\bs{x}-{\bs y})^T\right)$$ is the $3\times 3$-matrix valued differential of $\mathbf{G}$ with respect to ${\bs y}$, and the bracket in \eqref{eq|b(mu)} is the duality product between
distributions with compact support and smooth functions on the support; see e.g.\cite{Schwartz}.	
We remark that Equation \eqref{PhiDef} defines a locally integrable function
$\Phi(\bs{\mu})$ on $\R^3$ that satisfies (in the distributional sense) the following Poisson equation where $\nabla\cdot$ stands for the  Euclidean  divergence
and $\Delta$ for the Euclidean
Laplacian \cite{IMP}:
\begin{equation}\label{eq|poisson}
	\Delta\Phi=\div\bs{\mu},\qquad \lim_{|x|\to\infty}\Phi(x)=0.
\end{equation}
Here the \emph{forward operator} $\Af:\mathcal{M}(S)^3\to \mathcal{H}$ maps $\bs{\mu}$ to the restriction on $Q$ of $\nabla\Phi$, or ${\bs v}\cdot \nabla\Phi$
if we measure only in a fixed direction ${\bs v}$,
and  $ \mathcal{H}=L^2(Q,\R^3)$ (if $A(\bs{\mu})=\bb{b}(\bs{\mu})_{|Q}$)
or  $ \mathcal{H}=L^2(Q)$ (if $A(\bs{\mu})=\bb{v}\cdot \bb{b}(\bs{\mu})_{|Q}$).
We call a set \emph{slender} if it has Lebesgue measure zero and every component of its complement has infinite measure; for a compact set,  slenderness  is equivalent to having zero measure and  connected complement (see \cite{BVHNS}).
When a compact set $S$ is slender and $Q$ is big enough (for instance if it contains a patch of real analytic surface), one can show that the null space of $\Af$ consists exactly of measures $\bs{\mu}=(\mu_1,\mu_2,\mu_3)^t\in \mathcal{M}(S)^3$ that are \emph{divergence-free}; {\it i.e.}, such that
  $\sum_{j=1}^3 \partial_{x_j}\mu_j=0$ as a distribution on $\R^3$ \cite{BVHNS}.

  More generally, we shall consider elliptic equations  with source term in divergence form on $\R^N$, $N\geq3$, given by
  \begin{equation}\label{genElliptic}
-\mathcal{L}(\Phi):=	\div(\bs{\Sigma}\nabla\Phi)=\div \bs{\mu}, \qquad \lim_{|x|\to\infty}\Phi(x)=0, 
\end{equation}
where
$\bs{\Sigma}:\R^N\to\R^{N\times N}$ is  a symmetric, uniformly elliptic  matrix-valued map satisfying certain assumptions to be specified later; see Conditions \ref{conE} in Section \ref{CondEUsec}.
The solution
$\Phi:\R^N\to \R$ is  called  the {\em potential},
and the minus sign in the definition of $\mathcal{L}$ is a convention to make
it a \emph{positive} operator. 
Such generality  naturally occurs in a variety of electro-magnetic problems  concerning inhomogeneous and non-isotropic media (e.g.,  in EEG or ECG)\footnote{  From the viewpoint of Geometric Analysis, the inverse source problem
  associated with 
\eqref{genElliptic} amounts to recover a compactly supported
$\R^N$-valued measure  $\bs{\mu}$ from the gradient summand $\nabla\Phi$ of its Helmholtz-Hodge
decomposition on $\R^N$, when the latter is
endowed with the Riemannian metric $(\mathrm{det}\bs{\Sigma})\bs{\Sigma}^{-1}$.
}.
We shall assume that data are of the form $L(\Phi)$ for some  linear {\em measurement operator} $L$, acting on an appropriate space of functions $\mathcal{Z}$ containing the solutions of \eqref{genElliptic}  and  taking values in a separable Hilbert space $\mathcal{H}$.   For example  $L$ may be the gradient or a directional derivative of $\Phi$ evaluated on $Q$,
and then $\mathcal{H}$ could be $L^2(Q)$ (or $L^2(Q,\R^m)$) as in the prior magnetization case. 

The core properties  
that the forward operator
$\Af(\bs{\mu}):=L(\Phi(\bs{\mu}))$
should  retain from the Poisson case discussed after \eqref{eq|poisson},
in order to carry out  our analysis,
are the following:
\begin{con}{U}\label{conU}
	\color{white}{a}\color{black}
\begin{enumerate}
\item $S\subset\R^N$ is compact and the forward operator $\Af:\mathcal{M}(S)^N\to\mathcal{H}$   is sequentially weak-star to strong continuous for some separable Hilbert space $\mathcal{H}$.
\item The null space of $\Af$ is comprised of those measures $\bs{\mu}\in \mathcal{M}(S)^N$ such that $\mathrm{div}\bs{\mu}=0$ as  a distribution on $\R^N$.
    \item The adjoint operator $\Af^*:\mathcal{H}\to C(S)^N$
      is actually valued in the space
      $\mathrm{Lip}(S,\R^N)$, and the restriction of  $\mathrm{Im}A^*$
      to any rectifiably connected component
      $E$ of $S$ is included in $\mathcal{G}(E)$, the space of continuous gradients on $E$.
\end{enumerate}
\end{con}

\begin{rmk}
  \label{onrange}
Since (vector-valued) functions in the image of $A^*$   integrate to zero against (vector-valued) measures in the null space of $A$, assumption (2) and the Hahn-Banach theorem imply that $\mathrm{Im}A^*\subset\mathcal{Q}(S)$,
the space of quasi-gradients on $S$ defined in Section \ref{preln}; {\it cf.}
Appendix \ref{Rgrad}. Note that assumption (3) requires more,
namely $\mathrm{Im}A^*$ should be included in $\mathcal{G}(S)\cap\mathrm{Lip}(S,\R^N)$, where $\mathcal{G}(S)$ denotes 
the continuous gradients on $S$ (also defined in Section \ref{preln}). However, when all rectifiably connected components of $S$  are quasi-convex, Lemma \ref{lemma|constructg} from Appendix \ref{Rgrad} implies that $\mathrm{Im}A^*\subset \mathcal{G}(S)$. Hence, under this mild  hypothesis on $S$,  the content of
(3) reduces to the assumption  that
$\mathrm{Im}A^*\subset \mathrm{Lip}(S,\R^N)$.
  \end{rmk}
We are now in position to state our main result, Theorem~\ref{thm|uniqueness}
below, on the  uniqueness of minimizers for the TV-regularized functional
$\mathcal{F}$ defined in \eqref{eq|main}. As stressed in the introduction,
this result teams up with \cite[Theorem 4.4]{BVHNS} to
allow for the construction of weak-star convergent
sparse recovery schemes in this infinite-dimensional context.

\begin{theorem} 
	\label{thm|uniqueness}
	Let $S\subset\R^\N$ be compact, $\HH$ a Hilbert space, and $\Af:\mathcal{M}(S)^\N\to\HH$ a linear 
        operator such that Conditions~\ref{conU} are satisfied.   
  	
	Then, for every $f\in\HH$, $\sigma\in \mathrm{Lip}(S)$ with $\sigma$ positive and $\lambda>0$, the functional $\mathcal{F}$ given in \eqref{eq|main} has a unique minimizer with respect to $\bs{\mu}\in\mathcal{M}(S)^\N$.
\end{theorem}

Conditions \ref{conU} are abstract in that
they do not refer to a specific forward operator anymore.
However, applications to inverse source problems that we have in mind deal with 
elliptic PDEs,
so we must still specify assumptions on
Equation \eqref{genElliptic} and on the geometry of $S$,
$Q$ guaranteeing that the forward operator
  $\Af(\bs{\mu}):=L(\Phi(\bs{\mu}))$ will meet Conditions \ref{conU}
  when $\Phi(\bs{\mu})$ is the solution of \eqref{genElliptic}, and $L$  the
  measurement operator restricting $\Phi$ (or $\nabla\Phi$) to $Q$.
  We shall address these issues in Section~\ref{CondEUsec}, once we have proven
  Theorem \ref{thm|uniqueness}.


\section{Divergence-free measures and their phase portraits}
\label{secDF}

\subsection{Rectifiable curves and divergence-free measures}

Recall that a {\it curve in $\R^N$} is a continuous $\R^N$ valued function $\bs{f}$ on an interval $[a,b]\subset \R$, where we further assume that $\bs{f}$ is non-constant on any open interval of $[a,b]$.   If $\psi$ is a continuous, strictly increasing function from $[c,d]$ onto  $[a,b]$ , then we refer to $\bs{f}\circ \psi$ as a {\it reparametrization} of the curve $\bs{f}$ and we regard $\bs{f}$ and $\bs{f}\circ \psi$ to be equivalent.     A curve $\bs{f}$ is {\it rectifiable} if it has  bounded variation; i.e., if there is an upper bound on sums of the form $\sum_{i=0}^m |\bs{f}(x_{i+1})-\bs{f}(x_i)|$ for any partition $a=x_0<x_1<\cdots <x_m=b$, in which case the least upper bound $\ell(\bs{f})$ of these sums is called the  {\it length} of curve.  It is elementary that a curve is rectifiable if and only if it admits a Lipschitz parametrization (take $\psi=:\phi^{-1}$ where $\phi:[a,b]\to[0,\ell(\bs{f})]$ is given
by $\phi(t)=\ell(\bs{f}_{|[a,t]})$),  so we may identify  elements of $\mathrm{Lip}([a,b],\R^N)$ with  {\it rectifiable curves in $\R^N$}.
For  $\bs{\gamma}\in \mathrm{Lip}([a,b],\R^N)$
the derivative  $\bs{\gamma}'$
exists a.e. on $(a,b)$, moreover we let $\#(\bs{\gamma},x)$ be the 
cardinality (finite or infinite) of the preimage $\bs{\gamma}^{-1}(x)$ for $x\in\Gamma:=\bs{\gamma}([a,b])$.  
The length $\ell(\bs{\gamma})$ may then alternatively be expressed as 
\begin{equation}
	\label{eq|areas|old}
	\ell(\bs{\gamma}):=\int_a^b |\bs{\gamma}^\prime(t)|\,dt=\int_\Gamma \#(\bs{\gamma},x) \,dH^1(x),
\end{equation} 
with  $H^1$ to mean
1-dimensional Hausdorff measure normalized so as to coincide
with standard arc length, and where the second equality follows from the  area formula 
\cite[Thm 3.2.3, part 1]{FedererBook}. 
Letting as before
$\psi=:\phi^{-1}$ where $\phi:[a,b]\to[0,\ell(\bs{f})]$ is given
by $\phi(t)=\ell(\bs{f}_{|[a,t]})=\int_a^t|\bs{\gamma}^\prime(\tau)|d\tau$,
one can check from the chain rule 
and a Sard-like theorem for Lipschitz functions (see \cite[Theorem~7.6]{Mattila})
that $\widetilde{\bs{\gamma}}:=\bs{\gamma}\circ\psi$  is  in fact a unit speed
reparametrization; that is,
$\widetilde{\bs{\gamma}}$ is Lipschitz and $|\widetilde{\bs{\gamma}}^\prime(t)|=1$ for a. e. $t\in[0,\ell(\bs{\gamma})]$.

The set of  maps valued in $S\subset\R^N$ with
Lipschitz constant less than  or equal to $\lambda$ will be  denoted by
$\mathrm{Lip}_\lambda([a,b],S)$.
When $S$ is compact and $\mathrm{Lip}_\lambda([a,b],S)$ is equipped with the uniform metric $d(\bs{f},\bs{g})=\|\bs{f}-\bs{g}\|_{C([a,b],S)}$ (of course $\mathrm{Lip}_\lambda([a,b],S)$ is not a vector space),  closed balls are compact
 by the Arzelà-Ascoli theorem \cite[Theorem 4.7]{dunford_schwartz_1958}.


We define a map $\bs{\pi}:\textup{Lip}([a,b],\R^\N)\longrightarrow\mathcal{M}(\R^\N)^\N$ by 
\begin{equation}\label{eq|Rgamma|old}
	\langle \bs{\pi}(\bs{\gamma}),\bs{G}\rangle 
	:=\int_a^b \bs{G}(\bs{\gamma}(t))\cdot\bs{\gamma}'(t) dt=\int_{\Gamma}\left( \sum_{t\in\bs{\gamma}^{-1}(x)}\bs{G}(x)\cdot\bs{\gamma}^\prime(t) \right)dH^1(x),\quad
	\bs{G}\in C_c(\R^N,\R^N),
\end{equation}
where  the second equality  follows from the co-area formula \cite[Thm 3.2.3, part 2]{FedererBook}. Clearly, $\bs{\pi}(\bs{\gamma})$  is invariant under a reparametrization of $\bs{\gamma}$, so we may assume if helpful that $\bs{\gamma}$  in \eqref{eq|Rgamma|old} has unit speed.
We note the obvious inequality
$\|\bs{\pi}({\bs\gamma})\|_{TV}\leq\ell(\bs{\gamma})$. Identifying vector fields on $\R^n$ with differential 1-forms, $\bs{\pi}(\bs{\gamma})$ may
as well be viewed as a 1-dimensional integral current, but we shall stick to the vector-valued measure interpretation taken above.

 We  shall need a continuity property of the map $\bs{\pi}$ just defined.

\begin{lemma}
	\label{lemma|map_pi}
For $\lambda>0$,	the restriction  $\bs{\pi}:\textnormal{Lip}_\lambda([a,b],\R^N)\to\mathcal{M}(\R^\N)^\N$ is continuous when $\mathcal{M}(\R^\N)^\N$ receives the weak-star topology and $\mathrm{Lip}_\lambda([a,b],\R^N)$ the topology of uniform convergence.
\end{lemma}
\begin{proof}
	  Let $\{\bs{\gamma}_n\}_{n\in\mathbb{N}}\subset\textnormal{Lip}_\lambda([a,b],\R^N)$ converge uniformly to $\bs{\gamma}^*\in\textnormal{Lip}_\lambda([a,b],\R^N)$.
	Define a map $\bs{p}:\mathrm{Lip}_\lambda([a,b],\R^N)\to\mathcal{M}([a,b])^\N$
         by $d\bs{p}(\bs{\gamma})(t):=\bs{\gamma}'(t)dt$, with $dt$ the differential of 1-dimensional Lebesgue measure.
         As Lipschitz functions are absolutely continuous, for each open interval $(s,t)\subset[a,b]$ we have   $\bs{p}(\bs{\gamma})((s,t))=\bs{\gamma}(t)-\bs{\gamma}(s)$ and similarly for closed intervals. Thus, $\bs{p}(\bs{\gamma}_n)(O)\to \bs{p}(\bs{\gamma}^*)(O)$ for each open subset $O$ of $[a,b]$.  Hence, $\bs{p}(\bs{\gamma}_n)\to \bs{p}(\bs{\gamma}^*)$ weak-star as $n\to\infty$, because the $\bs{p}(\bs{\gamma}_n)$ have uniformly bounded total variation so that $\{\bs{p}(\bs{\gamma}_n)\}_{n\in\mathbb{N}}$ is relatively weak-star compact in
        $\mathcal{M}(\R^\N)^\N$ by the Banach-Alaoglu theorem, while 
        $\bs{p}(\bs{\gamma}^*)$ is the only possible limit point.

	Next, pick $\bs{f}\in C_c(\R^N,\R^N)$ and $\varepsilon>0$.  Let $M,\delta>0$ and $m\in\mathbb{N}$ be such that, for $n>m$, one has:
	\begin{itemize}
		\item[(i)] $|\ \langle\ \bs{p}(\bs{\gamma}_n)-\bs{p}(\bs{\gamma}^*),\ \bs{f}\circ\bs{\gamma}^*\ \rangle\ | < \frac{\varepsilon}{2}$,
		\item[(ii)] $\|\bs{\gamma}^*-\bs{\gamma}_n\|_\infty<\delta$,
		\item[(iii)] for any $\bs{x},\bs{y}\in \R^N$ with $|\bs{x}-\bs{y}|<\delta$,
                  it holds that $|\bs{f}(\bs{x})-\bs{f}(\bs{y})|<\frac{\varepsilon}{2\lambda(b-a)}$.
	\end{itemize}
	Then, as $|\bs{\gamma}_n^\prime|\leq\lambda$ because $\bs{\gamma}\in\textnormal{Lip}_\lambda([a,b])^\N$, it follows for $n>m$ that
	\begin{align*}
		|\ \langle\ {\bs{\pi}}({\bs{\gamma}_n})-{\bs{\pi}}({\bs{\gamma}^*}),&\ \bs{f}\ \rangle\ |
		=\left|\ \int_{a}^{b}\bs{f}(\bs{\gamma}_n(t))\cdot\bs{\gamma}_n'(t)\ dt - \int_{a}^{b}\bs{f}(\bs{\gamma}^*(t))\cdot(\bs{\gamma}^*)'(t)\ dt\ \right| \\
		&\leq \left|\ \int_{a}^{b}\bs{f}(\bs{\gamma}_n(t))\cdot\bs{\gamma}_n'(t)\ dt - \int_{a}^{b}\bs{f}(\bs{\gamma}^*(t))\cdot\bs{\gamma}_n'(t)\ dt\ \right| \\
		&\ \ \ \ \ \ \ \ \quad\, + \left|\ \int_{a}^{b}\bs{f}(\bs{\gamma}^*(t))\cdot\bs{\gamma}_n'(t)\ dt - \int_{a}^{b}\bs{f}(\bs{\gamma}^*(t))\cdot(\bs{\gamma}^*)'(t)\ dt\ \right| \\
		&\leq \lambda\int_{a}^{b}|\bs{f}(\bs{\gamma}_n(t))-\bs{f}(\bs{\gamma}^*(t))|\ dt 
		+ |\ \langle\ \bs{p}(\bs{\gamma}_n)-\bs{p}(\bs{\gamma}^*),\ \bs{f}\circ\bs{\gamma}^*\ \rangle\ |
		<\varepsilon.
	\end{align*}
	By definition of the weak star neighborhoods, this yields the result.
      \end{proof}

For each $\ell> 0$,  we consider  the collection $\mathfrak{C}_\ell$
of those $\bs{\pi}(\bs{\gamma})$ associated with
a rectifiable curve $\bs\gamma$ of length $\ell$, and we let
$\mathfrak{U}_\ell:=\cup_{0< l \leq\ell} \mathfrak{C}_l$. These sets contain the zero measure
(think of a segment of length $\ell/2$ percursed once in each direction so that
the integral cancels out  in \eqref{eq|Rgamma|old}), and they are contained in the closed ball centered at $\bs{0}$ of radius $\ell$ in $\mathcal{M}(\R^N)^N$
which is a compact metric space when endowed  with the weak-star topology.
\textit{We next contend} that $\mathfrak{U}_\ell\cap \mathcal{M}(\overline{B(\bs{0},R)})^N$ is   closed  for all $\ell$ and $R$ positive.     Indeed,
take a  sequence of nonzero $\bs{\pi}(\bs{\gamma}_n)$ in this set converging weak-star to some $\bs{\mu}$. Clearly  $\supp\bs{\mu}\subset \overline{B(\bs{0},R)}$, but the point sets $\Gamma_n:=\bs{\gamma}_n([0,\ell_n])$ need not be contained  in $\overline{B(\bs{0},R)}$ even though $\bs{\pi}(\bs{\gamma}_n)$ is supported there, because self-cancellation may occur in \eqref{eq|Rgamma|old}; however,
the $\Gamma_n$ must intersect $\overline{B(\bs{0},R)}$ since
$\bs{\pi}(\bs{\gamma}_n)\neq0$ by assumption, and their diameter cannot exceed $\ell$ whence $\Gamma_n\subset \overline{B(\bs{0},R+\ell)}$. Thus, 
we may extract a subsequence, denoted again by $\bs{\pi}(\bs{\gamma}_n)$, where  each $\bs{\gamma}_n$ is a unit speed curve  valued in  $\overline{B(\bs{0},R+\ell)}$,
whose lengths $l_n$ converge to some $l_0\leq\ell$ as $n\to\infty$.  If $l_0=0$, then $\bs{\pi}(\bs{\gamma}_n)$ converges to the zero measure (even strongly). 
 If $l_0>0$,  then reparametrizing each $\bs{\gamma}_n$ to the interval $[0,l_0]$ by a simple scaling leads to a sequence of curves in  $\mathrm{Lip}_\lambda([0,l_0],\overline{B(\bs{0},R+\ell)})$ for some $\lambda$ independent of
$n$.  We may then extract a uniformly convergent subsequence of these curves, by the
Arzelà-Ascoli theorem, with limit $\bs{\gamma}$ such that $\bs{\mu}=\bs{\pi}(\bs{\gamma})$.  Since  $\ell(\bs{\gamma})$ cannot exceed $l_0$ we have that $\bs{\mu}\in \mathfrak{U}_\ell\cap \mathcal{M}(\overline{B(\bs{0},R)})^N$. \textit{This proves our contention.} 
By the latter,   $\mathfrak{U}_\ell$ is a countable union of closed sets $  \mathfrak{U}_\ell\cap \mathcal{M}(\overline{B(\bs{0},n)})^N$, $n\in \mathbb{N}$,
while  $\mathfrak{C}_\ell=\mathfrak{U}_\ell\setminus(\cup_{n>1/\ell}\mathfrak{U}_{\ell-1/n})$ is a Borel subset of $\mathcal{M}(\R^N)^N$.

For $\varphi\in C^1_c(\R^N)$, definition \eqref{eq|Rgamma|old}
  yields
	$\langle \bs{\pi}(\bs{\gamma}),{\nabla\varphi}\rangle =\varphi(b)-\varphi(a)$ for $\bs{\gamma}\in \mathrm{Lip}([a,b],\R^N)$.
	Thus, the distributional divergence of $\bs{\pi}(\bs{\gamma})$
	is $\div\bs{\pi}(\bs{\gamma})=\delta_{\bs{\gamma}(b)}-\delta_{\bs{\gamma}(a)}$ (the difference of a Dirac delta at $\bs{\gamma}(b)$ with one at $\bs{\gamma}(a)$), and  $\bs{\pi}(\bs{\gamma})$ is divergence free if and only if $\bs{\gamma}$ is closed; {\it i.e.} $\bs{\gamma}(a)=\bs{\gamma}(b)$. Moreover,
	the beginning and  ending points of $\bs{\gamma}$, namely $\bs{\gamma}(a)$ and  $\bs{\gamma}(b)$,
	are uniquely determined by  $\bs{\pi}(\bs{\gamma})$
	if and only if the latter is not   divergence free; in this case, we refer to $\bs{\gamma}(a)$ and $\bs{\gamma}(b)$
	as the beginning point and ending point of $\bs{\pi}(\bs{\gamma})$,
        respectively. Note that when $\bs{\pi}(\bs{\gamma})$ is viewed as a 1-dimensional integral current, $\delta_{\bs{\gamma}(b)}-\delta_{\bs{\gamma}(a)}$ is simply its boundary current; \cite[Section 4.1.7]{FedererBook}.

For each $\ell>0$,  we consider  the collection $\mathcal{C}_\ell$
of those $\bs{\pi}(\bs{\gamma})$ associated with
a parametrized rectifiable curve $\bs\gamma$ of length $\ell$ and such that
$\|\bs{\pi}({\bs\gamma})\|_{TV}=\ell$.
Note that  $\mathcal{C}_\ell$  is a Borel subset of $\mathcal{M}(\R^N)^N$ endowed with the weak topology,
because the set $\mathcal{T}_\lambda$ of measures with total variation at most $\lambda$ is closed by the Banach-Alaoglu theorem, while $\mathcal{C}_\ell=\mathfrak{C}_\ell\setminus
(\cup_{n>1/\ell}\mathcal{T}_{\ell-1/n})$.  
Since $\bs{\pi}(\bs{\gamma})$ is absolutely continuous with respect to the
restriction $H^1\mathcal{b}\Gamma$ which is finite, we may consider the Radon-Nikodym derivative $\bs{\tau}$ of $\bs{\pi}(\bs{\gamma})$ with respect to $H^1$ and  remark in view of \eqref{eq|Rgamma|old}
that $\bs{\tau} (x) = \sum_{t\in\bs{\gamma}^{-1}(x)}\bs{\gamma}' (t)$ for $H^1$-a.e. $x\in\Gamma$. Moreover, one can see that
  $\bs{\gamma}:[a,b]\to\R^N$ is such that $\bs{\pi}(\bs{\gamma})\in\mathcal{C}_\ell$ if and only if, for any $c\in[a,b]$, the restrictions $\bs{\gamma}_{|[a,c]}$ and $\bs{\gamma}_{|[c,b]}$
  are such that  $\bs{\pi}(\bs{\gamma}_{|[a,c]})$ and $\bs{\pi}(\bs{\gamma}_{|[c,b]})$ lie in $\mathcal{C}_{\ell(\bs{\gamma}_{|[a,c]})}$ and $\mathcal{C}_{\ell(\bs{\gamma}_{|[c,b]})}$ respectively.
Thus, by \eqref{eq|areas|old}, $\bs{\pi}(\bs{\gamma})\in\mathcal{C}_{\ell(\bs{\gamma})}$ if and only if $|\bs{\tau} (x)| = \#(\bs{\gamma},x)$ for $H^1$-a.e. $x\in \Gamma$, and when $\bs{\gamma}$ is a unit speed parametrization this amounts to
say that $\bs{\gamma}' (t)$ depends only on $\bs{\gamma}(t)$ for a.e. $t$, in which case
the Radon-Nikodym derivative of $\bs{\pi}(\bs{\gamma})$ with respect to 
 $|\bs{\pi}(\bs{\gamma})|$ is given by
 \begin{equation}
   \label{RNPi}
   \bb{u}_{\bs{\pi}(\bs{\gamma})}(\bs{x})=\bs{\gamma}' (t), \qquad H^1-\mathrm{a.e.\ } {\bs x\ }\in\Gamma\mathrm{\ and\ every\ }t\in\bs{\gamma}^{-1}(\bs{x}),\quad\bs{\gamma}\in\mathcal{C}_{\ell(\bs{\gamma})}.
 \end{equation}
 We put
   $\mathcal{U}_\ell\subset \mathrm{Lip}_1([0,\ell],\R^N)$ for the set of
   unit speed  parametrizations $\bs{\gamma}$ satisfying  \eqref{RNPi}, and we just showed that $\bs{\pi}(\bs{\gamma})\in\mathcal{C}_\ell$ 
   with unit speed $\bs{\gamma}$ of length $\ell$ 
   if and only if  $\bs{\gamma}\in\mathcal{U}_\ell$; in this case
   the support of  $\bs{\gamma}$
   is precisely $\Gamma=\bs{\gamma}([0,\ell])$ (no cancellation),  and it
   inherits an oriented tangent
at $H^1$ a.e. point, given by \eqref{RNPi}.
To recap the previous discussion:
\begin{lemma}
    \label{lemcurCl}
    Given $\ell>0$, the map $\bs{\gamma}\mapsto \bs{\pi}(\bs{\gamma})$ is
    a surjection from $\mathcal{U}_\ell$ onto $\mathcal{C}_\ell$; in this case
   $\supp\bs{\pi}(\bs{\gamma})=\bs{\gamma}([0,\ell])$ which has a well-defined oriented tangent  at $H^1$-a.e. point,
    the tangent at $\bs{x}$ being given by $\bs{\gamma}^\prime(t)$ for any
    $t\in[0,\ell]$ such that     $\bs{\gamma}(t)=\bs{x}$.
    It holds that $\div\bs{\pi}(\bs{\gamma})=\delta_{\bs{\gamma}(\ell)}-\delta_{\bs{\gamma}(0)}$, and 
    $\bs{\gamma}$ is a closed curve if and only if $\bs{\pi}(\bs{\gamma})$ is divergence-free.
  \end{lemma}
\begin{rmk}
The map $\bs{\gamma}\mapsto \bs{\pi}(\bs{\gamma})$ is
\emph{not} injective on $\mathcal{U}_\ell$. For instance if $C_1$ (resp. $C_2$) is the circle of radius $1$ centered at $(1,0,\cdots,0)^t$ (resp. $(3,0,\cdots,0)^t)$ in the plane $\{x_j=0,j\geq3\}$, then the unit speed curve
starting at $C_1\cap C_2=\{(2,0,\cdots,0)^t\}$ and winding clockwise
twice around $C_1$  before switching
to $C_2$ and winding twice counterclockwise 
around the latter   has the same image under $\bs{\pi}$ as the unit speed curve
also starting at $(2,0,\cdots,0)^t$ but winding only  once clockwise
around $C_1$ and then once counterclockwise around $C_2$,
before winding once again  clockwise   around $C_1$ and finally once 
counterclockwise around  $C_2$.
  \end{rmk}
  
  Let $\mathcal{C^J_\ell}$ denote the divergence-free elements of $\mathcal{C}_\ell$  and define the functions $\mathfrak{b}$ and $\mathfrak{e}$ from
           $\mathcal{C}_\ell\setminus \mathcal{C^J_\ell} $ to  $\R^\N$
        by making $\mathfrak{b}(\bs{\mu})$ and $\mathfrak{e}(\bs{\mu})$  the beginning point and ending point, respectively, of $\bs{\mu}\in \mathcal{C}_\ell\setminus \mathcal{C^J_\ell}$.
	It follows from what precedes that $\mathfrak{b}$, $\mathfrak{e}$ are well-defined, and that
   \[\mathfrak{b}(\bs{\pi}(\bs{\gamma}))=\bs{\gamma}(a),\quad
		\mathfrak{e}(\bs{\pi}(\bs{\gamma}))=\bs{\gamma}(b),\qquad
	 \text{for }\  \bs{\pi}(\bs{\gamma})\in\mathcal{C}_\ell\setminus\mathcal{C_\ell^J}	
	 \text{ and } \bs{\gamma}\in\textup{Lip}\left([a,b],\R^N\right).\]
      
        \begin{lemma}
          \label{ouvertferme}
          The set $\mathcal{C_\ell^J}$ 
          is closed in $\mathcal{C_\ell}$ for the weak-star topology,
	and the maps $\mathfrak{b}$, $\mathfrak{e}$ are continuous and open
        from $\mathcal{C}_\ell\setminus\mathcal{C_\ell^J}$ to $\R^N$  when
        $\mathcal{C}_\ell\setminus\mathcal{C_\ell^J}$ has the induced weak-star topology.
        \end{lemma}
\begin{proof}
  Because $\mathcal{C}_\ell$ is contained in the closed ball of radius $\ell$
  of $\mathcal{M}(\R^N)^N$ we may restrict to sequential continuity, and
  weak-star convergence of a sequence of measures clearly implies distributional convergence of their divergence. Thus, since $\mathcal{C_\ell^J}$ is the kernel of the divergence operator on $\mathcal{C}_\ell$,  it must be  closed in the latter for the weak-star topology.
            Moreover, distributional convergence of the divergence
  on $\mathcal{C}_\ell\setminus\mathcal{C_\ell^J}$  implies
  convergence of its support (namely the beginning and ending points), say, for the Hausdorff distance in $\R^N$; therefore, $\mathfrak{b}$ and $\mathfrak{e}$ are weak-star  continuous
        $\mathcal{C}_\ell\setminus\mathcal{C_\ell^J}\to\R^N$.
        Let now $\bs{\mu}\in \mathcal{C}_\ell\setminus\mathcal{C_\ell^J}$
    and put $\bs{x}_0:=\mathfrak{e}(\bs{\mu})$. By Lemma \ref{lemcurCl} we can write $\bs{\mu}=\bs{\pi}(\bs{\gamma})$ with $\bs{\gamma}\in\mathcal{U}_\ell$,
    and translating $\bs{\gamma}$ in any direction $\bs{u}\in\mathbb{S}(0,1)$ by an amount $\rho<\varepsilon$ we define a collection of 
    oriented curves $\bs{\gamma}_{\rho,\bs{u}}\in\mathcal{U}_\ell$ such that 
    $\mathfrak{e}(\bs{\gamma}_{\rho,\bs{u}})$ ranges over $B(\bs{x}_0,\varepsilon)$ as $\bs{u}$ ranges over $\mathbb{S}(0,1)$ and $\rho$ over $[0,\varepsilon)$.
    By the continuity of $\bs{\pi}$ restricted to $\textup{Lip}_1([0,\ell],\R^N)$ asserted in Lemma \ref{lemma|map_pi},
    $\bs{\pi}(\bs{\gamma}_{\rho,\bs{u}})$ belongs to any given weak-star neighborhood  of $\bs{\mu}$ for $\varepsilon$ small enough. Thus, $\mathfrak{e}:\mathcal{C}_\ell\setminus\mathcal{C_\ell^J}\to\R^N$ is an open map and a similar argument works for
    $\mathfrak{b}$.
  \end{proof}
  
Now, suppose that $\bs{\mu}\in\mathcal{M}(\R^\N)^\N$ is a solenoid;
{\it i.e.},   $\div\bs\mu=0$ 
as a distribution.
Then, it follows from  \cite[Theorem A]{Smi94} that for \emph{any} fixed $\ell>0$, the measure
$\bs{\mu}$ can be decomposed as a superposition of elements from $\mathcal{C}_\ell$.
More precisely, there is a positive  finite Borel
measure $\Xi$ on $\mathcal{C}_\ell$  such that $\Xi$-a.e. $\bs{\nu}$ is
supported in $\textrm{supp}\,\bs{\mu}$ and
\begin{equation}\label{Smiw}
\langle\bs{\mu},\bs{g}\rangle=\int_{\mathcal{C}_\ell} \langle
\bs{\nu},\bs{g}\rangle d\Xi(\bs{\nu}),
\qquad
\langle|\bs{\mu}|,\varphi\rangle=\int_{\mathcal{C}_\ell} \langle
|\bs{\nu}|,\varphi\rangle d\Xi(\bs{\nu}),
\end{equation}
for all $\bs{g}\in  C_c^\infty(\R^N,\R^N)$ and $\varphi\in  C_c^\infty(\R^N)$.
Actually, in  \cite{Smi94}, the set $\mathcal{C}_\ell$ is not introduced
and the integrals corresponding to
\eqref{Smiw}  are  taken over $\mathfrak{C}_\ell$. However  it follows
from \cite[Equation (1.10)]{Smi94}  that
$\Xi$ does not charge $\mathfrak{C}_\ell\setminus\mathcal{C}_\ell$, and therefore we may restrict these integrals to $\mathcal{C}_\ell$ without loss of generality.
By mollification the identities in
\eqref{Smiw}  more generally hold for 
$\bs{g}\in  C_c(\R^N,\R^N)$ and $\varphi\in  C_c(\R^N)$,  ensuing  by approximation of characteristic functions (see \cite[Lemma 8.3]{BVH}) that
for each Borel set $B\subset\R^\N$:
\begin{equation}\label{eq|Smi1|old}
	\bs{\mu}(B)=\int_{\mathcal{C}_\ell} \bs{\nu}(B)\ d\Xi(\bs{\nu}),
	\qquad
	|\bs{\mu}|(B)=\int_{\mathcal{C}_\ell} |\bs{\nu}|(B)\ d\Xi(\bs{\nu}).
\end{equation}

%



\subsection{Phase portraits associated with divergence-free measures}

The theorem below  is fundamental to our purposes.
It allows us to construe (generally in a non-unique manner) Smirnov's decomposition of a divergence-free measure $\bs{\mu}$ as a ``phase portrait''
of the differential equation $\bs{x}'=\bb{u}_{\bs{\mu}}\circ\bs{x}$, at least
 $|\bs{\mu}|$-a.e. 
The main idea of the proof 
is to find a large set of curves in the decomposition
\eqref{eq|Smi1|old} that
we can concatenate without leaving the support of $\bs{\mu}$,
and then to show that the image of curves not belonging to this set has zero $|\bs{\mu}|$-measure.   We mention  \cite[Proof of Theorem 2.1]{SmirnovKhavin1998} for a (different) 
construction of special 
integral curves joining given initial and terminal points,
when both $\bs{\mu}$ and $\mathrm{div}\bs{\mu}$ are  measures and $\bs{\mu}$ is
an extreme point  of the convex set $\{\supp \bs{\mu}\subset K,\ \|\bs{\mu}\|_{TV}\leq1,
\|\mathrm{div}\bs{\mu}\|_{TV}\leq1\}$.

\begin{theorem}\label{prop|smi_tra}
	Let $\bs{\mu}\in\mathcal{M}(\R^\N)^\N$ be  divergence-free and  $X$ be a $\bs{\mu}$-measurable set of full $|\bs{\mu}|$-measure.	
	Given $t_0\in\R$, there exists for $|\bs{\mu}|$-a.e. $x_0\in X$ a Lipschitz continuous function $\bs{f}:\R\to\R^\N$ with $\bs{f}(t_0)=x_0$
        such that $\bs{f}(t)\in X$ and $\bs{f}'(t)=\bb{u}_{\bs{\mu}}(\bs{f}(t))$
        for a.e. $t\in\R$.
%
\end{theorem}
\begin{proof}
	Fix a positive Borel measure $\Xi$
        on $\mathcal{C}_1$ such that \eqref{eq|Smi1|old} is satisfied with $\ell=1$. We may assume without loss of generality that $\Xi$ is a Borel measure
        on the closed unit ball $\mathcal{B}$ of $\mathcal{M}(\R^\N)^\N$ which is carried by $\mathcal{C}_1$, and that $\Xi$ is 
        complete meaning that subsets of $\mathcal{B}$ squeezed between two nested Borel sets having the same $\Xi$-measure are $\Xi$-measurable \cite[Theorem 1.36]{Rudinrca}. Appealing to  Lusin's theorem we can find a sequence of
        functions  ${\bs u}_k\in C_c(\R^N,\R^N)$, with $|{\bs u}_k|\leq1$ everywhere, that converges pointwise  to $\bb{u}_{\bs{\mu}}$ outside a Borel set $E$
        of $|\bs{\mu}|$-measure zero. By the Banach-Alaoglu theorem
        the  function $\bs{\nu}\mapsto|\bs{\nu}|(F)$ is lower semicontinuous
        for the weak-star topology on $\mathcal{B}$ when $F$ is closed, thus also for any  Borel set $F\subset\mathcal{B}$ by inner regularity of finite Borel measures on $\R^N$.
        In particular $\bs{\nu}\mapsto|\bs{\nu}|(E)$ is measurable, and since
        $|\bs{\mu}|(E)=0$
        it integrates to zero against $\Xi$ in view of \eqref{eq|Smi1|old}.
        Hence, there is a Borel set $O$ of $\Xi$-measure zero such that whenever $\mathcal{C}_1\ni\bs{\nu}\not\in O$, one has $\lim_k {\bs u}_k(\bs{x})={\bs u}_{\bs{\mu}}(\bs{x})$
        for $|\bs{\nu}|$-a.e. $\bs{x}$. Thus,
	applying  \eqref{Smiw} with ${\bs g}={\bs u}_k$ and passing to the limit when $k\to\infty$ using the dominated convergence theorem, we obtain that
	\begin{align}
           \label{eq|RaNi_deco}
          \lim_k\int_{\mathcal{C}_1} \left(\int \bb{u}_k\cdot d\bs{\nu}\right)d\Xi(\bs{\nu}) =\lim_k\int\bb{u}_k\cdot d\bs{\mu}=\int\bb{u}_{\bs{\mu}}\cdot d\bs{\mu}\\
          =\|\bs\mu\|_{TV}=|\bs{\mu}|(\R^N) =\int_{\mathcal{C}_1} \|\bs{\nu}\|_{TV}\,d\Xi(\bs{\nu})=\Xi(\mathcal{C}_1),
	\end{align}
	where the next-to-last equality uses \eqref{eq|Smi1|old}  and the
        last uses that $\|\bs{\nu}\|_{TV}=1$ for $\bs{\nu}\in\mathcal{C}_1$.
        Besides, for fixed $k$ the set of $\bs{\sigma}\in\mathcal{M}(\R^N)^N$ such that $\langle {\bs u}_k,\bs{\sigma}\rangle\leq 1-\varepsilon$ is closed in the weak-star topology by definition of the latter, and since
        $|\int \bb{u}_k\cdot d\bs{\nu}|\leq 1$ for all $\bs{\nu}\in\mathcal{C}_1$ while $\lim_k\int\langle{\bs u}_k,\bs{\nu}\rangle d\Xi(\bs{\nu})\to \Xi(\mathcal{C}_1)$ by \eqref{eq|RaNi_deco}, we deduce that there is a Borel subset $V$ of
        $\mathcal{C}_1$ with $\Xi(V)=0$ such that   $\lim_k\int \bb{u}_k\cdot d\bs{\nu}= 1$ when $\bs{\nu}\not\in V$. Altogether, $W:=O\cup V$ is a Borel  subset of $\mathcal{C}_1$ of $\Xi$-measure zero such that each $\bs{\nu}\in\mathcal{C}_1\setminus W$ satisfies    $\lim_k{\bs u}_k\to{\bs u}_{\bs{\mu}}$ pointwise $\bs{\nu}$ a.e as well as $\lim_k\int \bb{u}_k\cdot d\bs{\nu}= 1$, hence also
        $\int \bb{u}_{\bs{\mu}}\cdot d\bs{\nu}= 1$ by the dominated convergence theorem.   Evidently for such $\bs{\nu}$  it holds $|\bs{\nu}|$-a.e. that
        $\bb{u}_{\bs\mu}=\bb{u}_{\bs{\nu}}$, and if we write $\bs{\nu}=\bs{\pi}(\bs{\gamma})$ with $\bs{\gamma}\in\mathcal{U}_1$ (see Lemma \ref{lemcurCl}), we see from \eqref{RNPi} that
      \[
{\bs u}_{\bs{\mu}}(\bs{\gamma}(t))=\bb{u}_{\bs{\pi}(\bs{\gamma})}(\bs{\gamma}(t))=\bs{\gamma}'(t)\qquad \bs{\gamma}\in\mathcal{U}_1,\quad\bs{\pi}(\bs{\gamma})\not\in W,\quad  \mathrm{\ a.e.\  }t\in[0,1].
        \]
        In another connection, since $X$ is $|\bs{\mu}|$-measurable of full measure while $|\bs{\mu}|$ is regular because it is a finite Borel measure on a metric space \cite[Theorem 1.1]{Bil}, there is a Borel set $G\supset \R^N\setminus X$ such that
        $|\bs{\mu}|(G)=0$ and therefore, arguing as before only replacing $E$ by $G$, we deduce that there is a Borel subset $O_1$ of $\Xi$-measure zero such that $|\bs{\nu}|(G)=0$ whenever $\mathcal{C}_1\ni\bs{\nu}\not\in O_1$.
        In particular,  if we write $\bs{\nu}=\bs{\pi}(\bs{\gamma})$ with $\bs{\gamma}\in\mathcal{U}_1$, we get since $d|\bs{\pi}(\bs{\gamma})|= \#(\bs{\gamma},x)d H^1\mathcal{b}\Gamma$ (see discussion before \eqref{RNPi})
         that $\bs{\gamma}(t)\in X$ for a.e. $t\in[0,1]$ when $\bs{\gamma}\in\mathcal{U}_1$ but $\bs{\pi}(\bs{\gamma})\not\in O_1$.

         Altogether, we find that $\widetilde{\mathcal{X}}:=\mathcal{C}_1\setminus(W\cup O_1)$ is a Borel set of full $\Xi$-measure such that 
                \[
	\widetilde{\mathcal{X}}\subset\{\bs{\pi}(\bs{\gamma})\in\mathcal{C}_1:
	\bs{\gamma}\in\mathcal{U}_1\ \text{and for  a.e. } t\in[0,1]\ 
 	\bb{u}_{\bs\mu}(\bs{\gamma}(t))=\bs{\gamma}'(t) \text{ as well as }\bs{\gamma}(t)\in X\}.
      	\]
        Since $\mathcal{C}^{\mathcal{J}}_1$ is closed in $\mathcal{C}_1$ by Lemma \ref{ouvertferme},
        the set $\widetilde{\mathcal{X}}\setminus\mathcal{C}^{\mathcal{J}}_1$ is Borel
        and so, by the regularity of Borel measures on metric spaces again,
        we can find         
           a countable union of nested closed sets in $\mathcal{B}$, say $\mathcal{X}=\cup_k\mathcal{F}_j$, which is included in  $\widetilde{\mathcal{X}}\setminus\mathcal{C}^{\mathcal{J}}_1$ an such that
           $\Xi(\mathcal{X})=\Xi(\widetilde{\mathcal{X}}\setminus\mathcal{C}^{\mathcal{J}}_1)$.
        The $\mathcal{F}_j$ are in fact compact for the weak-star topology, because so is $\mathcal{B}$.
	
	Let $\beta,\epsilon\in\mathcal{M}(\R^\N)^+$ (positive Borel measures on
        $\R^N$) be defined for $\phi\in C_c(\R^\N)$ by
	\begin{equation}
          \label{defbe}
		\langle{\beta},\phi\rangle:=\int\limits_{\mathcal{C}_1\setminus\mathcal{C}^{\mathcal{J}}_1}\langle\delta_{\mathfrak{b}(\bs{\nu})},\phi\rangle\,d\Xi(\bs{\nu})
		\quad \text{ and } \quad
		\langle{\epsilon},\phi\rangle:=\int\limits_{\mathcal{C}_1\setminus\mathcal{C}^{\mathcal{J}}_1}\langle\delta_{\mathfrak{e}(\bs{\nu})},\phi\rangle\,d\Xi(\bs{\nu}).
	\end{equation}
	Computing the divergence under the integral sign in \eqref{Smiw}
         and using Lemma \ref{lemcurCl}   (where $\ell$ is set to $1$), we deduce  that for $g\in C_c^\infty(\R^N)$
                 \[
            \langle \div\bs{\mu},g\rangle=\int\limits_{\mathcal{C}_1}\langle\div\bs{\nu}, g\rangle\,d\Xi(\bs{\nu})=
            \int\limits_{\mathcal{C}_1\setminus\mathcal{C^J}}\langle\delta_{\mathfrak{e}(\bs{\nu})}-\delta_{\mathfrak{b}(\bs{\nu})}\,,\, g\rangle\,d\Xi(\bs{\nu}).
            \]
            Thus, since $\bs{\mu}$ is divergence-free, we get  that $\langle\epsilon,\phi\rangle=\langle\beta,\phi\rangle$ for $\phi\in C_c^\infty(\R^N)$, whence $\epsilon=\beta$.
	For  $U\subset\R^\N$ an open set, it follows  from \eqref{defbe} by a standard argument based on the inner regularity of finite Borel measures on $\R^N$ and Urysohn's lemma  that 
	\begin{equation*}
          \beta(U)=\!\!\!\!\sup_{\stackrel{\phi\in C_c(U)}{0\leq\phi\leq1}}\langle\beta,\phi\rangle
          =\int\limits_{\mathcal{C}_1\setminus\mathcal{C^J}}\chi_U(\mathfrak{b}(\bs{\nu}))\,d\Xi(\bs{\nu}),
		\quad \text{ and likewise} \quad
		\epsilon(U)=\int\limits_{\mathcal{C}_1\setminus\mathcal{C^J}}\chi_U(\mathfrak{e}(\bs{\nu}))\,d\Xi(\bs{\nu}).
	\end{equation*}
	By outer regularity of finite Borel measures on $\R^N$,
        the previous equality holds in fact for any Borel set $B$ instead of $U$. That is to say,  for every Borel set $B\subset\R^N$ one has
	\begin{equation}\label{eq|betaeta}
		\beta(B)
		=\int\limits_{\mathcal{C}_1\setminus\mathcal{C^J}}\chi_{\mathfrak{b}^{-1}(B)}(\bs{\nu})\,d\Xi(\bs{\nu})
		=\Xi(\mathfrak{b}^{-1}(B)),
		\ \text{ and \ likewise }\ 
		\epsilon(B)=\Xi(\mathfrak{e}^{-1}(B)).
	\end{equation}

	Our  next step is to define two specific sets $\mathcal{B}^F,\mathcal{E}^F\subset\mathcal{X}$;
        roughly speaking, $\mathcal{X}\setminus\mathcal{B}^F$ will be  the image under $\bs{\pi}$ of
       a ``big'' set of curves in $\bs{\pi}^{-1}(\mathcal{X})\cap\mathcal{U}_1$ that can be indefinitely continued
     ``backwards'' by concatenation with other curves from
       $\bs{\pi}^{-1}(\mathcal{X})\cap\mathcal{U}_1$; analogously, $\mathcal{X}\setminus\mathcal{E}^F$ will correspond to  curves in
        $\bs{\pi}^{-1}(\mathcal{X})\cap\mathcal{U}_1$ that can be indefinitely continued ``forwards''. 
	Here, the continuation may use infinitely many times the same curve.
	We start by defining
	\[\mathcal{B}^0:=\{\bs{\nu}\in\mathcal{C}_1\setminus\mathcal{C}^{\mathcal{J}}_1 : \mathfrak{b}(\bs{\nu})\not\in\mathfrak{e}(\mathcal{X})\}
	=\mathfrak{b}^{-1}(\R^\N\setminus\mathfrak{e}(\mathcal{X}))\]
	(corresponding to
        curves that do not begin where a curve of $\bs{\pi}^{-1}(\mathcal{X})\cap\mathcal{U}_1$ ends) and
        note, since $\mathfrak{e}(\mathcal{X})$ is a countable union of nested compact sets and $\mathfrak{b}$ is continuous by Lemma \ref{ouvertferme}, that $\mathcal{B}^0$ is a countable intersection of nested relative open sets in $\mathcal{C}_1\setminus\mathcal{C}^{\mathcal{J}}_1$, hence a Borel set in $\mathcal{B}$.
	We first  show by induction that the same is true of every $\mathcal{B}^n$.  To this end, let us assume for some $n>0$ that $\mathcal{B}^{n-1}$ is a countable intersection of nested relative open sets $\mathcal{O}^{n-1}_k$ in $\mathcal{C}_1\setminus\mathcal{C}^{\mathcal{J}}_1$, so that
        $\left(\mathcal{C}_1\setminus \mathcal{C}_1^{\mathcal{J}}\right)\setminus \mathcal{O}^{n-1}_k=\mathcal{K}_k^{n-1}\cap \left(\mathcal{C}_1\setminus \mathcal{C}_1^{\mathcal{J}}\right)$ where  the $\mathcal{K}_k^{n-1}$ are nested  closed
        subsets of $\mathcal{B}$,  and let
	\begin{equation}
          \label{recB}
          \mathcal{B}^{n}:=
          \mathfrak{b}^{-1}(\cap_k\mathfrak{e}({\mathcal{O}}_k^{n-1})\setminus\mathfrak{e}(\mathcal{X}\setminus{\mathcal{B}}^{n-1}))\supset\mathfrak{b}^{-1}(\mathfrak{e}({\mathcal{B}}^{n-1})\setminus\mathfrak{e}(\mathcal{X}\setminus{\mathcal{B}}^{n-1}))
	\end{equation}
	($ \mathcal{B}^{n}$ contains the members of
          $\mathcal{C}_1\setminus\mathcal{C}^{\mathcal{J}}_1$
          whose beginning  point is an endpoint of a member of
        $\mathcal{B}^{n-1}$ but of no member of
        $\mathcal{X}$ not in $\mathcal{B}^{n-1}$).
      By definition of $\mathcal{F}_j$, $\mathcal{O}^{n-1}_k$ and $ \mathcal{K}_k^{n-1}$ it holds:
        \[
          \mathcal{X}\setminus{\mathcal{B}}^{n-1}=\left(\cup_j\mathcal{F}_j\right)\setminus\left(\cap_k\mathcal{O}^{n-1}_k\right)=
          \bigcup_j\mathcal{F}_j\cap  \left(\cup_k\left(\mathcal{C}_1\setminus \mathcal{C}_1^{\mathcal{J}}\right)\cap
                \mathcal{K}_k^{n-1}\right)=\cup_{j,k}\mathcal{F}_j\cap \mathcal{K}_k^{n-1}
            \]
          where we used in the last equality that $\mathcal{F}_j\subset\mathcal{C}_1\setminus \mathcal{C}_1^{\mathcal{J}}$. Thus,  $\mathcal{X}\setminus{\mathcal{B}}^{n-1}$ is a countable union of nested compact sets, and we deduce that 
        $\mathfrak{e}(\mathcal{X}\setminus{\mathcal{B}}^{n-1})$ is in turn a countable union of compact sets in $\R^N$, by the continuity of $\mathfrak{e}$ asserted in by Lemma \ref{ouvertferme}. As
        $\cap_k\mathfrak{e}({\mathcal{O}}_k^{n-1})$
        is a countable intersection of nested open sets in
        $\R^N$ by the openness of $\mathfrak{e}$ also asserted in
        Lemma \ref{ouvertferme},
        we conclude that
$\cap_k\mathfrak{e}({\mathcal{O}}_k^{n-1})\setminus\mathfrak{e}(\mathcal{X}\setminus{\mathcal{B}}^{n-1})$ is a countable intersection of nested open sets.
Hence, by the continuity of $\mathfrak{e}$, the set $\mathcal{B}^{n}$ defined in \eqref{recB} is  a countable intersection of nested  relative open sets in $\mathcal{C}_1\setminus\mathcal{C}^{\mathcal{J}}_1$. This completes the induction step.

Let now $\mathcal{B}^F=\bigcup_n\mathcal{B}^n$ and define $\mathcal{E}^F$
analogously, only exchanging throughout $\mathfrak{b}$ and $\mathfrak{e}$.
	By what precedes,  $\mathcal{B}^F$ and $\mathcal{E}^F$ are Borel subsets of
        $\mathcal{C}_1\setminus\mathcal{C}^{\mathcal{J}}_1$.
		Since $\bb{u}_{\bs{\mu}}(\bs{\gamma}(t))=\bs{\gamma}'(t)$ for a.e. $t\in[0,1]$ whenever $\bs{\gamma}\in\mathcal{U}_1$ and $\bs{\pi}(\bs{\gamma})\in\mathcal{X}$, while  curves in $\bs{\pi}^{-1}(\mathcal{C}^{\mathcal{J}}_1\cap\widetilde{X})\cap\mathcal{U}_1$ can be continued indefinitely both backwards and forwards in that set as they are closed curves, we can extend any $\bs{\gamma}\in\mathcal{U}_1$
                  with $\bs{\pi}(\bs{\gamma})\in(\mathcal{X}\setminus(\mathcal{B}^F\cup\mathcal{E}^F))\cup (\widetilde{\mathcal{X}}\cap\mathcal{C}^{\mathcal{J}}_1)$ to a Lipschitz continuous function $\bs{f}:\R\to\supp\bs{\mu}$ such that  $\bs{f}'(t)=\bb{u}_{\bs{\mu}}(\bs{f}(t))$ and $\bs{f}(t)\in X$ for a.e.
                $t\in\R$. 
	Let us pick $\mathcal{Y}\subset(\mathcal{X}\setminus(\mathcal{B}^F\cup\mathcal{E}^F))\cup (\widetilde{\mathcal{X}}\cap\mathcal{C}^{\mathcal{J}}_1)$ to be a countable union of nested closed sets in $\mathcal{B}$ such that $\Xi(\mathcal{Y})=\Xi((\mathcal{X}\setminus(\mathcal{B}^F\cup\mathcal{E}^F))\cup (\widetilde{\mathcal{X}}\cap\mathcal{C}^{\mathcal{J}}_1))$; such a $\mathcal{Y}$ exists by inner regularity of Borel measures on metric spaces.
	Equip the product space $\mathrm{Lip}_1([0,1],\R^\N)\times[0,1]$ with the
        product of the uniform topology and the standard topology of $[0,1]$.
        	Then,  the function $\mathfrak{Ev}:\mathrm{Lip}_1([0,1],\R^\N)\times[0,1]$ defined by $\mathfrak{Ev}(\bs{\gamma},t):=\bs{\gamma}(t)$, is continuous.
	Also, by the continuity of $\bs{\pi}$ asserted in Lemma \ref{lemma|map_pi},
        $\bs{\pi}^{-1}(\mathcal{Y})$ is a countable union of closed sets in
        $\mathrm{Lip}_1([0,1],\R^\N)$ equipped with the uniform topology which, since closed balls of $\mathrm{Lip}_1([0,1],\R^\N)$ are compact for that topology by the Arzelà-Ascoli theorem, can evidently be rewritten as a countable union of compact sets.
	Thus, by Tikhonov's theorem, the subset $I:=\mathfrak{Ev}(\bs{\pi}^{-1}(\mathcal{Y})\times[0,1])$ of $\R^N$ is a countable union of compact sets and in particular a Borel set.
	Noticing that
	\[
	I\subset\bigcup_{\bs{\gamma}\in\bs{\pi}^{-1}\left((\mathcal{X}\setminus(\mathcal{B}^F\cup\mathcal{E}^F))\cup (\widetilde{\mathcal{X}}\cap\mathcal{C}^{\mathcal{J}}_1)\right)\cap\mathcal{U}_1}\bs{\gamma}([0,1])
	\]
	(the right hand side is a union of images of curves that can be continued indefinitely backwards and forwards by concatenation with
          curves in $\bs{\pi}^{-1}(\widetilde{\mathcal{X}})$
      ),
        it is enough
        to finish the proof that we establish the equality $|\bs{\mu}|(I)=\|\bs{\mu}\|_{TV}$, for this will imply that the initial condition $x_0$ in the statement can indeed be chosen over a set of full $|\bs{\mu}|$-measure in $X$.
	
	First, let us prove that $\Xi(\mathcal{B}^F)=0$.
	For this,
        observe from \eqref{eq|betaeta} that
	\begin{equation}
          \label{egebinduc}
	\beta(\R^\N\setminus\mathfrak{e}(\mathcal{X}))
	=\Xi(\mathfrak{b}^{-1}(\R^\N\setminus\mathfrak{e}(\mathcal{X})))
	=	\Xi(\mathcal{B}^0).
	\end{equation}
	Analogously, one has
	\begin{equation}
          \label{egebinduc1}
	\epsilon(\R^\N\setminus\mathfrak{e}(\mathcal{X}))
	=\Xi(\mathfrak{e}^{-1}(\R^\N\setminus\mathfrak{e}(\mathcal{X})))
	\leq\Xi(\left(\mathcal{C}_1\setminus\mathcal{C}^{\mathcal{J}}_1\right)\setminus\mathcal{X})=0
	\end{equation}
	where the last equality holds because
          $\mathcal{X}$ has full $\Xi$-measure in $\mathcal{C}_1\setminus\mathcal{C}^{\mathcal{J}}_1$. As  $\epsilon=\beta$, we conclude from \eqref{egebinduc} and \eqref{egebinduc1} that $\Xi(\mathcal{B}^0)=0$.
	Next, assume for some $n>0$ that $\Xi(\mathcal{B}^{n-1})=0$.
	Then,  using \eqref{recB}  together with \eqref{eq|betaeta} and the fact that $\beta=\epsilon$ again, we can write
	\begin{align*}
		\Xi(\mathcal{B}^n)
          &=
            \beta(\cap_k\mathfrak{e}({\mathcal{O}}_k^{n-1})\setminus\mathfrak{e}(\mathcal{X}\setminus{\mathcal{B}}^{n-1}))
            =\epsilon(\cap_k\mathfrak{e}({\mathcal{O}}_k^{n-1})\setminus\mathfrak{e}(\mathcal{X}\setminus{\mathcal{B}}^{n-1}))
          \\
          &=\Xi(\mathfrak{e}^{-1}(\cap_k\mathfrak{e}({\mathcal{O}}_k^{n-1})\setminus\mathfrak{e}(\mathcal{X}\setminus{\mathcal{B}}^{n-1})))
		\leq\Xi(\mathcal{B}^{n-1})=0,
	\end{align*}
        where the last inequality is obvious since we deprived $\cap_k\mathfrak{e}({\mathcal{O}}_k^{n-1})$ from $\mathfrak{e}(\mathcal{X}\setminus{\mathcal{B}}^{n-1})$ before taking the inverse image under $\mathfrak{e}$.
	Thus, we conclude by induction that $\Xi(\mathcal{B}^F)=0$.
	An analogous argument shows that $\Xi(\mathcal{E}^F)=0$.
	Recall that for all $\bs{\pi}(\bs{\gamma})$ with $\bs{\gamma}\in\mathcal{C}_1$ we have $\|\bs{\pi}(\bs{\gamma})\|_{TV}=1$, and also that $\|\bs{\mu}\|_{TV}=\Xi(\mathcal{X})=\Xi(\widetilde{\mathcal{X}})$.
	Hence,  applying the right equation of \eqref{eq|Smi1|old} to $I$, we
        obtain that
	\begin{align*}
		|\bs{\mu}|(I)
		&=\int_{\mathcal{C}_1}|\bs{\nu}|(I)\,d\Xi(\bs{\nu})
		\geq \int_{\mathcal{Y}}|\bs{\nu}|(I)\,d\Xi(\bs{\nu}) 
		=\int_{\mathcal{Y}}\|\bs{\nu}\|\,d\Xi(\bs{\nu})\\
		&=\Xi(\mathcal{Y})=\Xi((\mathcal{X}\setminus(\mathcal{B}^F\cup\mathcal{E}^F))\cup (\widetilde{\mathcal{X}}\cap\mathcal{C}^{\mathcal{J}}_1))=\Xi(\mathcal{X})+\Xi(\widetilde{\mathcal{X}}\cap\mathcal{C}^{\mathcal{J}}_1)\\
		&=\Xi(\widetilde{\mathcal{X}}\setminus\mathcal{C}^{\mathcal{J}}_1)+\Xi(\widetilde{\mathcal{X}}\cap\mathcal{C}^{\mathcal{J}}_1)=\Xi(\widetilde{\mathcal{X}})=\|\bs{\mu}\|_{TV},
	\end{align*}
        where  the last equality in the first line is because $I$ contains the support of every member of $\mathcal{Y}$ by its very definition, and
        the last equality in the second line  uses  the fact that $\mathcal{B}^F\cup\mathcal{E}^F$ is a Borel set of $\Xi$-measure zero.
        \end{proof}

\begin{exa}\label{exa|infinf}
  This illustrates that the function $\bs{f}$ in Theorem~\ref{prop|smi_tra}  may be highly non-unique.
  
    Let $a:\R^2\to\R$ be defined by $a(\bs{x})=1-|\bs{x}|$ and $R$ be the counter-clockwise $\pi/2$ rotation on $\R^2$.
    For any $\theta\in[0,2\pi]$, let $\bs{x}_{\theta}:=(\cos(\theta),\sin(\theta))^t$ and let $L_{\theta}$ be the line segment from the origin to $\bs{x}_{\theta}$ (without endpoints).
    For $0\leq\theta<\phi\leq2\pi$ let $C_{\theta,\phi}:=\bigcup_{\theta<\psi<\phi}L_{\psi}$.
    Define the measure $\bs{\mu}_{\theta,\phi}\in\mathcal{M}(\R^2)^2$ by:
    \begin{enumerate}
        \item $\supp{\bs{\mu}_{\theta,\phi}}=\overline{C_{\theta,\phi}}$.
        \item $|\bs{\mu}_{\theta,\phi}|=a\mathcal{L}_2$ on $C_{\theta,\phi}$.
        \item $|\bs{\mu}_{\theta,\phi}|=a\mathcal{H}^1$ on $L_{\theta}$ and $L_{\phi}$.
        \item \[
            \bb{u}_{\bs{\mu}_{\theta,\phi}}(\bs{x})=
             \begin{cases}
                \frac{R(\bs{x})}{|\bs{x}|}   &   \text{for } \bs{x}\in C_{\theta,\phi},\\
                \frac{\bs{x}}{|\bs{x}|}   &   \text{for } \bs{x}\in L_{\theta},\\
                -\frac{\bs{x}}{|\bs{x}|}   &   \text{for } \bs{x}\in L_{\phi}.
            \end{cases}   
        \]
    \end{enumerate}
    
    Alternatively, we may write $\bs{\mu}_{\theta,\phi}=\int_0^1 \bs{p}_r\, dr$ where $\bs{p}_r=\bs{\pi}(\gamma_r)$ and $\gamma_r$ denotes the arclength parametrized closed  loop traversing the boundary of $rC_{\theta,\phi}$ oriented in the counterclockwise direction.   As $\bs{p}_r$ is divergence free, it follows that   $\bs{\mu}_{\theta,\phi}$   is also divergence-free.   If $\bs{f}$ is as in Theorem~\ref{prop|smi_tra} for  $\bs{\mu}_{\theta,\phi}$, we note that $\bs{f}(t)\in L_\theta$  infinitely often and for any such $t$ we may choose to continue either along $L_\theta$ or along the circular arc from that point.
    Superposing countably many weighted phase portraits of this kind lead to examples with dense subsets of point of  non-uniqueness for $\bs{f}$, where
    switching to another trajectory is permitted.
\end{exa}


\section{Uniqueness of  the regularized solution  under conditions \ref{conU}}
\label{sec|Uni}

      
In this section, we will consider a linear operator $\Af:\mathcal{M}(S)^\N\to \HH$, where $\HH$ is a real Hilbert space, that is sequentially continuous in the weak-star topology on $\mathcal{M}(S)^\N$. By the Banach-Alaoglu Theorem, any $TV$-norm bounded sequence
in $\mathcal{M}(S)^\N$ contains a weak-star convergent subsequence and hence, the sequential weak-star continuity of $\Af$ implies that
$\Af$ is compact; i.e., $\Af$ maps TV-norm bounded sets into precompact sets in $\HH$. 

It is well-known that the compactness of $\Af$ makes the solving of $A\bs{\mu}=f$ an ill-posed problem when $f$ is given and $\bs{\mu}$ is the unknown.
Moreover, for our main motivation of magnetic source recovery
problems, $\Af$ is not even injective.  It is classical to use regularization methods to produce `good' approximate solutions to $\Af\bs{\mu}=f$, by adjoining a penalty functional to a  criterion which,  though natural,  fails to attain its minimum in general due to ill-posedness.
In \cite{BVHNS} and \cite{BVH},   we considered regularization using the TV-norm and showed this is consistent for recovering  a certain class of measures including sparse measures, e.g., those with finite support (see discussion below).
Specifically,  for $f\in  \HH$, $\lambda>0$, and $\sigma$ a positive Lipschitz weight function on $S$, recall from \eqref{eq|main}  the functional  
\begin{equation} 
	\label{eq|defcrit0}
	\mathcal{F}(\bs{\mu})=
	\|f-\Af\bs{\mu}\|_{\HH}^2+\lambda \|\sigma\bs{\mu}\|_{TV}.
\end{equation}

Since $\mathcal{F}(\bs{0})=\|f\|_{\HH}^2<\infty$ and there is a constant $C>0$ such that $1/C<\sigma<C$ on $S$, it follows that  for any $\bs{\mu}$ with $\|\bs{\mu}\|_{TV}>(C/\lambda )\|f\|_{\HH}^2$, we have
$$\mathcal{F}(\bs{\mu})\ge \lambda \|\sigma\bs{\mu}\|_{TV}\ge  (\lambda/C) \|\bs{\mu}\|_{TV}>\|f\|_{\HH}^2\,\geq \mathcal{F}(\bs{0}).$$
It then follows from the Banach–Alaoglu theorem  and our assumptions on $\sigma$  that there exists a minimizer $\widetilde{\bs{\mu}}$ of 
$\mathcal{F}$ over $\mathcal{M}(S)^\N$; i.e., 
\begin{equation} \label{eq|mulamfVreg}
	\mathcal{F}(\widetilde{\bs{\mu}})= \inf_{\bs{\mu}\in \mathcal{M}(S)^\N} \mathcal{F}(\bs{\mu}).
\end{equation}

The next lemma provides us with a variational characterization of
such minimizers
(compare \cite{BrePikk}). Note that $\sigma$ is merely assumed positive and continuous there.
 
\begin{lemma}\label{lemma|CPV}
  Let $S\subset\R^N$ be compact, $\HH$ be a real Hilbert space and  $\Af:\mathcal{M}(S)^\N\to \HH$  a linear sequentially continuous operator for the weak-star topology on $\mathcal{M}(S)^\N$.
  Let moreover $\sigma$ be a positive continuous weight function on $S$.  
 Then, for  $f\in \mathcal{H}$, a measure $\bs{\mu}\in \mathcal{M}(S)^\N$ is a minimizer of $\mathcal{F}$   	if and only if
\begin{equation}
	\label{eq|CP}
	\begin{array}{ll}
		 \Af^*(f-\Af\bs{\mu})&=\frac{\lambda}{2} \sigma\bb{u}_{\bs{\mu}} 
		\qquad |\bs{\mu}|\text{\rm -a.e. and}\\ \\
		\left | \Af^*(f-\Af\bs{\mu})\right|&\leq \frac{\lambda}{2}\sigma\quad\text{\rm everywhere on } S.
	\end{array}
\end{equation}
	Moreover, if $\bs{\mu}' \in\mathcal{M}(S)^\N$ is another minimizer of $\mathcal{F}$, then  
 $A(\bs{\mu}'-\bs{\mu})={0}$. \end{lemma}
\begin{proof}
  For $\mathcal{G}$ a convex functional  on ${\mathcal{M}(S)^\N}$ and  $\bs{\mu}\in {\mathcal{M}(S)^\N}$, we let $\partial\mathcal{G}(\bs{\mu})$ denote the subdifferential of $\mathcal{G}$ at $\bs{\mu}\in {\mathcal{M}(S)^\N}$;
  that is, denoting by $(\mathcal{M}(S)^\N)^*$ the dual of $\mathcal{M}(S)^\N$ endowed with the strong topology, one has:
	\[
		\partial\mathcal{G}(\bs{\mu}):=\left\{\mathcal{N}\in {(\mathcal{M}(S)^\N)^*}\ :\ \mathcal{G}(\bs{\nu})-\mathcal{G}(\bs{\mu})\geq \langle\mathcal{N},\bs{\nu}-\bs{\mu}\rangle \text{, for any } \bs{\nu}\in {\mathcal{M}(S)^\N}\right\}.
	\]
	Let $\mathcal{R}$ be the convex functional on ${\mathcal{M}(S)^\N}$ sending 
        $\bs{\mu}$ to $ \|f-A\bs{\mu}\|_{\HH}^2$. Since $A$ is a fortiori strongly continuous, we
        note that $\mathcal{R}$ has a Fr\'echet derivative $\grad\mathcal{R}$ such that, for any $\bs{\nu}\in {\mathcal{M}(S)^\N}$, it holds  $\langle \grad\mathcal{R},\bs{\nu}\rangle= \langle\bs{\nu},2A^*(A\bs{\mu}-f)\rangle$.
Let $\|\cdot\|_{TV}^*$ indicate the norm in $(\mathcal{M}(S)^\N)^*$ and observe that
	$$\partial(\|\cdot\|_{{TV}})(\bs{\mu})=
	\left\{\begin{array}{lr}
		\left\{\mathcal{N}\in {(\mathcal{M}(S)^\N)^*}\ :\ \|\mathcal{N}\|{_{TV}^*}=1 \text{ and } \langle\mathcal{N},\bs{\mu}\rangle=\|\bs{\mu}\|_{TV}\right\} & \text{if } \bs{\mu}\neq\bs{0} \\
		\left\{\mathcal{N}\in {(\mathcal{M}(S)^\N)^*}\ :\ \|\mathcal{N}\|{_{TV}^*}\leq1\right\} & \text{if } \bs{\mu}=\bs{0}
	\end{array}\right..$$
      In particular, the weak-star continuous functionals in $\partial(\|\cdot\|_{{TV}})(\bs{\mu})$, if any,  are of the form $\bs{\nu}\mapsto\int\bs{f}\cdot d\bs{\nu}$ where $\bs{f}$ is  a function
      of unit norm  in $C(S,\R^N)$ equal to $\bb{u}_{\bs{\mu}}$ at $\bs{\mu}$-a.e. point if $\bs{\mu}\neq\bs{0}$, and any function with norm at most $1$
      in $C(S,\R^N)$ in case $\bs{\mu}=\bs{0}$.
      
	Now, by classical properties of subdifferentials (e.g., see \cite[Chapter III, Propositions 9 \& 12 and Corollary 1]{ekeland_turnbull_1983}), {we obtain upon denoting multiplication by $\sigma$ with $M_\sigma$ that}
	$$
	\partial\mathcal{F}(\bs{\mu}) = \{\grad\mathcal{R}\} + \lambda{(M_\sigma)^*}\left(\partial(\|\cdot\|_{{\mathcal{M}(S)^\N}})({\sigma}\bs{\mu})\right).
	$$
	The first summand on the right hand side consists of the singleton 
        $\nabla\mathcal{R}$; that is, the member of $(\mathcal{M}(S)^N)^*$
        acting on $\bs{\nu}\in\mathcal{M}(S)^N$ by integrating  the \emph{function}  $\nabla\mathcal{R}\cdot{\mathbf u}_{\bs{\nu}}$ against $|\bs{\nu}|$.
        Thus, as $\bs{\mu}$ being a minimizer of $\mathcal{F}$ over ${\mathcal{M}(S)^\N}$ is equivalent to say that $\bs{0}\in \partial\mathcal{F}(\bs{\mu})$, this will occur if and only if
        \begin{equation}
          \label{minsubdiff}
          -\frac{1}{\lambda}\grad\mathcal{R}\in (M_\sigma)^*({\partial(\|\cdot\|_{{TV}})({\sigma}\bs{\mu})}).
          \end{equation}
          Clearly,  since $\sigma>0$, it holds when $\mathcal{N}\in (\mathcal{M}(S)^N)^*$ that $(M_\sigma)^*\mathcal{N}$ is a weak star continuous functional given by integration against a function $\varphi\in C(S,\R^N)$, if and only if $\mathcal{N}$ is the weak-star continuous functional
           given by integration against the function $\varphi/\sigma\in C(S,\R^N)$. 
    Moreover,  $\bs u_{\sigma\bs\mu}=\bs u_{\bs\mu}$ holds $|\bs{\mu}|$-a.e. by the definition of $\sigma\bs{\mu}$ given after Equation \eqref{eq|main}, and
    thus the weak-star  functionals in ${\partial(\|\cdot\|_{TV})({\sigma}\bs{\mu})}$ are the same as in ${\partial(\|\cdot\|_{{TV}})(\bs{\mu})}$. So,
    \eqref{minsubdiff} is equivalent to say that
    $\frac{1}{\lambda\sigma}\grad\mathcal{R}$, which  is a weak-star
    continuous member of $(\mathcal{M}(S)^N)^*$, actually lies in 
    $\partial(\|\cdot\|_{TV})(\bs{\mu})$.
                      Consequently, $\bs{\mu}$ is a minimizer of $\mathcal{F}_{\lambda, f}$ over ${\mathcal{M}(S)^\N}$ if and only \eqref{eq|CP} is satisfied.
	
	{Finally, the last statement of the lemma follows from the convexity of the $TV$-norm and the strict convexity of the norm in $\mathcal{H}$.}
\end{proof}

We will also need the following  consequence of classical existence  and uniqueness results on solutions to the Cauchy
problem (Carath\'{e}odory's Theorem).       
 
 \begin{lemma}\label{lemma|CP}
Suppose $\Phi:\R^N\to \R^N$ is Lipschitz and,
    for $x_0\in \R^N$, let $w(t)=w(t,x_0)$ denote the unique  absolutely continuous solution of the autonomous Cauchy problem
    \begin{equation} \dot{w}(t)=H(w(t)), \qquad t\in\R, \quad  w(0)=x_0,\label{wEqn}\end{equation}
    where a dot above a quantity indicates its derivative with respect to $t$.
    Let  $s:\R\to \R$ be a bounded measurable function 
    and put
    \begin{equation}\label{betaDef}\beta(t):= \int_0^t s(\tau)d\tau. \end{equation} Then $z(t):=w(\beta(t),x_0)$ is the unique absolutely continuous solution of 
\begin{equation}\dot{z}(t)=s(t)H(z(t)), \qquad {\rm a.e.}\,   t\in \R, \quad  z(0)=x_0.\label{xEqn}\end{equation}   
 \end{lemma}
\begin{proof}
  It follows from    \cite[Thms 5.1 and 5.2]{HaleODE} that there exists a unique
  absolutely continuous solution  of \eqref{xEqn}.  Using the chain rule, one sees by inspection that $z(t):=w(\beta(t),x_0)$ is this solution.    
\end{proof}

We are now in position to prove our main result Theorem~\ref{thm|uniqueness}.

\begin{proof}[Proof of Theorem~\ref{thm|uniqueness}.]
	Assume for a contradiction that $\widetilde{\bs{\mu}}$ and $\check{\bs{\mu}}$ 	are two distinct minimizers of $\mathcal{F}$ and let $\bs{\mu}:=\check{\bs{\mu}}-\widetilde{\bs{\mu}}\neq{\bs 0}$. 
	As $\bs{\mu}$ is absolutely continuous with respect to $|\bs{\mu}|$, the  Radon-Nikodym decompositions of $\widetilde{\bs{\mu}}$ and $\check{\bs{\mu}}$ with respect to $|\bs{\mu}|$ must have the same singular term.
	That is, these decompositions are of the form
	\begin{equation}
          \label{decRNm}
		d\widetilde{\bs{\mu}}=\tilde{\bs{\phi}} d|\bs{\mu}|+d\bs{\nu},\qquad
		d\check{\bs{\mu}}=\check{\bs{\phi}} d|\bs{\mu}|+d\bs{\nu},
	\end{equation}
	where $|\bs{\nu}|$ is singular with respect to $|\bs{\mu}|$ and $\tilde{\bs{\phi}}$, $\check{\bs{\phi}}$ are $|\bs{\mu}|$-integrable $\R^\N$-valued functions, so that $\bs u_{\bs{\mu}}=\check{\bs{\phi}}- \tilde{\bs{\phi}}$   holds $|\bs{\mu}|$-a.e.
			
	The last assertion of Lemma \ref{lemma|CPV} implies  $\Af\bs{\mu}=0$ and so,
defining $\psi:=(2/\lambda)(f-\Af(\widetilde{\bs{\mu}}))$,  we also have $\psi=
	(2/\lambda)(f-\Af(\check{\bs{\mu}}))$.
	Thanks to \eqref{eq|CP}, we know the equalities $\bs u_{\widetilde{\bs{\mu}}}= {\frac{1}{\sigma}}\Af^*\psi$ and $\bs u_{\check{\bs{\mu}}}={\frac{1}{\sigma}}\Af^*\psi$ hold $\widetilde{\bs{\mu}}$-a.e. and $\check{\bs{\mu}}$-a.e. respectively.
	Now, since $d|\widetilde{\bs{\mu}}|=|\tilde{\bs{\phi}}| d|\bs{\mu}| + d|\bs{\nu}|$ and $d|\check{\bs{\mu}}|=|\check{\bs{\phi}}| d|\bs{\mu}| + d|\bs{\nu}|$, we have that
	\begin{align*}
		\bs u_{\bs{\mu}} d|\bs{\mu}|= d\bs{\mu} = \bs u_{\check{\bs{\mu}}}d|\check{\bs{\mu}}|-\bs u_{\widetilde{\bs{\mu}}}d|\widetilde{\bs{\mu}}| = {\frac{1}{\sigma}}\Af^*\psi d|\check{\bs{\mu}}| - {\frac{1}{\sigma}}\Af^*\psi d|\widetilde{\bs{\mu}}| = {\frac{1}{\sigma}}\Af^*\psi \left( |\check{\bs{\phi}}|-|\tilde{\bs{\phi}}| \right) d|\bs{\mu}|,
	\end{align*}
by \eqref{decRNm}, and so $\bs u_{\bs{\mu}} = {\frac{1}{\sigma}}\Af^*\psi (|\check{\bs{\phi}}|-|\tilde{\bs{\phi}}|)$ at $|\bs\mu|$-a.e. point.   Since \[
		\left||\check{\bs{\phi}}|- |\tilde{\bs{\phi}}|\right|\le   |\check{\bs{\phi}}- \tilde{\bs{\phi}}|=|\bs u_{\bs{\mu}}|=1,\qquad |\bs\mu| -{\rm a.e.},
	\] and
$\left|{\frac{1}{\sigma}}\Af^*\psi\right|\leq1$ everywhere by \eqref{eq|CP},
    it follows  that $\left||{\check{\bs{\phi}}}|- |\tilde{\bs{\phi}}|\right|=1$ holds $|\bs\mu|$-a.e.  Hence, there is a $|\bs\mu|$-measurable function $x\mapsto\xi_x:=\left||{\check{\bs{\phi}}}|- |\tilde{\bs{\phi}}|\right|(x)$ 
    on $\supp|\bs{\mu}|$, valued in $\{-1,1\}$, such that $\bs u_{\bs{\mu}}(x) = \xi_x \frac{\Af^*\psi(x)}{{\sigma(x)}}$ for $|\bs{\mu}|$-a.e. $x$.
	Let us define the set of points where this equality holds:	 
	\[
		X:=\left\{x\in \supp|\bs{\mu}| : \bs u_{\bs{\mu}}(x) = \xi_x \frac{\Af^*\psi(x)}{{\sigma(x)}}\right\}.
	\] 
	Being the inverse image of $\{0\}$ under a $|\bs{\mu}|$-measurable map,
        the set $X$ is $|\bs{\mu}|$ measurable, it is nonempty since $\bs{\mu}\neq 0$ and
        it has full $|\bs{\mu}|$ measure.  
	
	Since  $\Af\bs{\mu}=0$, the vector-measure $\bs{\mu}$ is divergence free by assumption.  
	Hence, by Theorem~\ref{prop|smi_tra},  there exists a Lipschitz function $\bs{f}:\R\to\R^\N$ such that $\bs{f}(t)\in X$ and $\bs{f}'(t)=\bb{u}_{\bs{\mu}}(\bs{f}(t))$ for a.e. $t\in\R$; the theorem even implies that we can choose a $|\bs{\mu}|$-almost arbitrary initial point, but this is immaterial to us here.
	Since $\supp|\bs{\mu}|$ is compact and contains $X$, it follows by continuity of $\bs{f}$ that $\bs{f}(\R)\subset\supp\bs{\mu}$.
    Clearly then, $\bs{f}(\R)$ is contained  in a single rectifiably connected
    component of $S$, say $S_0$.
	Using the definition of $X$ and the relation 
 $\bs{f}'(t)=\bb{u}_{\bs{\mu}}(\bs{f}(t))$ 
yields
 \begin{equation*}
   \bs{f}'(t)=\xi_{\bs{f}(t)} \frac{\Af^*\psi(\bs{f}(t))}{{\sigma(\bs{f}(t))}},\qquad \mathrm{a.e.}\  t\in\R.
 \end{equation*}
 
	Let $s:\R\to\{\pm 1\}$ be such that $s(t) =\xi_{\bs{f}(t)}$ for almost every $t\in \R$.  Applying  Lemma~\ref{lemma|CP} with this $s$ and $H={\frac{1}{\sigma}}\Af^*\psi$ while observing this $H$ is Lipschitz continuous by the assumptions on $\sigma$ and $\Af$, we deduce that
 	$\bs{f}(t)=w(\beta(t))$ where $w$ is the solution of the autonomous problem \eqref{wEqn} with $w(0)=\bs{f}(0)$ and $\beta$ is defined in \eqref{betaDef}.     Suppose, for a contradiction, that $\beta$ fails to be monotonic; i.e., suppose there exist $t_1<t_2<t_3$ such that $\beta(t_1)=\beta(t_3)\neq \beta(t_2)$.   Since $\beta$ is Lipschitz and $\beta'=s$ a.e. on $\R$, it follows that the set of `regular values' $J:=\{y\in \beta([t_1,t_3]): \beta'(t)=s(t), \forall t\in \beta^{-1}(y)\cap [t_1,t_3]\}$   is of full measure\footnote{
 		This follows from a  Sard-type theorem for Lipschitz
	   	functions valid in any dimension \cite[Theorem 7.6]{Mattila} but also, more trivially, from 
  		the fact that $\beta^\prime\neq0$ a.e. and  Lipschitz functions $\R^1\to \R^1$
   		map sets of measure zero to sets of measure zero.
   	}     in the interval $\beta([t_1,t_3])$.
 	
	For $y\in J$ there are points 
	$t_y^\pm \in \beta^{-1}(y)\cap [t_1,t_3]$  with $s(t_y^\pm)=\pm 1$    (e.g., we can choose $t_y^\pm$  to be the max and min points in $\beta^{-1}(y)\cap [t_1,t_3]$) leading to the absurdity  $-1=s(t_y^{-})=\xi_{w(y)}=s(t_y^{+})=1$. 
	Hence, $\beta$ is strictly monotonic with either $s\equiv 1$ or $s\equiv -1$ for a.e. $t$ and so, as item (3) of Conditions \ref{conU}
        provides us with a $C^1$-map $g:\R^\N\to\R$ such that $\grad g=\Af^*\psi$ on $S_0$, we have  that
	\begin{equation}\label{eq|der_f}
		\text{either} \quad \bs{f}'
			=\left(\frac{\Af^*\psi}{{\sigma}}\right)\circ\bs{f} 
			=\left(\frac{\nabla g}{{\sigma}}\right)\circ\bs{f} 
		\quad \ \text{\ or\ }\  \quad \bs{f}'
			=-\left(\frac{\Af^*\psi}{{\sigma}}\right)\circ\bs{f}
			=-\left(\frac{\nabla g}{{\sigma}}\right)\circ\bs{f}.
	\end{equation}
 
	Taking the scalar product of \eqref{eq|der_f} with $\bs{f}'$ leads to
	\begin{equation}\label{eq|der_fc2}
		\text{either} \quad {(\sigma\circ\bs{f})}|\bs{f}'|^2=\nabla (g\circ\bs{f})\cdot \bs{f}' = \frac{d}{dt}(g\circ\bs{f})
		\quad \text{or} \quad  {(\sigma\circ\bs{f})}|\bs{f}'|^2=-\nabla (g\circ\bs{f})\cdot \bs{f}' = -\frac{d}{dt}(g\circ\bs{f}).
    \end{equation}

    Hence, with  $\pm$ fixed either to $+$ or to $-$ 
     we have for any   $t_1,t_2\in\R$ that
	\begin{align*}
		(g\circ\bs{f})(t_2)-(g\circ\bs{f})(t_1) 
		=\int_{t_1}^{t_2}\nabla g(\bs{f}(t))\cdot\bs{f}'(t)dt
		=\pm\int_{t_1}^{t_2}{\sigma(\bs{f}(t))}|\bs{f}'(t)|^2dt.
        \end{align*}
	By assumption ${\sigma(\bs{f}(t))}{|\bs{f}'(t)|^2}=\sigma(\bs{f}(t))>c $
	for some positive constant $c$ and a.e. $t$.
	Thus, it follows from the equation displayed above that
	$$
		|(g\circ\bs{f})(t_2)-(g\circ\bs{f})(t_1)| \ge c|t_2-t_1|,
	$$	
	but this is impossible since $g$ is continuous and hence bounded on  $\overline{\bs{f}(\R)}\subset\supp\,\bs{\mu}\subset S$, which is compact.
\end{proof}

\section{Applications to elliptic equations and integral operators}
\label{CondEUsec}
In this section, we specify assumptions on
Equation \eqref{genElliptic} and the geometry of $S$,
$Q$ guaranteeing that the forward operator
  $\Af(\bs{\mu}):=L(\Phi(\bs{\mu}))$ will meet Conditions \ref{conU},
  when $\Phi(\bs{\mu})$ is the solution of \eqref{genElliptic}
  defined in Theorem \ref{soldim1} and $L$  the
  measurement operator restricting $\Phi$ (or $\nabla\Phi$) to $Q$.
  We address these issues 
   {\it via}  Conditions \ref{conE} below and scenarios
  (S.1), (S.2) further down, that will
  eventually be shown to imply Conditions \ref{conU} using results
  on elliptic PDE  derived in
  Appendices A and B. This  provides us with a framework in which 
  inverse problems with source in divergence form
  associated to more general elliptic operators than
  the Laplacian continue to  have a unique approximate solution under $TV$-regularization, thanks to
  Theorem~\ref{thm|uniqueness}.
  Below, we employ standard  notation and terminology
  regarding Sobolev
  or Hölder spaces and the like, for which we
  refer the reader to Appendix~\ref{A}, Section \ref{prelSobetc}. We 
let  $\mathcal{Z}_{\mathcal{L}}$ denote the subspace of
  $W^{1,2}_{loc}(\R^N\setminus S)$ comprised of weak solutions,
  in the usual sense of \cite[Chapter 8]{GiTru}, to the equation
  $\mathcal{L}(u)=0$ on $\R^N\setminus S$; the space $\mathcal{Z}_{\mathcal{L}}$ is closed in $W^{1,2}_{loc}(\R^N\setminus S)$ and thus a
  Fréchet space in its own right. Theorem \ref{soldim1} shows that the restriction of $\Phi(\bs{\mu})$ to $\R^N\setminus S$ indeed belongs to
  $\mathcal{Z}_L$. Hereafter, we often abuse notation by dropping the restriction symbol ``$|_{\R^N\setminus S}$'' and consider $\Phi(\bs{\mu})$ itself as an element of $\mathcal{Z}_L$; this should cause no confusion as
  $\Phi(\bs{\mu})$ is entirely characterized by its restriction to $\R^N\setminus S$ because, by Theorem \ref{soldim1}, it is a locally integrable function
  (not a distribution of higher order) while  our assumptions
  on $S$ will imply it has  measure zero.
     

\begin{con}{E}\label{conE}
	\color{white}{a}\color{black}
  \begin{itemize}
  \item[(1)] $S\subset \R^N$ is compact and slender.
    \item[(2)] In \eqref{genElliptic} the map $\bs{\Sigma}:\R^N\mapsto\R^N\times\R^N$,
      valued in the positive definite matrices, is uniformly elliptic
      with Dini continuous
    derivatives in a neighborhood $V$ of $S$. 
  \item[(3)]    
    There is a separable 
      Hilbert space $\mathcal{H}$ and a  measurement operator  $L:\mathcal{Z}_{\mathcal{L}}\to \mathcal{H}$ that is continuous, linear and injective, such that the forward operator
      $A:\mathcal{M}(S)^N\to \mathcal{H}$
                 is given by $A(\bs{\mu}):=L(\Phi(\bs{\mu}))$ (with the abuse of notation mentioned above).
  \end{itemize}        
\end{con}

To put Conditions \ref{conE} into perspective,
we shall recast them  within
a more general class of operators satisfying Conditions \ref{conU}, described in Conditions \ref{conK} below  (see Lemma~\ref{LemKU}). This will connect our framework with
classical ``off-the-grid'' approaches to inverse problems
\cite{Poon2018TheGO,LaBlAu}, and stress how the weaker
regularity and the failure of the so-called \emph{source condition} in our case
get compensated, in part, by the structure of our forward operator (it has a kernel which is a Lipschitz gradient).
To state Conditions \ref{conK}, recall from Section \ref{preln} the space $\mathcal{Q}(S)$ of quasi-gradients on $S$ which is the uniform closure
of restrictions to $S$ of the gradient of $C^1$-functions in a neighborhood of $S$.
Clearly $\bs{\mu}\in\mathcal{M}(S)^N$ is such that $\langle \bs{\mu},\bs{\psi}\rangle=0$ for  $\bs{\psi}\in \mathcal{Q}(S)$ if and only if
$\bs{\mu}$ is divergence free.

\begin{con}{K}\label{conK}
	\color{white}{a}\color{black}
	\begin{enumerate}
		\item  $\mathcal{H}$ is a separable Hilbert space, and  $S\subset V\subset\R^\N$ with $S$ compact and $V$ open.
		\item ${G}:V\to \mathcal{H}$ has  Lipschitz continuous
                  gradient\footnote{The derivatives of $\nabla G$ may be understood
                    either in  a weak, vector-valued distributional sense
                    \cite[Chapter 1, \S 5 no. 1] {SchwartzTOD} or in a
                    strong Fréchet sense, because Rademacher theorem extends to Lipschitz functions on $\R^N$ valued in a separable Hilbert space \cite[Chapter II, Theorem 1 \& Chapter III, Theorem 2]{Aronszajn-LH}.} $\nabla{G}$ on $V$.
                \item Defining for every ${h}\in\mathcal{H}$ a
                  function   $g_{{h}}:V\to\R$ by
                  $g_{{h}}(y):=\langle {G}(y),{h}\rangle_{\mathcal{H}}$
                  so that
                  $\partial_{y_j}g_{{h}}(y)=\langle \partial_{y_j}{G}(y),{h}\rangle_{\mathcal{H}}$ for $ 1\leq j\leq N$, we
                  assume that the mapping ${h}\to \nabla g_{{h}}|_S$ has dense range in 	$\mathcal{G}(S)$.
                \item 
                  The operator  $\Af:\mathcal{M}(S)^\N\to  \mathcal{H}$
                  is the integral operator with kernel $- \nabla{G}|_S$, namely:
\begin{equation}
  \label{defA}
  \Af(\bs{\mu}):=
  -\int \nabla{G}(y)\cdot d\bs{\mu}(y)
                \qquad\bs{\mu}\in\mathcal{M}(S)^\N,
		\end{equation}
		where integrals  are taken in the sense of Bochner (see for example \cite[{Chapter~II}]{diestel1977vector}).
              
	\end{enumerate} 
\end{con}



\begin{lemma}\label{LemKU}
	Conditions~\ref{conK} imply Conditions~\ref{conU} for $A$ as in \eqref{defA}.
\end{lemma}
\begin{proof}
  Assume Conditions~\ref{conK}. 
  Since
  linear forms commute with the Bochner integral (a special case of 
  Hille's theorem \cite[{Chapter~II, Thm~6}]{diestel1977vector}), it is
  clear from its kernel structure that $A$ is sequentially weak-star
  to weak continuous. 
In addition, the uniform continuity of
$\nabla G$ on $S$ implies that it is the uniform limit of a sequence of simple functions and therefore, by definition of the Bochner integral
\cite[Chapter II, Section 2]{diestel1977vector},  $A$  is a limit
in operator norm of finite rank operators. Hence $A$ is compact and, since we already know that it is sequentially weak-star to weak continuous, we deduce that
$A$ is sequentially weak-star to strong continuous. Thus, part (1) of Conditions~\ref{conU} holds.
  	
  Since $A$ is {\it a fortiori} strong-to-strong continuous so is
  $A^*:\mathcal{H}\to C(S,\R^N)$, and by Hille's theorem again
  we get for ${h}\in\mathcal{H}$ and $y\in S$ that
          \begin{equation}\label{Astar}
            \Af^*({h})(y) 
		=-\langle{\nabla{G}}(y),{h}\rangle_{\mathcal{H}^N\times\mathcal{H}}=-\grad g_{h}(y),
              \end{equation}
              where we write for short                
$\langle \bs{f},h\rangle_{\mathcal{H}^N\times\mathcal{H}}:=(\langle f_1,h\rangle_\mathcal{H},\ldots, \langle f_N,h\rangle_\mathcal{H})^T $  for $\bs{f}\in\mathcal{H}^N$ and $h\in \mathcal{H}$.
 From assumption (2) of Conditions \ref{conK} we see that $g_h$ lies in $C^1(V)$
 and has Lipschitz-continuous gradient there. Since $g_h$ extends to a 
 $C^1(\R^N)$-function upon multiplying it by a smooth function in $C_c(V)$ which is identically $1$ on a neighborhood of $S$, part
	(3) of Conditions~\ref{conU} is satisfied.
	
	Finally, for any $\bs{\mu}\in\mathcal{M}(S)^N$ and ${h}\in\mathcal{H}$,
	\begin{align*}
		\langle A\bs{\mu}, {h}\rangle_{\mathcal{H}} = \langle  \bs{\mu}, A^*(h)\rangle=-\langle \bs{\mu},\grad g_{{h}}\rangle,
	\end{align*}
	and by the density assumption in item (3) of Conditions~\ref{conK}
        it follows that $A\bs{\mu}=0$ if and only if $\div\bs{\mu}=0$.
	Therefore part (2) of Conditions~\ref{conU} holds.
	\end{proof}

        We now show that Conditions \ref{conE} define a forward operator $A$,
        described below, 
        that is indeed an integral operator meeting Conditions \ref{conK}.

\begin{lemma}\label{LemEK}
  Conditions~\ref{conE} imply Conditions~\ref{conK} with
  ${G}(y):=L(E(\cdot,y))$ for $y\in V$, where $E(x,y)$ is the fundamental solution at $x$ 
  of the differential operator $\mathcal{L}$, defined in \eqref{genElliptic}, with source term $\delta_y$.
\end{lemma}

\begin{proof}
  Assume Conditions~\ref{conE}  hold.   Without loss of generality, we assume that the countably many seminorms $\|\cdot\|_{W^{1,2}(\Omega_n)}$ defining the topology of
  $W^{1,2}_{loc}(\R^N\setminus S)$ are taken over a \emph{nested} family of
  relatively compact open
  subsets $\Omega_n$ exhausting $\R^N\setminus S$. Then, by definition of that topology, the continuity of $L$ implies that $\|L(h)\|_{\mathcal{H}}<1$ as soon as $h\in W^{1,2}_{loc}(\R^N\setminus S)$ satisfies  $\|h\|_{W^{1,2}(\Omega_n)}<\eta$
  for a certain $\eta>0$ and a certain $n\in\mathbb{N}$. Hence, by the homogeneity of $L$, one has $\|L(h)\|_{\mathcal{H}}<\eta^{-1}\|h\|_{W^{1,2}(\Omega_n)}$.
  Applying now Corollary \ref{corLipab} with $O=\Omega_n$ and letting $U\subset V$ be the corresponding open neighborhood of $S$, we deduce using
  the additivity of $L$ that for $j\in\{1,\cdots,N\}$ and $y,y^\prime\in U$ it holds
  \[\|L\left(\partial_{y_j}E(\cdot,y)\right)-L\left(\partial_{y_j}E(\cdot,y^\prime\right))\|_{\mathcal{H}}=\|L\left(\partial_{y_j}E(\cdot,y)-\partial_{y_j}E(\cdot,y^\prime\right))\|_{\mathcal{H}}\leq \frac{K}{\eta}|y-y^\prime|
    \]
    for some constant $K$. Thus,
    the map $y\mapsto (L\partial_{y_1}E(\cdot,y),\cdots,L\partial_{y_N}E(\cdot,y))^t$ is Lipschitz continuous from  $U$ into $\mathcal{H}^N$. In another connection,
     $y\mapsto E(\cdot,y)$ is  also Lipschitz continuous from $U$ into $W^{1,2}(\Omega_n,\R^N)$, by Corollary \ref{corLipab}.
Thus, since $W^{1,2}(\Omega_n,\R^N)$ is a separable Hilbert space, it follows from  \cite[Chapter III, Theorem 2]{Aronszajn-LH} (in the notation of which
we take $K$ to be the identity operator on $\R^N$) that  $y\mapsto E(\cdot,y)$
is Fréchet differentiable at a.e. $y\in U$. In particular, if $\bb{e}_j$ is the $j$-th canonical unit vector in $\R^N$,
    the limit when $\R\ni \rho\to0$ of the difference quotient
    $\frac{E(\cdot,y+\rho \bb{e}_j)-E(\cdot,y)}{\rho}$ exists in $W^{1,2}(\Omega_n)$ for
    $j\in\{1,\cdots,N\}$, and arguing as in \eqref{derquotf} we find that
    this limit must be
    $\partial_{y_j}E(\cdot,y)$. That is to say, the Fréchet derivative applied to
    $\bs{\rho}:=(\rho_1,\cdots,\rho_N)^t\in\R^N$ is
    \begin{equation}
      \label{expFrd}
      DE(\cdot,y)\cdot\bs{\rho}=\sum_{j=1}^N\partial_{y_j}E(\cdot,y)\rho_j,\qquad a.e.\  y\in U.
    \end{equation}

    Moreover, as the restriction of $y\mapsto E(\cdot,y)$ to
    every  line $\{y_0+t \bb{e}_j,\,t\in\R\}$ is Lipschitz continuous, 
    it holds for $|\rho|$ small enough  the Bochner integral representation in $W^{1,2}(\Omega_n)$:
    \begin{equation}
      \label{Bochner1D}
E(\cdot,y+\rho \bb{e}_j)-E(\cdot,y)=\int_0^\rho\partial_{y_j}E(\cdot,y+t \bb{e}_j)\,dt
\end{equation}
(see \cite[Chapter II, Section 2, proof of Lemma 1]{Aronszajn-LH}). Because
$y\mapsto \partial_{y_j}E(\cdot,y)$ is continuous (even Lipschitz) on $U$ by Corollary \ref{corLipab}, we deduce from \eqref{Bochner1D} that
the partial derivatives $\partial_{y_j}E(\cdot,y)$ continuously exist at \emph{every} $y\in U$, that is to say $y\mapsto E(\cdot,y)$ is a $C^1$ map
from $U$ into $W^{1,2}(\Omega_n)$, with Fréchet derivative given by \eqref{expFrd} where
``a.e.'' is omitted. Since $n$ can be any sufficiently large integer
the limit $\frac{E(\cdot,y+\rho \bb{e}_j)-E(\cdot,y)}{\rho}\to \partial_{y_j}E(\cdot,y)$ when
$\rho\to0$ in fact holds in $W^{1,2}_{loc}(\R^N\setminus S)$,
and by the linearity and  continuity of $L:W^{1,2}_{loc}(\R^N\setminus S)\to \mathcal{H}$, we get that $y\mapsto L(E(\cdot,y))$ is in turn continuously  Fréchet differentiable from $U$ into $\mathcal{H}$, with
\begin{align}
  \nonumber\partial_{y_j}L(E(\cdot,y))&=\lim_{\rho\to0}\frac{L(E(\cdot,y+\rho \bb{e}_j))-L(E(\cdot,y))}{\rho}\\
  \label{Lcommder}
  &=L\left(\lim_{\rho\to0}\frac{E(\cdot,y+\rho \bb{e}_j)-E(\cdot,y)}{\rho}  \right)=L(\partial_{y_j}E(\cdot,y)),\qquad y\in U.
  \end{align}
        
  This being said, Theorem \ref{soldim1}
  in Appendix A shows that the solution to \eqref{genElliptic} can be expressed as a Bochner integral in terms of the fundamental solution $E$ of $\mathcal{L}$ as
$$\Phi(\bs{\mu})(x)=-\int_S \nabla_y E(x,y) \cdot d\bs{\mu}(y).$$
Consequently, the forward operator $\Af:\mathcal{M}(S)^\N\to  \mathcal{H}$
specified in item (3) of conditions \ref{conE} can be rewritten as

\begin{align*}\Af(\bs{\mu})&:= L(\Phi(\bs{\mu}))=-L\int_S\nabla_y E(\cdot,y) \cdot d\bs{\mu}(y)\\&=-\sum_{j=1}^N  \int_SL\partial_{y_j}E(\cdot,y)  \, d {\mu_j}(y)=-\int_S \nabla_y\left(  LE(\cdot,y) \right)\cdot d\bs{\mu}(y)\\&=-\int  \nabla {G}(y)\cdot d\bs{\mu}(y),
\end{align*}
where the commutation between $L$ and the integral is justified by Hille's theorem,  and its commutation with differentiation follows from \eqref{Lcommder}. We gather from what precedes that
\[
  \nabla G(y)=\nabla(LE(\cdot,y))=\left(L\partial_{y_1}E(\cdot,y),\cdots, L\partial_{y_N}E(\cdot,y)\right)^t\quad \mathrm{is\ Lipschitz\ }U\to\mathcal{H}^N,
\]
thereby accounting for item (2) of Conditions \ref{conK}.
Furthermore, taking the $V$ of Conditions \ref{conK} to be our $U$, the map $g_h:U\to\R$
defined in Conditions \ref{conK} is given for $h\in\mathcal{H}$ by
$g_h(y)=\langle L(E(\cdot,y)),h\rangle_{\mathcal{H}}$. \emph{We claim} that
the space $\{\nabla g_h|S,\,h\in\mathcal{H}\}$
is  dense in $\mathcal{Q}(S)$. Otherwise indeed, by the Hahn-Banach theorem, there is a
$\R^N$-valued measure $\bs{\mu}\in\mathcal{M}(S)^N$ which is not divergence free
(for it does not define the zero form on $\mathcal{Q}(S)$, see Appendix \ref{Rgrad}) such that, for all $h\in\mathcal{H}$,
\begin{align*}0=\int\nabla_y\langle L(E(\cdot,y)),h\rangle_{\mathcal{H}}\cdot d\bs{\mu}(y)=
\sum_{j=1}^N \int \langle L(\partial_{y_j}E(\cdot,y)),h\rangle_{\mathcal{H}} d\mu_j(y)\\
=\left \langle\sum_{j=1}^N\int L(\partial_{y_j}E(\cdot,y)) d\mu_j(y),h\right\rangle_{\mathcal{H}}
=\left \langle L\left(\sum_{j=1}^N\int \partial_{y_j}E(\cdot,y)) d\mu_j(y)\right),h\right\rangle_{\mathcal{H}}.
\end{align*}
Consequently the left argument of the last scalar product in $\mathcal{H}$ must be zero, and since $L$ is injective by item (3) of Conditions \ref{conE}
we deduce that
\[
0=\sum_{j=1}^N\int \partial_{y_j}E(\cdot,y)) d\mu_j(y)=\int\nabla_yE(\cdot,y)\cdot d\bs{\mu}.
\]
In other words, the solution  $\Phi(\bs{\mu})$ to \eqref{genElliptic} (characterized by
Equation \eqref{vwsRv1p}) is identically zero.
However, as $\Phi(\bs{\mu})$ is a \emph{locally integrable function} by Theorem \ref{soldim1} and  since $S$ has zero Lebesgue measure because it is slender by item (1) of Conditions \ref{conE}, those $\bs{\mu}\in \mathcal{M}(S)^N$ such that $\Phi(\bs{\mu})=0$ a.e. on $\R^N$ are  precisely the divergence-free measures
supported on $S$ (in the distributional sense),
by Corollary \ref{corDinib}. This contradiction \emph{proves the claim}, and so
Conditions \ref{conK} are completely satisfied.
\end{proof}

To  make use of   Conditions \ref{conE}  in situations we have in mind, it remains for us to construct appropriate $\mathcal{H}$ and $L:\mathcal{Z}\to \mathcal{H}$
that satisfy the requirements of item (3).
In our application to inverse magnetization problems, the measurement operator
$L$  acts by restricting the gradient (or a directional derivative) of $\Phi$ to a compact set $Q\subset \R^N\setminus S$,
but since we now deal with more general elliptic operators $\mathcal{L}$ than
the Laplacian (possibly with nonsmooth coefficients),
care must be used to make sense of that restriction and ensure 
that $L$ is injective.
Studying  in depth uniqueness sets for (very weak) solutions to \eqref{genElliptic}
is beyond the scope of this paper,
but we shall at least discuss two scenarios, dwelling on known unique
continuation results, each of which will require extra assumptions on $\bs{\Sigma}$ and $Q$.
The first one allows for the conductivity $\bs{\Sigma}$  to have discontinuities along smooth interfaces, and it implies that one measures all components of the field $\nabla\Phi$ on $Q$;
this  is adapted to the setting of inverse EEG problems (see for instance  \cite{DLMT}) and  other inverse problems dealing with heterogeneous media.  The second scenario assumes that
$\bs{\Sigma}$ is analytic, in which case it is enough to measure one component
of $\nabla\Phi$ on $Q$ if the latter is an  analytic set; this can accommodate the needs of inverse magnetization problems from SQUID  or QDM  microscopy that deal with the Laplacian
(see for instance \cite{LiWe,FuLiVoTru}).
To set up the first scenario, we say that $\bs{\Sigma}$
is \emph{piecewise Lipschitz} on
an open set $\mathcal{U}$ if there are disjoint domains $\{\mathcal{O}_j\}_{1\leq j\leq l}$ with compact boundary  in $\R^N$ such that
\begin{itemize}
	\item $\mathcal{U}=(\cup_j\overline{\mathcal{O}}_j)\setminus \partial\mathcal{U}$,
	\item if a point  $x\in \mathcal{U}$ lies on some  $\partial\mathcal{O}_i$
then the entire connected component of $\partial\mathcal{O}_i$ containing $x$
is locally the graph of a $C^{1,\alpha}$-function (thus any point in $\mathcal{U}$ belongs to at most two $\partial\mathcal{O}_j$, compare \cite{LiNir}),
	\item $\bs{\Sigma}$ is  Lipschitz-continuous on  $\mathcal{O}_j\cap \mathcal{U}_\varepsilon$ for all $j$ and every $\varepsilon>0$, where $\mathcal{U}_\varepsilon\subset\mathcal{U}$ designates the set of points at distance strictly greater
than $\varepsilon$ from $\partial\mathcal{U}$.
\end{itemize}
The $\mathcal{O}_j$ are the \emph{pieces} of the Lipschitz decomposition.
The two scenarios are as follows.
\color{white}{a}\color{black}
\begin{itemize}
\item[(S.1)]
  $\bs{\Sigma}$ is piecewise Lipschitz  on $\R^N\setminus S$ with Dini-continuous derivatives in a neighborhood of $S$\footnote{Equivalently:  $\bs{\Sigma}$ is piecewise Lipschitz  on $\R^N\setminus S$, each component of $S$ is included in a piece, and in these pieces $\bs{\Sigma}$ has Dini-continuous derivatives.}. Besides $Q$ is a closed subset of
  $\R^N\setminus S$ and either (i) $Q$
  is the closure of a nonempty  open set, or (ii) $Q$ is the closure of a nonempty open subset of   a $C^{1,\alpha}$-hypersurface for some $\alpha\in(0,1)$.
  The measurement operator $L$ maps $\Phi$
  to $\Phi_{|Q}$ or 
  $\nabla\Phi_{|Q}$ in case (i), and in case (ii) it maps $\Phi$ to
  $\Phi_{|Q}$  (or the tangential gradient of $\Phi$ on $Q$) together with
  the normal derivative $\partial\Phi/\partial n^\pm$ from either side of $Q$;
   here, $n^+$ (resp. $n^-$) indicates the unit normal
   to  $Q$  pointing in the positive (resp. negative) direction for some choice of orientation, and 
   the normal derivatives from each side may be different if $Q$ has non-trivial intersection with some $\partial\mathcal{O}_i$.
   Clearly
   $L:\mathcal{Z}_{\mathcal{L}}\to\mathcal{H}$  is continuous
    with  $\mathcal{H}:=L^2(Q)$ or $L^2(Q,\R^N)$ in case (i), while
      $\mathcal{H}:=L^2(Q)\times L^2(Q)$
      or $L^2(Q,\R^N)$  in case (ii); this
      follows from \cite[Theorem 1.1]{LiVog} which  implies
      that $\nabla \Phi$
      extends Hölder-continuously to each $\overline{\mathcal{O}}_j\cap (\R^N\setminus S)_\varepsilon$ for  any $\varepsilon>0$.       
      The injectivity of $L$ is proven in  Lemma \ref{S1}.

     \item[(S.2)] $Q$ is the closure of an open subset of a closed, nonsingular
       real analytic
       hypersurface $\mathcal{A}$ that does not meet $S$,  while $\bs{\Sigma}:\R^N\to\R^{N\times N}$ is analytic
        and $L$ maps $\Phi$ to the restriction  $\Phi_{|Q}$ or to
       $\partial\Phi/\partial n^{\pm}_{|Q}$ ($+$ or $-$) with $\mathcal{H}:=L^2(Q)$.
       This framework
suits in particular the inverse magnetization problem \eqref{PhiDef}.
\end{itemize}

The assumptions of the scenarios could be improved a little.
For instance, Lipschitzianity could be traded for
local Lipschitzianity plus Dini-continuity of  $\bs{\Sigma}$ on
$\partial\mathcal{O}_j\cap \mathcal{U}_\varepsilon$ in Scenario (S.1),
and  global analyticity could be traded for Lipschtzianity plus
analyticity about  $\mathcal{A}$  in Scenario (S.2).
No such hypotheses would be sharp anyway,
and we do not strive for  generality here.

\begin{appendices}

\section{Elliptic equations  with source  the divergence of a measure.}
\label{A}
We will study existence and uniqueness 
of a very weak solutions to the  elliptic PDE
\begin{equation}
  \label{eqelf}
  \mathcal{L}\Phi:=\div(\bs{\Sigma}\nabla \Phi)
  =\mathrm{div}\,\bs{\mu}\qquad \mathrm{on}\ \R^N,\qquad \lim_{|\bs{x}|\to\infty}\Phi(\bs{x})=0,
  \end{equation}
with $\bs{\mu}$ a compactly supported $\R^N$-valued measure
and $\bs\Sigma$ a real symmetric, uniformly elliptic matrix-valued function
which is Dini-continuous in a neighborhood of $\supp \bs{\mu}$.
The 
Dini-continuity assumption saves us
from resorting to even weaker types of solutions,
and is sufficient for us as
we shall make  stronger hypotheses anyway
to fulfill Conditions E in Section \ref{Sect2.1}. 
Equation  \eqref{eqelf} is standard when $\bs{\mu}$ is an $L^2$-function,  and
is also well-studied if $\mathrm{div}\,\bs{\mu}$ gets replaced by a
(scalar-valued) measure, but the case where the right hand side is the divergence of a measure was apparently not much studied. This is the \emph{raison d'être}
for the present Appendix. We base our analysis on Green functions
which is a classical tool, but we need to set out some of their properties
in the case of $\R^N$ which is considered a bit marginally in the literature.
We assume throughout that $N\geq3$.
\subsection{Background on Sobolev spaces}
\label{prelSobetc}
Hereafter for convenience,  $E$ denotes a measurable set and $\Omega$ an open set in $\R^N$. Pieces of notation regarding measures,
continuous and smooth functions were made in 
Section \ref{preln}.
 We let $m$ denote Lebesgue measure on $\R^N$ no matter $N$, and
  $L^p(E,\R^m)$ be the Lebesgue space  of (equivalence classes
of a.e. coinciding) functions $f:E\to\R^m$ such that  $\|f\|_{L^p(E,\R^m)}^p:=\int_E|f|^pdm<\infty$ if $p\in[1,\infty)$, or  $\|f\|_{L^\infty(E,\R^m)}:=\mathrm{ess. sup}_E |f|$. The space
  $ L_{loc}^p(\Omega,\R^m)$ consists of functions whose restriction to each
  compact  $K\subset \Omega$ lies in
  $L^p(K,\R^m)$, endowed with the projective topology of the restriction maps
  $L^p_{loc}(\Omega,\R^m)\to L^p(K,\R^m)$.
In the previous notation, we  omit $\R^m$ if $m=1$. 

When $C_c^\infty(\Omega)$ is endowed with the usual
topology\footnote{{\it i.e.}, the inductive topology of the subspaces  
$\mathcal{D}_K$ comprised of functions  supported on a compact set 
$K\subset\Omega$, each $\mathcal{D}_K$ being topologized by uniform
convergence of 
all derivatives, see  \cite[Ch. I, Sec.2]{Schwartz}.},
its dual $\mathcal{D}'(\Omega)$ 
is the space of distributions on $\Omega$. 
We let $W^{1,p}(\Omega)$  be the Sobolev space of functions that belong to  $L^p(\Omega)$ as well as their first distributional derivatives. Letting
$\partial_j$ indicate the derivative with respect to the $j$-th variable,
it is a Banach space with norm
$\|f\|^p_{W^{1,p}(\Omega)}:=\|f\|^p_{L^p(\Omega)}+\sum_{j=1}^N\|\partial_j f\|^p_{L^p(\Omega)}$ for $ 1\leq p<\infty$,  and $\|f\|_{W^{1,\infty}(\Omega)}:=\max\{\|f\|_{L^\infty(\Omega)},\max_{1\leq j\leq N}\|\partial_j f\|_{L^\infty(\Omega)}\}$.
One can show that $W^{1,\infty}(\Omega)=\mathrm{BLip}(\Omega)$ with equivalent norms
\cite[Chapter V, Section 6.2]{Stein}.
The closure of $C^\infty_c(\Omega)$ in $W^{1,p}(\Omega)$ is denoted by
$W^{1,p}_0(\Omega)$, and $W^{1,p}_0(\R^N)=W^{1,p}(\R^N)$ for $1\leq p<\infty$.
The space 
$ W_{loc}^{1,p}(\Omega)$ consists of functions whose restriction to each relatively compact open subset $V$ of $\Omega$ lies in
$W^{1,p}(V)$, equipped with the seminorms $\|\cdot\|_{W^{1,p}(\Omega_k)}$
where $\Omega_k$ is a sequence of such open subsets exhausting $\Omega$. 


If $\Omega$ is Lipschitz then   $W^{1,p}(\Omega)$ coincides
with  restrictions to $\Omega$ of functions in $W^{1,p}(\R^n)$, and the restriction map  $W^{1,p}(\R^N)\to W^{1,p}(\Omega)$ has a  continuous right inverse $W^{1,p}(\Omega)\to W^{1,p}(\R^N)$, called an \emph{extension operator}
\cite[Ch. VI, Thm 5]{Stein}.
The boundary $\partial\Omega$ is a Lipschitz Riemannian manifold under the Euclidean scalar product, whose  volume measure $\sigma$ coincides with
$(N-1)$-Hausdorff measure restricted to $\partial\Omega$ \cite[Section 5.8.3]{Ziemer}.
Each $f$ in $W^{1,p}(\Omega)$, $1<p\leq\infty$,
has a trace on $\partial\Omega$ lying in the fractional 
Sobolev space $W^{1-1/p,p}(\partial\Omega)$ where, for $0<\theta<1$ and $\psi$ a
$\sigma$-measurable function on
$\partial\Omega$, we put
\begin{equation}
  \label{defBeN}
    \|\psi\|_{W^{\theta,p}(\partial\Omega)}:=
\|\psi\|_{L^p(\partial\Omega)}+\left(\int_{\partial\Omega}\int_{\partial\Omega}
                                              \frac{|\psi(x)-\psi(y)|^p}{|x-y|^{N-1+\theta p}}d\sigma(x)d\sigma(y)\right)^{1/p}.
\end{equation}
More precisely the trace operator, initially defined on
$C_c^\infty(\R^n){|_\Omega}$, extends to a continuous map
$W^{1,p}(\Omega)\to W^{1-1/p,p}(\partial\Omega)$. In fact, membership in $W^{1-1/p,p}(\partial\Omega)$
characterizes traces of $W^{1,p}(\Omega)$-functions on the boundary of a
Lipschitz domain,
since the trace operator
has a continuous right inverse
$W^{1-1/p,p}(\partial\Omega)\to W^{1,p}(\Omega)$ \cite[Thm. 1.5.1.3]{Grisvard}
\cite[Theorem 3.1]{Mattila}.
The space $W^{1,p}_0(\Omega)$ is comprised of $W^{1,p}(\Omega)$-functions having zero
trace   \cite[Thm. 1.5.1.4]{Grisvard}. Traces of
$W^{1,1}(\Omega)$-functions  exist in $L^1(\partial\Omega,\sigma)$.

The Sobolev space
$W^{k,p}(\Omega)$   for $1<k\in\mathbb{N}$ consists of
functions whose first order derivatives belong to $W^{k-1,p}(\Omega)$
with norm $\|f\|^p_{W^{k,p}(\Omega)}=\|f\|^p_{L^p(\Omega)}+\sum_{j=1}^N\|\partial_j f\|^p_{W^{k-1,p}(\Omega)}$ when $p<\infty$, and the usual variant when $p=\infty$.
The definition of  $W^{k,p}_{loc}(\Omega)$ runs parallel to the one of $W^{1,p}_{loc}(\Omega)$ and, though we are mainly concerned with $k=1$, we shall have an occasion to deal with $k=2$.

If $0>k\in\Z$ and $p>1$, the space $W^{k,p}(\Omega)$ is the dual of  $W_0^{-k,p^\prime}(\Omega)$, $\frac{1}{p}+\frac{1}{p^\prime}=1$.
This space of distributions  may be realized as the completion of $L^p(\Omega)$  for the norm 
\begin{equation}
  \label{defdS}
  \|g\|_{W^{k,p}(\Omega)}:=\sup_{f\in W_0^{-k,p^\prime}(\Omega),\|f\|_{W^{-k,p^\prime}(\Omega)=1}}\left|\int fg\,dm\right|,
\end{equation}
see \cite[Theorem 3.13]{AdamsFournier}.
In a similar vein, the space $W^{-1/2,2}(\partial\Omega)$ is the dual of $W^{1/2,2}(\partial\Omega)$,
and can be realized as the completion of $L^2(\partial\Omega)$ for a norm analogous to \eqref{defdS} (with $p=2$) except that the $\sup$ is taken over functions of unit
norm in $W^{1/2,2}(\partial\Omega)$ (see \eqref{defBeN}) and the integral is computed against $d\sigma$.

A distribution on $\Omega$ whose first derivatives 
belong to $L^p(\Omega)$ necessarily lies in  $W^{1,p}_{loc}(\Omega)$
\cite[Sec. 1.1.2]{Mazya}.
The quotient space modulo constants
of  such distributions is the \emph{homogeneous Sobolev space} of
index $p$ on $\Omega$, denoted by $\dot{W}^{1,p}(\Omega)$. If $\dot{g}$ denotes
the equivalence class of a distribution $g$, then
the norm
$\|\dot{g}\|_{\dot{W}^{1,p}(\Omega)}:=\|\nabla g\|_{L^p(\Omega,\R^N)}$
turns $\dot{W}^{1,p}(\Omega)$ into a Banach space \cite[Sec. 1.1.13, Thm. 1]{Mazya}.
 Clearly $C_c^\infty(\Omega)$
 injects into $\dot{W}^{1,p}(\Omega)$, and we let
$\dot{W}^{1,p}_0(\Omega)$ indicate its closure. 
When $1\leq p<N$, each member of  $\dot{W}^{1,p}_0(\Omega)$ has a representative in
$L^{p^*}(\Omega)$ where $p^*:=Np/(N-p)$. This representative is unique
if $m(\Omega)=\infty$, and the natural map
$\dot{W}^{1,p}_0(\Omega)\to L^{p^*}(\Omega)$ is continuous; to see
this, it is enough to consider  $\Omega=\R^N$ in which case the result
follows from
\cite[Corollary 1.77 \& Theorem 1.78]{MalZie}. So, if for $1\leq p<\infty$
we let $Y^{1,p}(\Omega)$
denote the space of those $g\in L^{p^*}(\Omega)$ such that
$\nabla g\in L^p(\Omega,\R^N)$ endowed with the norm $\|g\|_{Y^{1,p}(\Omega)}:=\|\nabla g\|_{L^p(\Omega,\R^N)}$,
 and   if we put $Y_0^{1,p}(\Omega)$ for the closure of $C_c^\infty(\Omega)$ in
 $Y^{1,p}(\Omega)$, we can   identify $\dot{W}^{1,p}_0(\Omega)$ with
 $Y^{1,p}_0(\Omega)$ (which is a space of functions).  Note  that
 $W^{1,p}_0(\Omega)\subset Y^{1,p}_0(\Omega)$ by Sobolev's inequality, and observe in addition that
 $Y^{1,p}(\R^N)=Y^{1,p}_0(\R^N)\sim \dot{W}^{1,p}_0(\R^N)=\dot{W}^{1,p}(\R^N)$
 \cite[Lemma 1.76 \& Theorem 1.78]{MalZie}.

 We shall be concerned with the case $p=2<N$ only.  Clearly, $Y^{1,2}(\Omega)$ is a Hilbert space with scalar product $\langle\phi,\psi\rangle_{Y^{1,2}(\Omega)}=\int_\Omega\nabla\phi\cdot\nabla\psi dm$.
 The dual space $(Y^{1,2}(\Omega))^*$ is thus isometric to $Y^{1,2}(\Omega)$, but
 its members may also be described (non isometrically) as distributions of the form $\mathrm{div}\,f$ with $f\in L^2(\Omega,\R^N)$. In fact, if $\mathcal{D}^2(\Omega)$
 denotes the divergence-free fields in $L^2(\Omega,\R^N)$ and 
 $\mathcal{D}_0^2(\Omega)$ indicates the closure of $\mathcal{D}^2(\Omega)\cap C_c^\infty(\Omega,\R^N)$ in that space, then
the norm of $\mathrm{div}\,f\in (Y^{1,2}(\Omega))^*$
is the distance from $f$ to  $\mathcal{D}_0^2(\Omega)$ in $L^2(\Omega,\R^N)$;
see \cite[Lemma 2]{BLN}.
Likewise  the Hilbert space $Y^{1,2}_0(\Omega)$
is its own dual,  but the latter also identifies with
 distributions $\mathrm{div}\,f$, $f\in L^2(\Omega,\R^N)$,
 endowed with the quotient norm  in $L^2(\Omega,\R^N)/\mathcal{D}^2(\Omega)$.


   
\subsection{Green functions} 
Consider a partial differential operator $\mathcal{L}$ on an open set
$\Omega\subset\R^N$, acting on a function $u:\Omega\to\R$ by the formula
\begin{equation}\label{genEllop}
  \mathcal{L}u:=-\div(\bs{\Sigma}\nabla u),
\end{equation}
where ``$\div$'' stands for ``divergence''; definition \eqref{genEllop} means that if $u\in W^{1,1}_{loc}(\Omega)$ then
 $\mathcal{L}u$ is the distribution acting on $\varphi\in C_c^\infty(\Omega)$ by
 $ \langle \mathcal{L}u,\varphi\rangle=\int(\nabla u)^t\bs{\Sigma}\nabla\varphi\,dm$, and we assume throughout  that $\bs{\Sigma}:\Omega\mapsto\R^N\times\R^N$ is valued in the symmetric positive definite real matrices satisfying, for fixed constants $c$, $C$, the uniform ellipticity condition:
\begin{equation}
  \label{sell}
  0<c \,\mathrm{I}_N\leq\bs{\Sigma}\leq C\,\mathrm{I}_N
\end{equation}
where $\mathrm{I}_N$ stands for the identity matrix of size $N$.
 For $T$ a distribution on $\Omega$, a
 \emph{weak solution} to $\mathcal{L}u=T$ is
a $u\in W^{1,1}_{loc}(\Omega)$ such that
\begin{equation}
  \label{wes1}
  \int (\nabla u)^t\bs{\Sigma}\nabla \varphi\,dm=\langle T,\varphi\rangle,\quad \forall \varphi\in C_c^\infty(\Omega).
\end{equation}
When $T\in (Y^{1,2}_0(\Omega))^*$, Equation \eqref{wes1} has a unique weak solution in $Y^{1,2}_0(\Omega)$  by \eqref{sell}
and the Lax-Milgram theorem, together with
 the density of $C_c^\infty(\Omega)$ in $Y^{1,2}_0(\Omega)$. We denote this solution  by
$G_\Omega(T)$ and  call $G_\Omega:(Y^{1,2}_0(\Omega))^*\to Y^{1,2}_0(\Omega)$  the \emph{Green operator} for $\Omega$. 
The Fredholm theory implies that   for each $g\in L^2(\Omega)$, $f\in L^2(\Omega,\R^N)$ and
$h\in W^{1,2}(\Omega)$, there is a unique weak solution $u\in W^{1,2}(\Omega)$
to $\mathcal{L}u=g+\mathrm{div}f$ such that  $u-h\in W^{1,2}_0(\Omega)$
\cite[Theorem 8.3]{GiTru}. In particular, when $g=h=0$ we get a  weak solution
to $\mathcal{L}u=\mathrm{div}f$ lying 
in $W^{1,2}_0(\Omega)$, and since the latter is included in $Y^{1,2}_0(\Omega)$
this solution is but $G(\mathrm{div}f)$.
Thus, our definition of the Green operator is  consistent with
\cite[Section 8.2]{GiTru} (we will only apply it to distributions of the form
$\mathrm{div}f$ with  $f\in L^2(\Omega,\R^N)$).

A famous result by De Giorgi- Nash asserts that
a weak solution $u\in W^{1,2}_{loc}(\Omega)$ to $\mathcal{L}u=0$ is locally H\"older-continuous \cite[Theorem 8.24]{GiTru}. The  assumption $u\in W^{1,2}_{loc}(\Omega)$ is important, for
other weak solutions may exists in $W^{1,p}_{loc}(\Omega)$, $p<2$,
that are not even locally bounded \cite{Serrin}.

The open set $\Omega$ is said to satisfy an \emph{exterior cone condition} at
$x\in\partial\Omega$ if there is a closed right circular truncated
cone $V_x$ with vertex at $x$ such that $\overline{\Omega}\cap V_x=\{x\}$;  the condition is said to be \emph{uniform} if it holds at
every  $x\in\partial\Omega$ and all the $V_x$ are congruent to a
fixed cone $V$. Clearly,  Lipschitz open sets have this property.
Assume for now that $\Omega$ is bounded and satisfies  a uniform exterior cone condition.
Each $\psi\in C(\overline{\Omega})$ defines a member of
$(Y^{1,2}_0(\Omega))^*$ by $\phi\mapsto\int\psi\phi dm$.
Hence,  $G_\Omega(\psi)$ is well-defined
and it follows from \cite[Theorem 8.29]{GiTru}
    that it extends  H\"older-continuously to
    $\overline{\Omega}$. Subsequently, 
for $\nu$ a signed measure in $\mathcal{M}(\Omega)$,
we say that  $u\in L^1(\Omega)$ 
is a \emph{very weak} solution vanishing on $\partial\Omega$
to the equation $\mathcal{L}u=\nu$ if
\begin{equation}
  \label{vwes1}
\int_\Omega u\psi\,dm=\int  G_\Omega(\psi)d\nu,\quad \psi\in C(\overline{\Omega}).
\end{equation}
Since $G_\Omega(\psi)$  is continuous on $\overline{\Omega}$,
      the right hand side of  \eqref{vwes1} indeed makes sense; {\it compare} 
    \cite[Sections 5 \&  10]{LSW}\footnote{In that reference, weak solutions are called (local) solutions and very weak ones are called  weak solutions.}.
    A very weak solution  lies in $W^{1,2}_0(\Omega)$ if and only if $\nu\in W^{-1,2}(\Omega)$, in which case it is also  a weak solution 
\cite[Remark 5.1 \& Theorem 5.2]{LSW}.

    To each
    $y\in\Omega$
    there is a unique very weak solution 
    to $\mathcal{L}u=\delta_y$, where $\delta_y$ is the Dirac mass at $y$; see  \cite[Theorem 5.1]{LSW}.
    We denote  it with $E_\Omega(\cdot,y)$ and call it the
    \emph{Green function}, or  \emph{fundamental solution} of $\mathcal{L}$
    in $\Omega$ with pole at $y$. From
    \cite[Theorem 5.1]{LSW} we know that $E_\Omega(\cdot,y)\in W^{1,q}_0(\Omega)$ for $1\leq q<\frac{N}{N-1}$, and the proof shows that
    $y\mapsto E_\Omega(\cdot,y)$ is continuous $\Omega\to W^{1,q}_0(\Omega)$; see
    \cite[Equation (5.2) {\it ff.}]{LSW}.
One has  $E_\Omega(x,y)\geq0$ and  $E_\Omega(x,y)=E_\Omega(y,x)$  for $x,y\in\Omega$,
see \cite[Section 6]{LSW}.

By  \cite[Theorem 6.1]{LSW}, for each $\nu\in \mathcal{M}(\Omega)$ there is a unique very weak solution vanishing on $\partial\Omega$ to  $\mathcal{L}u=\nu$, it lies in $W^{1,q}_0(\Omega)$ for $1\leq q<\frac{N}{N-1}$ and  it can be expressed as
    \begin{equation}
      \label{defvwm}
u(x)=\int E_\Omega(x,y)d\nu(y),\qquad \mathrm{a.e.}\ x\in\R^N.
\end{equation}
From  \eqref{defvwm} we see that  $E_\Omega(x,y)$ is the same as the fundamental solution constructed in \cite{GruWid}, for   it must be the Schwartz kernel of the map
$\varphi\mapsto \,G_\Omega(\varphi)$ from $C_c^\infty(\Omega)$ into $\mathcal{D}^\prime(\Omega)$. Thus, $E_\Omega(\cdot,y)$ is not just a very weak solution to
 $\mathcal{L}u=\delta_y$ vanishing on $\partial\Omega$, but actually a weak solution in $W^{1,p}_0(\Omega)$
for $p\in[1,\frac{N}{N-1})$, see \cite[Theorem (1.1)]{GruWid}.

In \cite[Theorem 6.1]{LSW}, formula \eqref{defvwm} is understood
pointwise and the right hand side  proven convergent for a.e.
$x\in\Omega$, but as we now show  it may also be regarded as the
Bochner integral,
       with respect to $\nu$, of the continuous map  $y\mapsto E_\Omega(\cdot,y)$,
       valued in $W^{1,q}_0(\Omega)$ for $q\in[1,\frac{N}{N-1})$; see
       {\it e.g.} \cite[Chapter II, Section 2]{diestel1977vector}
       for a definition of the Bochner integral as
       a strong limit of integrals of  Banach-valued simple functions
       converging pointwise to the integrand.  Indeed, when $\supp\nu\subset\Omega$ is compact, reference \cite[Chapter 3, Theorem 27 \& Ex. 23]{Rudin} implies that
       the Bochner integral $I:=\int E_\Omega(\cdot,y)d\nu(y)$ is
       characterized by the identity $x^*(I)=\int x^*(E_\Omega(\cdot,y))d\nu(y)$ for every $x^*\in (W^{1,q}_0(\Omega))^*$.  Hence,   
      the realization of the dual space of $W^{1,q}_0(\Omega))$ given in
      \eqref{defdS} along with the Fubini and the dominated convergence theorems yield that \eqref{defvwm} is a Bochner integral;
        when  $\nu$ is merely 
      carried by $\Omega$ but not supported there, it can be approximated
      by a sequence of restrictions $\nu\mathcal{b}\Omega_n$ where $\overline{\Omega}_n\subset\Omega$  increases to $\Omega$, and the corresponding 
      Bochner integrals form a Cauchy sequence by \cite[Theorem 3.29]{Rudin},
      thereby showing that \eqref{defvwm} is a Bochner integral in all cases.
      

      Our first lemma deals with the $W^{1,2}$-continuity of $y\mapsto E_\Omega(\cdot,y)$ away from the pole, a result that we could not reference in the literature.
      \begin{lemma}
        \label{reg2locb}
       For $\Omega$ a bounded Lipschitz domain, $\mathcal{L}$ an  elliptic operator given by  \eqref{genEllop} where \eqref{sell} holds, $B$ an open ball with
        $\overline{B}\subset\Omega$ and $z\in B$, the  restriction
    $E_\Omega(\cdot,z)_{|\Omega\setminus\overline{B}}$,  of its Green function to
    $\Omega\setminus\overline{B}$ is
    a weak solution to $\mathcal{L}u=0$
    and
    $\|E_\Omega(\cdot,z)_{|\Omega\setminus\overline{B}}\|_{W^{1,2}(\Omega\setminus\overline{B})}$ is bounded  solely in terms of $\Omega$, $N$, the constants in \eqref{sell}, the radius of $B$ and $\mathrm{dist}(z,\partial B)$ as well as  $\mathrm{dist}(\partial\Omega,\partial B)$;
     the map
    $z\mapsto E_\Omega(\cdot,z)_{|\Omega\setminus\overline{B}}$ is continuous
    $B \to W^{1,2}(\Omega\setminus\overline{B})$.
  \end{lemma}
  \begin{proof}
    That $E_\Omega(\cdot,z)_{|\Omega\setminus\overline{B}}$  lies in
    $W^{1,2}(\Omega\setminus\overline{B})$ follows from \cite[Theorem (1.1)]{GruWid}. Since $E_\Omega(\cdot,y)$ is
    a weak solution to  $\mathcal{L}u=\delta_y$  in  $W_0^{1,p}(\Omega)$
    for $p\in[1,\frac{N}{N-1})$, as pointed out earlier, $E_\Omega(\cdot,z)_{|\Omega\setminus\overline{B}}$  is 
    a weak solution to $\mathcal{L}u=0$ in $W^{1,2}(\Omega\setminus\overline{B})$       when $z\in B$. Putting for simplicity $B_1:=B(z,\mathrm{dist}(z,\partial B)/2)$, we know from \cite[Theorem (1.1)]{GruWid} that   $E_\Omega(\cdot,z)$ is bounded on $\Omega\setminus\overline{B}_1$ by a constant depending only on $N$, $c$, $C$  and $\mathrm{dist}(z,\partial B_1)$, hence
    $\|E_\Omega(\cdot,z)_{|\Omega\setminus\overline{B}_1}\|_{L^2(\Omega\setminus\overline{B}_1)}$ is bounded  above by a constant depending on these quantities and
    on $m(\Omega)$. Thus, putting $r_1:=\min\{\mathrm{dist}(B,\partial\Omega)/2,
      \mathrm{dist}(z,\partial B)/4)\}$ and 
      $O$ for the  $r_1$-neighborhood of
    $\partial B$,  we get from the Cacciopoli inequality
    \cite[Lemma 2.105]{MalZie} and a compactness argument
    that $\|E_\Omega(\cdot,z)\|_{W^{1,2}(O)}$
    is bounded above in terms of the previous parameters. Thus,
    $\|E_\Omega(\cdot,z)\|_{W^{1/2,2}(\partial B)}$ is controlled by these parameters and the radius of $B$, by continuity of the trace operator and homogeneity of the derivative. 
     Subsequently, letting 
     $S$ be the outer boundary sphere of
     the shell $V:=(\Omega\setminus\overline{B})\cap O$,
    we can use the extension theorem to find $\Phi\in W^{1,2}(V)$ whose trace is
    $-E_\Omega(\cdot,z)_{|\partial B}$ on $\partial B$ and $0$ on $S$. By homogeneity of the derivative, 
      $\Phi$ (prolongated by zero on $\Omega\setminus(O\cup B)$) satisfies
      $\|\Phi\|_{W^{1,2}(\Omega\setminus\overline{B}))}\leq C\|E_\Omega(\cdot,z)_{|\partial B}\|_{W^{1/2,2}(\partial B)}$ where $C$ depends only on the radii of $B$, $S$, and standard estimates for the Dirichlet problem  now show that
      $\|E_\Omega(\cdot,z)\|_{W^{1,2}(\Omega\setminus\overline{B})}$ is bounded above
      in terms of the announced quantities, see \cite[Corollary 8.7]{GiTru}.


    Next, if   
    $z^\prime\to z$ in $B_1$ then  $\|E_\Omega(\cdot,z^\prime)-E_\Omega(\cdot,z)\|_{L^\infty(\Omega\setminus\overline{B}_1)}\to0$ by \cite[Theorem (1.9)]{GruWid} and the symmetry of $E$, 
    so  Cacciopoli's inequality together with
    the compactness of $\partial B$ imply
 that $\|E_\Omega(\cdot,z^\prime)-E_\Omega(\cdot,z)\|_{W^{1,2}(O)}\to0$,
    whence  $\|E_\Omega(\cdot,z^\prime)-E_\Omega(\cdot,z)\|_{W^{1/2,2}(\partial B)}\to0$ by continuity of the trace operator. So, the restriction to $\Omega\setminus\overline{B}$ of $E_\Omega(\cdot,z^\prime)-E_\Omega(\cdot,z)$ is a weak solution to
    $\mathcal{L}u=0$ with boundary values vanishing on $\partial\Omega$ and
    tending to zero in $W^{1/2,2}(\partial B)$
    as $z^\prime\to z$. Thus,
    the same estimates as before  (\cite[Corollary 8.7]{GiTru}) 
    imply that $\|E_\Omega(\cdot,z^\prime)-E_\Omega(\cdot,z)\|_{W^{1,2}(\Omega\setminus\overline{B})}\to0$ when  $z^\prime\to z$.
    \end{proof}

    Hereafter, when restricting  $\mathcal{L}$ to some ball $B(0,R)\subset\Omega$, we put for simplicity $E_R(x,y)$ instead of $E_{B(0,R)}(x,y)$. 
    The  next lemma connects $E_{R_1}(x,y)$ and $E_{R_2}(x,y)$ when $R_1\leq R_2$.
    \begin{lemma}
      \label{inclusE}
      Let $\mathcal{L}$ be an elliptic operator satisfying \eqref{genEllop}-\eqref{sell}  on a bounded  Lipschitz open set $\Omega\subset\R^N$. For $y\in B(0,R_1)\subset B(0,R_2)\subset\Omega$, the restriction $E_{R_2}(\cdot,y)_{|B(0,R_1)}$ is equal to $E_{R_1}(\cdot,y)+v_{R_1}$, where  $v_{R_1}\in W^{1,2}(B(0,R_1))$  is  the
        weak solution  to
        $\mathcal{L}v=0$ having the same trace as $E_{R_2}(\cdot,y)$  on the sphere $\mathbb{S}(0,R_1)$.
      \end{lemma}
      Note that $E_{R_2}(\cdot,y)$ has a trace on $\mathbb{S}(0,R_1)$
      that lies in $W^{1/2,2}(\mathbb{S}(0,R_1))$, by Lemma \ref{reg2locb}. Hence, $v_{R_1}$
      uniquely exists by well-posedness of the Dirichlet problem in
      $W^{1,2}(B(0,R_1))$ with boundary data in  $W^{1/2,2}(\mathbb{S}(0,R_1))$
      (which follows for instance from \cite[Theorem 8.3]{GiTru} and
      the surjectivity of the trace $W^{1,2}(B(0,R_1))\to W^{1/2,2}(\mathbb{S}(0,R_1))$). The non-immediate  bit of Lemma \ref{inclusE} is that
       $E_{R_2}(\cdot,y)_{|B(0,R_1)}-E_{R_1}(\cdot,y)$ lies in $W^{1,2}(B(0,R_1))$.
      \begin{proof}
        By the discussion after \eqref{defvwm},
        $H:= E_{R_2}(\cdot,y)_{|B(0,R_1)}- E_{R_1}(\cdot,y)- v_{R_1}$ is a weak solution
        of $\mathcal{L}u=0$ that lies in $W^{1,p}_0(B(0,R_1))$ for $1\leq p<\frac{N}{N-1}$. If $\bs{\Sigma}\in C^\infty(\overline{\Omega})$, then $H\equiv0$
        by \cite[Theorem A5.1]{Ancona}. In the general case, consider 
the Green functions $E^{(s)}_{R_1}(\cdot,y)$ and  $E^{(s)}_{R_2}(\cdot,y)$  of a sequence  $\mathcal{L}^{(s)}$
of smooth mollifications  of $\mathcal{L}$,
as well as the sequence $v_{R_1}^{(s)}\in W^{1,2}(B(0,R_1))$ of weak solutions to
$\mathcal{L}^{(s)}v=0$ having the same trace as $E_{R_2}^{(s)}(\cdot,y)$  on  $\mathbb{S}(0,R_1)$. By \cite[Theorem 5.4]{LSW},   $E^{(s)}_{R_j}(\cdot,y)$
converges weakly to $E_{R_j}(\cdot,y)$ in $W_0^{1,p}(B(0,R_j))$ when $s\to+\infty$ for
$j=1,2$ and   $p\in(1,\frac{N}{N-1})$. Choose $\varepsilon>0$
with $\overline{B}(y,\varepsilon)\subset B(0,R_1)$. As the ellipticity constants  of $\bs{\Sigma}$ in \eqref{sell} also hold for the mollifications $\bs{\Sigma}^{(s)}$, 
Lemma \ref{reg2locb}  yields
that $\|E^{(s)}_{R_2}(\cdot,y)\|_{W^{1,2}(B(0,R_2)\setminus \overline{B}(y,\varepsilon))}$ is bounded  independently of $s$, hence some subsequence
converges weakly to
$E_{R_2}(\cdot,y)$ in $W^{1,2}(B(0,R_2)\setminus \overline{B}(y,\varepsilon))$.
Since  the trace map
$W^{1,2}(B(0,R_2)\setminus \overline{B}(y,\varepsilon))\to W^{1/2,2}(\mathbb{S}(0,R_1))$ is linear and continuous, it is weakly continuous so
we may assume that  $E_{R_2}^{(s)}(\cdot,y)_{|\mathbb{S}(0,R_1)}$ converges weakly to $E_{R_2}(\cdot,y)_{|\mathbb{S}(0,R_1)}$ in
$W^{1/2,2}(\mathbb{S}(0,R_1))$. Then, classical estimates for the Dirichlet problem imply that $v_{R_1}^{(s)}$ converges weakly  to $v_{R_1}$ in $W^{1,2}(B(0,R_1))$ 
when $s\to \infty$\footnote{By right invertibility of the trace operator there is $h^{(s)}$ with   trace $E_{R_2}^{(s)}(\cdot,y)_{|\mathbb{S}(0,R_1)}$
converging weakly in $W^{1,2}(B(0,R_1))$ to $h$ with trace $E_{R_2}(\cdot,y)_{|\mathbb{S}(0,R_1)}$. Trading  $v_{R_1}^{(s)}$ for $v_{R_1}^{(s)}-h^{(s)}$
and $v_{R_1}$ for $v_{R_1}-h$,  one is left to check that the weak solution
$u^{(s)}$ to
  $\mathrm{div}(\bs{\Sigma}^{(s)}\nabla u)=\mathrm{div}(\bs{\Sigma}\nabla h^{(s)})$ in $W^{1,2}_0(B(0,R_1))$ converges weakly to the one of
  $\mathrm{div}(\bs{\Sigma}\nabla u)=\mathrm{div}(\bs{\Sigma}\nabla h)$, say $u^*$. By
  \cite[Corollary 8.7]{GiTru} $u^{(s)}$ is  bounded in $W^{1,2}_0(B(0,R_1))$ independently of $s$, therefore we may assume that it converges weakly there
  to some $w$ as $s\to\infty$. Because 
  $\bs{\Sigma}^{(s)}$ converges uniformly to $\bs{\Sigma}$, it follows from \eqref{wes1} and the definition of weak convergence that $w=u^*$.}.
Taking now the weak limit of
 $H^{(s)}:= E^{(s)}_{R_2}(\cdot,y)_{|B(0,R_1)}- E^{(s)}_{R_1}(\cdot,y)- v^{(s)}_{R_1}$
in $W^{1,p}(B(0,R_1))$ yields that  $E_{R_2}(\cdot,y)_{|B(0,R_1)}-E_{R_1}(\cdot,y)-v_{R_1}(\cdot)=0$, by the smooth case already proven.
      \end{proof}
       
    We turn to the case $\Omega=\R^N$, for which
        the definition of very weak solutions needs adjustment. Recall that
    $Y^{1,2}(\R^N)=Y^{1,2}_0(\R^N)$ and let us simply write
    $G$ instead of $G_{\R^N}$ for the Green operator.
    When $g\in C_c(\R^N)$,
    we contend that 
    $G(g)$ is (locally H\"older) continuous. Indeed,
    on a ball $B(0,R)$ large enough to contain $\supp g$,
    the  function $G(g)$  differs from 
     $G_{B(0,R)}(g)\in W^{1,2}_0(B(0,R))$ by a weak solution $w\in W^{1,2}(B(0,R))$ of $\mathcal{L}w=0$. The latter is H\"older continuous  on $B(0,R/2)$ by
     De Giorgi-Nash estimates,
    while $G_{B(0,R)}(g)$ is H\"older continuous on $B(0,R)$
    \cite[Theorem 8.29]{GiTru}. 
Thus, $\psi$ is (H\"older)  continuous on $B(0,R/2)$, which proves our contention since $R$ was arbitrary large.
Now, for $\nu$ a finite measure compactly supported in $\mathcal{M}(\R^N)$,
we define a \emph{very weak solution} vanishing at $\infty$ 
of $\mathcal{L}u=\nu$ 
to be a $u\in L^1_{loc}(\R^N)$ such that
\begin{equation}
  \label{vwes12}
  \int ug\,dm=\int G(g)d\nu,\quad g\in C_c(\R^N).
    \end{equation}
   By the previous contention $G(g)$ is continuous, therefore it is integrable against $\nu$ and  
\eqref{vwes12} makes sense.
Clearly, a  very weak solution
is unique when it exists.

The next theorem deals with  Green functions  on $\R^N$.

      \begin{theorem}\label{kerell}
        Let 
                $\mathcal{L}$ be the elliptic operator from {\rm{\eqref{genEllop}}} with
        $\bs{\Sigma}:\R^N\to \R^{N\times N}$ a  symmetric measurable map satisfying  {\rm{\eqref{sell}}}. There exists a function $E:\R^N\times\R^N\to [0+\infty]$ such that
  \begin{itemize}
  \item[(i)]$E(x,y)=E(y,x)$ for $(x,y)\in\R^N\times\R^N$;
  \item[(ii)] It holds for some constant $C_0=C_0(\mathcal{L})>0$  and all $x,y\in \R^N$ that
    \begin{equation}
      \label{compNew}
      C_0^{-1} |x-y|^{2-N}\leq E(x,y)\leq C_0|x-y|^{2-N};
      \end{equation}
      \item[(iii)] $E:\R^N\times\R^N\setminus\mathfrak{D}\to\R$ is locally H\"older-continuous, where $\mathfrak{D}:=\{(x,x)\in\R^N\times\R^N\}$.
      \item[(iv)] $y\mapsto E(\cdot,y)$ is continuous from
      $\R^n$ into $W_{loc}^{1,p}(\R^N)$ for $1\leq p<\frac{N}{N-1}$;
    \item[(v)] if $O\subset\R^N$ is open, then $y\mapsto E(\cdot,y)$ is continuous from
      $O$ into $W_{loc}^{1,2}(\R^N\setminus \overline{O})$;
    \item[(vi)] if $\nu$ is a compactly supported measure in $\mathcal{M}(\R^N)$,
      then
      \begin{equation}
        \label{vwsR}
        u(x):=\int E(x,y)d\nu(y), \qquad x\in\R^N,
      \end{equation}
      is the very weak solution to
      $\mathcal{L}u=\nu$ on $\R^N$, and it lies in
      $W_{loc}^{1,p}(\R^N)$ for $1\leq p<\frac{N}{N-1}$. The integral in
      \eqref{vwsR}
      is well-defined in that $E(x,\cdot)$ lies in $L^1(|\nu|)$ for a.e. $x$,
            and as a function of $x$ it may also be regarded as a Bochner integral, valued in  $W_{loc}^{1,p}(\R^N)$\footnote{
         Since we integrate continuous vector-valued functions on compact sets, \cite[Theorem 3.27]{Rudin} provides us with an adequate notion of Bochner integral.}.
          
      \end{itemize}
    \end{theorem}
    Note that \eqref{vwsR} makes $E$ unique, for $E(\cdot,y)$
    is the very weak solution 
    of $\mathcal{L}u=\delta_y$.
    \begin{proof}
      That $E_R(\cdot,y)$ increases locally uniformly, when $R\to+\infty$,
      to a function $E(\cdot.y)$ valued in $W_{loc}^{1,p}(\R^N)$ for $1\leq p<\frac{N}{N-1}$ that  satisfies $(i)$ and $(ii)$        
      is stated in \cite[Section 7]{LSW}\footnote{The first version of this work at         https://inria.hal.science/hal-05779520 contains a detailed proof.}.  {\it A fortiori} then
      $E_n(\cdot,y)$ is a Cauchy sequence in $L^2_{loc}(\R^N\setminus\{y\})$, therefore also in $W^{1,2}_{loc}(\R^N\setminus\{y\})$ by the Caccioppoli inequality. Thus, $E_R(\cdot,y)$ converges to $E(\cdot,y)$ in $W^{1,2}_{loc}(\R^N\setminus\{y\})$, in particular  $E(\cdot,y)$
        is a weak solution of $\mathcal{L}u=0$ in $\R^N\setminus\{y\}$.  
        Let $B,B_1\subset\R^N$ be two balls  with                     $\overline{B}\cap \overline{2B}_1=\emptyset$, and pick
                        $R$ such that
        $B(0,R)\supset \overline{2B}_1\cup\overline{B}$.   
        From  De Giorgi-Nash estimates  we get
        that $|E_R(x_1,y)-E_R(x_2,y)|\leq C_5 |x_1-x_2|^\alpha$
        for $x_1,x_2\in \frac{1}{2}B_1$ and $y\in B$, where $\alpha$ and $C_5$ depend on the
        ellipticity constants and the radius of $B_1$ as well as on
        $\|E_R(\cdot,y)\|_{L^\infty(B_1)}$  which is uniformly bounded, by (ii) 
        and the inequality $E_R(\cdot,y)\leq E(\cdot,y)$; see \cite[Theorem 8.22]{GiTru}. Likewise,
        recalling (i), the same estimates imply that $|E_R(x,y_1)-E_R(x,y_2)|\leq C_6 |y_1-y_2|^\beta$
        for $x\in 2B_1$ and $y_1,y_2\in B$, where $\beta$ and $C_6$ depend on the
        ellipticity constants and the radius of $B$ as well as on
        $\mathrm{dist}(2B_1,B)$. Letting $R$ tend to $\infty$ yields both (v) and (iii).

        Furthermore, for  $R^\prime>R>|y|$, Lemma \ref{inclusE} implies that
        \begin{equation}
          \label{dilun}
          E_{R^\prime}(\cdot,y)_{|B(0,R)}=E_{R}(\cdot,y)+v_{R}^{(R^\prime)}(\cdot),
        \end{equation}
        with $v_{R}^{(R^\prime)}$   the
        weak solution in $W^{1,2}(B(0,R))$ of $\mathcal{L}v=0$ which has the same trace as
        $E_{R^\prime}(\cdot,y)$  on  $\mathbb{S}(0,R)$.
        By the previous step, $E_{R^\prime}(\cdot,y)_{|B(0,R)\setminus\overline{B}(y,\varepsilon)}$ tends
        inductively to $E(\cdot,y)_{|\R^n\setminus\overline{B}(y,\varepsilon)}$ in
        $W^{1,2}_{loc}(\R^N\setminus\overline{B}(y,\varepsilon))$ when  $R^\prime\to\infty$, for fixed $\varepsilon<R-|y|$. Taking the limit in \eqref{dilun}, we get by continuity of solutions of the Dirichlet problem with 
respect to boundary conditions that
\begin{equation}
  \label{reppp}
  E(\cdot,y)_{|B(0,R)}=E_{R}(\cdot,y)+w_{R}(\cdot),
  \end{equation}
where $w_{R}\in  W^{1,2}(B(0,R))$  is the weak solution to
        $\mathcal{L}w=0$ having the same trace as $E(\cdot,y)$  on $\mathbb{S}(0,R)$. Relation \eqref{reppp} implies at once that
$E(\cdot,y)\in W^{1,p}_{loc}(\R^N)$ for
$p\in[1,\frac{N}{N-1})$. It also entails that
$y\mapsto E(\cdot,y)$ is continuous
from $\R^N$ into $W^{1,p}_{loc}(\R^N)$, because we stressed after \eqref{vwes1} that $y\mapsto E_{R}(\cdot,y)$ is continuous $B(0,R)\to W^{1,p}(B(0,R))$ and we know from (v) that
$y\mapsto E(\cdot,y)$ is continuous $B(y_0,\varepsilon)\to W^{1,2}(B(0,R)\setminus\overline{B}(y_0,\varepsilon))$ for $y_0\in B(0,R)$ and $\varepsilon$ small enough,
whence the trace $E(\cdot,y)_{|\mathbb{S}(0,R)}$ depends continuously of
$y\in B(y_0,\varepsilon)$ in $W^{1/2,2}(\mathbb{S}(0,R))$ and we can appeal again  to the continuity of solutions to the Dirichlet problem in $W^{1,2}(B(0,R))$ with respect to boundary conditions.
 This proves (iv).


 To establish (vi), fix $g\in C_c(\R^N)$ and $R_0$ such that
 $\mathrm{supp}\,g\subset B(0,R_0)$, so we may  regard $g$ as a member of $C_c(B(0,R))$ for $R\geq R_0$. Put for simplicity 
$\phi_R:= G_{B(0,R)}(g)$ and $\widetilde{\phi}_R$  for  its extension by zero to
  the whole of $\R^N$, which lies in $W^{1,2}(\R^N)$.
  Let also  $\psi:=G(g)\in Y^{1,2}(\R^N)$.
\emph{We claim} that
  $\widetilde{\phi}_R\to\psi$  as $R\to\infty$, 
  locally uniformly on $\R^N$. To see this, observe
from  \eqref{sell} and the definition of  $\widetilde{\phi}_R$ that for $R\geq R_0$ it holds
  \begin{align}
    \nonumber
    \|\nabla\widetilde{\phi}_R\|^2_{L^2(\R^N)}=\|\nabla\phi_R\|^2_{L^2(B(0,R))}\leq c^{-1}\int_{B(0,R)}(\nabla\phi_R)^t\bs{\Sigma}\nabla\phi_R=-c^{-1}\int_{B(0,R)}g\phi_R\\
   \nonumber
    \leq c^{-1}\|g\|_{L^{\frac{2N}{N+2}}(B(0,R))}\|\phi_R\|_{L^{\frac{2N}{N-2}}(B(0,R))}\leq  \frac{C_S}{c}\|g\|_{L^{\frac{2N}{N+2}}(B(0,R))}\|\nabla\phi_R\|_{L^2(B(0,R))}\\
    =\frac{C_S}{c}\|g\|_{L^{\frac{2N}{N+2}}(\R^N)}\|\nabla\widetilde{\phi}_R\|_{L^2(\R^N)},\label{chderf}
\end{align}                                                
where the second equality uses definition \eqref{wes1}
and the fact  that $\Phi_R$ lies in $W^{1,2}_0(B(0,R))$,
while the second inequality is H\"older's and the last follows from 
Sobolev's inequality with $C_S$ a dimensional constant, see \cite[Corollary 1.57]{MalZie}. The estimate \eqref{chderf} entails that
$\|\nabla\widetilde{\phi}_R\|_{L^2(\R^N)}$ is uniformly bounded with respect to $R$, hence from any sequence $R_n$ tending to $\infty$ we can extract a subsequence $R_{k_n}$ such that $\widetilde{\phi}_{R_{k_n}}$ converges weakly  to some $\Phi$ in $Y^{1,2}(\R^N)$; we endow $Y^{1,2}(\R^N)$ with  the equivalent norm $\|f\|^2_{\bs{\Sigma}}:=\int(\nabla f)^t\bs{\Sigma}\nabla f dm$. For each
$\varphi\in C^\infty_c(\R^N)$ and $n$  so
large that $B(0,R_{k_n})\supset \supp g\cup \supp\varphi$, we get by definition of weak
solutions that
\[
  \int_{\R^N}(\nabla \psi)^t\bs{\Sigma}\nabla \varphi=-  \int_{\R^N}g\varphi
  =\int_{B(0,R_{k_n})}(\nabla \phi_{R_{k_n}})^t\bs{\Sigma}\nabla \varphi
  =\int_{\R^N}(\nabla \widetilde{\phi}_{R_{k_n}})^t\bs{\Sigma}\nabla \varphi,
\]
and since the rightmost term tends to
$\int_{\R^N}(\nabla \Phi)^t\bs{\Sigma}\nabla \varphi$ by weak convergence,
we conclude that $\psi=\Phi$ because $ C^\infty_c(\R^N)$ is dense in $Y^{1,2}(\R^N)$. Thus, 
$\psi$  is the weak limit of $\widetilde{\phi}_R$ in $Y^{1,2}(\R^N)$ as $R\to+\infty$, and we record for later use that  weak convergence  $\widetilde{\phi}_R\to \psi$  also holds in $L^{2^*}(\R^N)$ with $2^*=2N/(N-2)>1$, since
$Y^{1,2}(\R^N)$ imbeds continuously in  $L^{2^*}(\R^N)$. To show that
$\psi$ is also the pointwise limit of  $\widetilde{\phi}_R$, locally uniformly in $\R^N$, it is enough to consider the case where $g\geq0$. Then, since $\phi_R$ is {\it a fortiori} a very weak solution to $\mathcal{L}\phi=g$ in $B(0,R)$ for $R\geq R_0$, the formula
\begin{equation}
  \label{repvwdw}
  \phi_R(x)=\int E_R(x,y)g(y)dm(y),\qquad x\in B(0,R),
  \end{equation}
implies that
$\widetilde{\phi}_R$ increases with $R$. Since it is bounded in $Y^{1,2}(\R^N)$, by \eqref{chderf}, it is also bounded in $L^{2^*}(\R^N)$ and by monotone convergence $\widetilde{\phi}_R$ converges in $L^{2^*}(\R^N)$
to some  $\Upsilon$. In particular, from any sequence $R_n$ tending to $\infty$ we can extract a subsequence $R_{k_n}$ such that $\widetilde{\phi}_{R_{k_n}}$ converges pointwise a.e.  to $\Upsilon$; 
necessarily $\Upsilon=\psi$ because the weak limit must be the pointwise limit when both exist.
  Finally, since $\psi=G(g)$ and $\widetilde{\phi}_{R_{k_n}}=G_{B(0,R_{k_n})}(g)$ are continuous,
  the pointwise convergence is locally uniform by Dini's theorem.
  \emph{This proves the claim}.

  Now, define $u$ as in \eqref{vwsR}.
  That it may be construed as a Bochner integral
    in $W^{1,p}_{loc}(\R^N)$ for $p\in[1,\frac{N}{N-1})$
  can be argued as in the case of \eqref{defvwm}, using (iv)
  and the compactness of $\supp\,\nu$.
  Put $\widetilde{E}_R(\cdot,y)$ for the extension by $0$
  to $\R^N$ of $\widetilde{E}_R(\cdot,y)$, and
pick $R_0$ large enough  that $\supp\nu\subset B(0,R_0)$.
Observing from \eqref{reppp} and the discussion before it
that $E(\cdot,y)-\widetilde{E}_{R}(\cdot,y)$
converges to zero in $W^{1,2}_{loc}(\R^N)$ and therefore also in  $W^{1,p}_{loc}(\R^N)$, locally uniformly with respect to $y$   when $R\to+\infty$, we deduce from  \cite[Chapter 3, Theorem 27]{Rudin} that
    \begin{align}
    \nonumber
    \int_{\R^N} u(x)g(x)dx&=\int\left(\int_{\R^N} g(x)E(x,y)dx\right)d\nu(y)=
                    \lim_{R\to\infty}\int\left(\int_{\R^N}g(x)\widetilde{E}_R(x,y)dx\right)d\nu(y)\\
      \nonumber                  &=\lim_{R\to\infty}\int\left(\int_{B(0,R)}g(x)E_R(x,y)dx\right)d\nu(y)= \lim_{R\to\infty}\int_{B(0,R)}\!\!\!\!\!\!G_{B(0,R)}(g)(y)d\nu(y)\\
                          &=\int G(g)(y)d\nu(y),\label{vwsolRN}
    \end{align}
    where the next-to-last inequality uses \eqref{repvwdw} and (i),
    while the last follows from an earlier claim
    that $G_{B(0,R)}(g)$ converges  to $G(g)$ locally uniformly as $R\to\infty$.
    Equation \eqref{vwsolRN} tells us that
$u$ is the very weak solution to $\mathcal{L}u=\nu$ in $\R^N$. Moreover,
$\int_{\R^N} E(x,y)d\nu(y)$ converges absolutely for a.e. $x$,
as follows from Fubini's theorem when we restrict to $g\geq0$ and $\nu\geq0$ in
\eqref{vwsolRN}, which is enough to get the same conclusion for general $g$ and $\nu$.
\end{proof}  

The function $E(\cdot,y)$ in  Theorem \ref{kerell} is called 
 the \emph{Green function} (or  \emph{fundamental solution}) of
 $\mathcal{L}$ in $\R^N$ with pole at $y$. We observe from Theorem \ref{kerell}
 points $(i)$  and $(vi)$ 
 together with  the definition \eqref{wes1} of very weak solutions and Fubini's theorem  that
 $G(g)(x)=\int E(x,y)g(y)dm(y)$ for $g\in C_c(\R^N)$.
 \subsection{Equation $\nabla\cdot(\bs{\Sigma}\nabla u)=\mathrm{div}\bs{\mu}$ with $\bs{\mu}$ a $\R^N$-valued  measure.}
 \label{EEae}
 The Green function  is instrumental to solve $\mathcal{L}u=\nu$ for $\nu$ a measure, but
our interest presently lies with solutions to 
$\mathcal{L}\phi =\mathrm{div}\,\bs{\mu}$ when $\bs{\mu}$ is a $\R^N$-valued 
compactly supported measure. In view of \eqref{vwsR}, one may surmise that
the solution
vanishing at infinity
can be written as 
      \begin{equation}
        \label{vwsRv1p}
        \Phi(\bs{\mu})(x)=\Phi(x)=\int \nabla_y E(x,y)\cdot d\bs{\mu}(y), \qquad \mathrm{a.e.\ }x\in\R^N.
           \end{equation}
      However, if $\bs{\Sigma}$ is  simply measurable and uniformly elliptic,
      then the very definition of \eqref{vwsRv1p} raises issues that we shall
      not take up here.
      Instead, we  shall assume that $\bs{\Sigma}$  not only satisfies
      {\rm \eqref{sell}} but also  restricts  to  a Dini-continuous
      function on an open  neighborhood $V$ of $\supp\bs{\mu}$.
          Then it is not difficult to
          make sense of \eqref{vwsRv1p} as follows.

      Under this assumption,
      weak solutions  of $\mathcal{L}u=0$ 
      on an open set $W\subset V$ belong to $C^1(W)$ if they lie in
      $W^{1,2}_{loc}(W)$, and when
      $\overline{W}\subset V$ then the modulus of continuity of their derivatives in $W$
      is controlled by $\omega_{\bs{\Sigma}_{|V}}$, the distance from
      $\partial W$ to $\partial V$, the  constants
      in \eqref{sell} and the dimension \cite[Theorem 1.1]{YLi}. In particular,
      for $O$ an open set such that $\overline{O}\cap\supp\bs{\mu}=\emptyset$
      and  $W$ open with $\overline{W}\subset V$, the restriction to $W\setminus\overline{O}$ of
      $y\mapsto E(x,y)$ lies in $C^1(W\setminus\overline{O})$ and its gradient has locally uniform modulus of continuity with respect to $x\in O$.
                Hence, by Theorem \ref{kerell}
                points (i), (ii)  and (v), we get since $L^2$-convergent sequences have subsequences converging pointwise a.e. that if $x\to x_0$ in  $O$ then some subsequence of $y\mapsto\nabla_y E(x_n,y)$ converges pointwise a.e. to $y\mapsto\nabla_yE(x_0,y)$  on $W\setminus\overline{O}$, and that convergence also holds in $L^\infty_{loc}(W\setminus\overline{O})$  by
                local uniform continuity. Altogether,
                \begin{equation}
                  \label{contdy}
                  (x,y)\mapsto \left(E(x,y),\nabla_yE(x,y)\right)\quad\mathrm{lies\ in}\quad \left(C\left((\R^N\times V)\setminus\mathfrak{D}\right)\right)^{N+1}.
                \end{equation}
                For $y_0\in V$ and $U$ an open set, $y_0\not\in\overline{U}$, the membership \eqref{contdy}  and the fundamental theorem of calculus imply that the functions $u_{n,j}(\cdot):=\frac{E(\cdot,y_0+h_n\bb{e}_j)-E(\cdot,y_0)}{h_n}$,
                $h_n\in(0,\mathrm{dist}(y_0,U))$,  tend to 
                $\partial{y_j}E(\cdot,y_0)$ boundedly pointwise on $U$ when
                $h_n\to0$ for $j\in\{1,\cdots,N\}$. So, as $u_{n,j}$ is a weak solution  to $\mathcal{L}u=0$ in $W^{1,2}(U)$,
               Cacciopoli's inequality implies that it 
                converges in $W^{1,2}_{loc}(U)$ 
                to a  weak solution of $\mathcal{L}u=0$ which is but $\partial{y_j}E(\cdot,y_0)$.
                               Hence, by \eqref{contdy} and Cacciopoli's inequality  again, we get if $U$, $W$ are open in $\R^N$ with $\overline{W}\subset V$ that                \begin{equation}
                  \label{ws2der}
                  y\mapsto \partial_{y_j}E(\cdot,y)\quad\mathrm{is\ continuous\  from}\quad W\ \mathrm{into}\  W^{1,2}_{loc}(U)\quad\mathrm{provided\  that}\quad \overline{U}\cap \overline{W}=\emptyset.
                  \end{equation}

                    In another connection, let $B$  be a ball of radius $R$ with
                    $\overline{B}\subset V$.
     From  \cite[Theorem 3.3 point (iv) \& Theorem 3.4 point(ii)]{GruWid},   
      we get that 
$|\nabla_y E_{B}(x,y)|\leq c_1|x-y|^{1-N}$ for $x,y\in B$
              where the constant $c_1$ depends on $N$ and $c$, $C$
              in \eqref{sell}, along with  $\omega_{\bs{\Sigma}_{|V}}$
      and $R$. Assume without loss of generality that $B$ has center $0$, and recall from \eqref{reppp} the notation $w_{R}$
      that we rewrite $w_{R}^{(y)}$ to show  dependence on $y$. From \cite[Lemma 3.1]{GruWid} we get that $|\nabla w_{R}^{(y)}|\leq \frac{C_1}{R-|y|}\|w_{R}^{(y)}\|_{L^\infty(B)}$ where the constant $C_1$ depends on
      the same parameters as $c_1$. So, observing
      from \eqref{reppp} and the symmetry of $E_{R}$, $E$ that
      $w_{R}^{(y)}(x)=w_{R}^{(x)}(y)$, we get from the boundedness of
      $\|E(\cdot,y)\|_{L^\infty(\mathbb{S}(0,R))}$ for $y\in \frac{1}{2}B$ (see \eqref{compNew}) together with the maximum principle that 
      \begin{equation}
        \label{bderpE}
        |\nabla_y E(x,y)|\leq C^\prime|x-y|^{1-N},\qquad x,y\in \frac{1}{2}B,\qquad \overline{B}\subset V,
      \end{equation}
      where the constant $C^\prime=C^\prime(R,c,C,N,\omega_{\bs{\Sigma}_{|V}})$ is
      locally uniform with respect to $x,y$. 

     Now, if we set by convention $\partial_{y_j}E(x,y)=+\infty$
     when $x=y$ then the  functions $\partial_{y_j}E(\cdot,\cdot)$,
     valued in $\R\cup\{+\infty\}$, are Borel measurable
      on $\R^N\times V$ for $j\in\{1,\cdot,N\}$, since they are continuous off the diagonal by \eqref{contdy}.
     Let us  cover 
      $\supp\bs{\mu}$ with finitely many balls $(\frac{1}{4}B_k)_{1\leq k\leq m}$ of radius $\frac{1}{4} R_k$ with  $B_k\subset V$.  
      Then, for $y\in \frac{1}{4}B_k$, it  holds that $\frac{1}{4}B_k\subset \overline{B}(y,\frac{R_k}{2})\subset B_k\subset V$, hence  $B(y,\frac{R_k}{2})$ qualifies for $B$ in \eqref{bderpE} and
      a quick estimate using spherical coordinates yields
      \[
\int_{\frac{1}{4}B_k}|\partial_{y_j} E(x,y)|dm(x)\leq C^\prime_k\,\sigma_N\,R_k,
\]
where $\sigma_N$ is the surface of the unit sphere in $\R^N$ and  $C^\prime_k=C^\prime(\frac{R_k}{2},c,C,N,\omega_{\bs{\Sigma}_{|V}})$ is the constant appearing in \eqref{bderpE}. So, if we pick $R_0>0$ so large that $B(0,R_0)\supset\overline{V}$ and  integrate separately $\int_{B(0,R_0)}|\partial_{y_j} E(\cdot,y)|dm$ over 
      $\frac{1}{4}B_k$ and $B(0,R_0)\setminus \frac{1}{4}(B_k)$
      for fixed $y$, we get:
      \begin{align}
        \int \left(\int_{B(0,R_0)}|\partial_{y_j} E(x,y)|dm(x)\right) d|\mu_j|(y)\leq C^{\prime\prime}|\mu_j|(\R^N)<\infty\label{majpF}\\
        \quad\mathrm{with}\quad
        C^{\prime\prime}:=\sup_{y\in S}\|\nabla_yE(\cdot,y)\|_{L^\infty(B(0,R_0)\setminus
        \cup_k\frac{1}{4}B_k)}+\max_kC^\prime_k\sigma_NR_k.\label{defCpp}
      \end{align}
      Note that the $\sup_{y\in S}$ on the right hand side of  \eqref{defCpp} is
      finite because of \eqref{contdy}.
      Thus, by Fubini's theorem,  $\phi(x)$ in  \eqref{vwsRv1p}
      exists for a.e. $x\in B(0,R_0))$ as an absolutely convergent integral,
      hence also  for a.e. $x\in\R^N$ since $R_0$ was arbitrary large. Moreover, by \cite[Chapter II, Section 1 Theorem 2 \& Section 2 Theorem 2]{diestel1977vector}, $y\mapsto \partial_{y_j}E(\cdot,y)$  is $|\mu_j|$-measurable when regarded as a $L^p_{loc}(\R^N)$-valued function of $y$
      for  $1\leq p<\frac{N}{N-1}$
      and  so $\phi$ can  be regarded as a Bochner integral in that space.


      Next we show that
      $\lim_{|x|\to\infty}|\phi(x)|=0$. Since $\supp\bs{\mu}$ is compact, it is enough to prove that
      $\lim_{|x|\to\infty}|\nabla_y E(x,y)|=0$, locally uniformly with respect to $y\in S$. For this, fix $y_0\in S$ together with $r_0>0$ such that
      $B(y_0,r_0)\subset V$, and for $x\notin V$ observe that 
      $u(z):= E(x,z)$ is a solution of $\mathcal{L}u=0$ in $B(y_0,r_0)$.
      By \cite[Lemma 3.1]{GruWid}, we get that
      \begin{equation}
        \label{estfreeze}
        |\nabla u(z)|\leq K\frac{\|u\|_{L^\infty(B(y_0,r_0))}}{\mathrm{dist}(z,\mathbb{S}(y_0,r_0))}\max\left\{2,\frac{2r_0}{d_0}  \right\},\qquad z\in B(y_0,\frac{r_0}{2}),
      \end{equation}
      where $K$ is a  constant and
      \begin{equation}
        \label{evmod}
d_0:=\sup\left\{d: \int_0^d\frac{\omega_{\bs{\Sigma}}(t)}{t}dt  \leq\frac{1}{4K}\right\};
      \end{equation}
      see \cite[Equation (3.10)]{GruWid}. It is clear from \eqref{compNew}
      that $\|u\|_{L^\infty(B(y_0,r_0))}\to0$
      when $|x|\to\infty$, therefore \eqref{estfreeze} implies that
      $\lim_{|x|\to\infty}|\nabla_y E(x,y)|=0$, uniformly for $y\in B(y_0,\frac{r_0}{2})\cap S$.

In analogy with \eqref{vwes12}, we say that $u$ is a \emph{very weak} solution
to $\mathcal{L}u=\mathrm{div}\,\bs{\mu}$ on $\R^N$ if
$u\in L^1_{loc}(\R^N)$ and 
\begin{equation}
  \label{vwes13}
\int  u\psi\,dm=\int  \nabla G(\psi)\cdot d\bs{\mu},\quad \psi\in C_c(\R^N).
\end{equation}
Note that \eqref{vwes13} makes sense when $\bs{\Sigma}$  is Dini-continuous
on an open  neighborhood $V$ of $\supp\bs{\mu}$,
as it follows from \cite[Remark 1.2]{YLi} that 
$G(\psi)$  is $C^1$-smooth in $V$ since $\bs{\Sigma}$  is Dini-continuous
there. Clearly, such a very weak solution is unique 
if it exists.

\begin{theorem}
        \label{soldim1} Let $\bs{\Sigma}:\R^N\to \R^{N\times N}$
        satisfy {\rm \eqref{sell}} and be Dini-continuous in an open set
         $V$. Let $\bs{\mu}\in(\mathcal{M}(\R^N))^N$ be a $\R^N$-valued measure compactly supported in $V$, and consider the differential operator $\mathcal{L}$ given by  {\rm\eqref{genEllop}} with  fundamental solution $E$. Then, 
                $\Phi$ defined a.e. by \eqref{vwsRv1p}
                lies in  $L^p_{loc}(\R^N)$ for $1\leq p<\frac{N}{N-1}$, and it                 is  the very weak solution to $\mathcal{L}u=\mathrm{div}\,\bs{\mu}$ on $\R^N$.  The function $\Phi$ is locally H\"older-continuous on $\R^N\setminus\supp\bs{\mu}$, it  tends to zero at infinity and the restriction $\Phi_{|\R^N\setminus\supp\bs{\mu}}$ is a weak solution to $\mathcal{L}\Phi=0$ in $W^{1,2}_{loc}(\R^N\setminus\supp\bs{\mu})$.
                Moreover, if $S\subset V$ is compact, the map $\bs{\mu}\mapsto \Phi_{|\R^N\setminus S}$ is sequentially continuous $\mathcal{M}(S)^N\to W^{1,2}_{loc}(\R^N\setminus S)$ when  $\mathcal{M}(S)^N$ receives the weak-star topology.
              \end{theorem}
      \begin{proof}
        By the discussion afer Equation \eqref{contdy}, $\Phi$ is well-defined both pointwise a.e. and as a Bochner integral in  $L^p_{loc}(\R^N)$ for $1\leq p<\frac{N}{N-1}$. Moreover
 it is a weak solution to $\mathcal{L}u=0$ in $W^{1,2}_{loc}(\R^N\setminus\supp\bs{\mu})$. In particular it is locally H\"older-continuous on $\R^N\setminus\supp\bs{\mu}$ by the De Giorgi-Nash  estimates, and if $\bs{\mu}_n$
        converges weak-star to $\bs{\mu}$ in  $\mathcal{M}(S)^N$, it follows from \eqref{vwsRv1p} and \eqref{contdy} that $\Phi(\bs{\mu}_n)$ converges locally uniformly to $\Phi(\bs{\mu})$ on $\R^N\setminus S$, hence also in  $W^{1,2}_{loc}(\R^N\setminus S)$ by Cacciopoli's inequality.
        We showed as well that
        $\lim_{|x|\to\infty}|\Phi(x)|=0$. Finally, to see that $\Phi$ is a very weak solution of \eqref{eqelf},
    fix $j\in\{1,\cdots,N\}$ and pick $\psi,\varphi\in C^\infty_c(\R^N)$ to write
    \begin{align}
      \label{dd1}
      \int \partial_j\varphi(x)\left(\int E(x,y)\psi(y) dm(y)\right)dx= \int \psi(y)\left(\int \partial_j\varphi(x)E(x,y)dx\right)dy\\
      \label{dd2}
      =-\int \!\!\psi(y)\left(\int\varphi(x) \partial_{x_j}
      E(x,y)dx\right)dy
       =-\int \varphi(x)\left(\int \partial_{x_j}E(x,y)\psi(y)\,dy\right)dx,
      \end{align}
      where the repeated use of Fubini's theorem 
      is justified by
      Theorem \ref{kerell} points (iv) and (i),
      while the second equality
      follows by definition of distributional derivatives.
            The chain of equations \eqref{dd1}-\eqref{dd2} means that the distributional derivatives of $G(\psi)$ can be computed by differentiating 
      under the integral sign:
\begin{equation}
  \label{derpsi23}
\partial_jG(\psi)(x)=\int \partial_{x_j}E(x,y)
  \psi(y)\,dy.
\end{equation}
Subsequently we can write
        \begin{align}
          \nonumber
          \int  \!\!\phi\psi\,dm&=\!\!\int\!\!\left(\int \!\!\nabla_y E(x,y)\!\cdot \!d\bs{\mu}(y)\!\!\right)\!\psi(x)dm(x)\!=\!\Sigma_{j=1}^N\!\int\!\!\left(\int \!\!\partial_{y_j} E(x,y)\psi(x)dm(x)\!
          \!\right)d\mu_j(y)\\
                                \nonumber& =      \Sigma_{j=1}^N\int \partial_{y_j}G(\psi)(y)d\mu_j(y)
          =\int \nabla G(\psi)(y)\cdot d\bs{\mu}(y),          
        \end{align}
        where the  appeal to Fubini's theorem to get the second equality
        is warranted by \eqref{majpF}, and we used \eqref{derpsi23} together with the symmetry of $E$ to get the third one.
      \end{proof}
      We complement Theorem \ref{soldim1} by
      strengthening  the nature of
      solutions when  $\bs{\Sigma}_{|V}$ belongs to $(W^{1,l}(V))^{N\times N}$
      for some $l>N$ on an open  neighborhood $V$ of $\supp\bs{\mu}$.
      In this case $\bs{\Sigma} \nabla \phi$ defines a distribution there,
      acting on
              $\varphi\in C_c^\infty(V)$ by  $\langle \bs{\Sigma} \nabla \phi,\varphi\rangle=\langle \phi,\div(\bs{\Sigma}\nabla\varphi)\rangle$, which is well-defined
              since $\bs{\Sigma}\nabla\varphi$ lies in
      $W^{1,l}(V,\R^N)$ and is compactly supported while $\phi$ lies
      in $L^p_{loc}(\R^N)$ for $1\leq p<\frac{N}{N-1}$. Moreover, $\bs{\Sigma}_{|V}$ is then locally Hölder-continuous by the Sobolev embedding theorem and hence, shrinking $V$ if necessary, we may as well assume that
      it is Dini-continuous on $V$ in which case
      the hypotheses of Theorem \ref{soldim1} are satisfied.
\begin{corollary}
                \label{corDinib}
                Assume in Theorem \ref{soldim1} that
                $\bs{\Sigma}_{|V}$ is not just Dini-continuous 
                but lies in $(W^{1,l}(V))^{N\times N}$
                for some $l>N$.  Then,
                 $\mathcal{L}\phi=\mathrm{div}\bs{\mu}$ in the distributional sense on $V$:
      \begin{equation}
        \label{distdive}
        \mathrm{div}\left(\bs{\Sigma} \nabla \phi_{|V}\right)=\mathrm{div}\bs{\mu}
          \qquad\mathrm{as\ distributions\ on}\ V.
        \end{equation}
      \end{corollary}
     
\begin{proof}
  Under the foregoing assumptions,
        if $\varphi\in C_c^\infty(V)$ then the distribution $\mathcal{L}\varphi$ identifies with the   compactly supported function  $g:=\mathrm{div}\,(\bs{\Sigma}\nabla \varphi)
        \in L^l(V)$.        
        In particular $g$
        lies in $(Y^{1,2}(\R^n))^*$ so that $\varphi=G(g)$; here, $g$, $\varphi$ are prolongated by $0$ off $V$. One can check that
        \eqref{derpsi23} remains valid when $\psi$ is compacly supported in $L^l(\R^N)$, for the use of Fubini's theorem in \eqref{dd1} and \eqref{dd2} is still justified since  $\partial_{x_j}E(x,\cdot)\in L^p_{loc}(\R^N)$ for
        $1\leq p<\frac{N}{N-1}$, locally uniformly with respect to $x$
        (see \eqref{bderpE}).
        So, letting $\psi=g$ in  \eqref{derpsi23}, we can write
\begin{equation}
  \label{derpsi2}
\partial_j\varphi(x)=\int \partial_{x_j}E(x,y)
  \mathrm{div}\bigl(\bs{\Sigma}\nabla\varphi)(y)\,dy.
\end{equation}
Now, let $\phi_j:=\int \partial_{y_j}E(x,y)d\mu_j(y)$.
For $\varphi\in C^\infty_c(V)$, we get from  \eqref{derpsi2} that
\begin{align}
  \label{distsoldivm}
\int \partial_j\varphi(x) d\mu_j(x)=\int\left(\int \partial_{x_j}E(x,y)
  \mathrm{div}\bigl(\bs{\Sigma}\nabla\varphi)(y)\,dy \right)d\mu_j(x)\\
  \nonumber
  =\int\!\!\left(\int \partial_{x_j}E(x,y) d\mu_j(x)\right)\mathrm{div}\bigl(\bs{\Sigma}\nabla\varphi)(y)dy=\int\!\! \phi_j(y)
  \mathrm{div}\bigl(\bs{\Sigma}\nabla\varphi)(y)\,dy=\langle \mathrm{div}\bigl(\bs{\Sigma}\nabla \phi_j),\varphi\rangle,
  \end{align}
  where the symmetry of $E$ was used to substitute $\phi_j$ in the next-to-last equality and we appealed to
  the very definition of the distribution
  $\mathrm{div}\bigl(\bs{\Sigma}\nabla \phi)$
  to get the last one; here, the use of Fubini's theorem is justified as
  before by the fact that $\partial_{x_j}E(x,\cdot)\in L^p_{loc}(\R^N)$ for
  $1\leq p<\frac{N}{N-1}$, locally uniformly with respect to $x$,
  and Hölder's inequality,
  since $\varphi$ has compact suppport in $V$ and $\mathrm{div}\bigl(\bs{\Sigma}\nabla\varphi)$ lies in $L^l(V)$ with $l>N$. Summing \eqref{distsoldivm} over $j$
  yields \eqref{distdive}.
  \end{proof}

     \begin{corollary}
                \label{corLipab}
                 Assume in Theorem \ref{soldim1} that the restriction
                 $\bs{\Sigma}_{|V}$ lies in $C^1(V,\R^{N\times N})$  and has
                 Dini continuous derivative there.
                Then, for any bounded open set $O\subset\R^N$ satisfying
                $\overline{O}\cap S=\emptyset$, there is a bounded open set
                $U$  with $\overline{U}\cap \overline{O}=\emptyset$ and  $S\subset U\subset\overline{U}\subset V$ for which
                the maps $y\mapsto E(\cdot,y)$ and  $y\mapsto \nabla_y E(\cdot,y)$ are Lipschitz continuous
                from $U$ into $W^{1,2}(O,\R^N)$.
                              \end{corollary}
              \begin{proof}
                Pick any  $U$ as in the statement. We know
                from Theorem \ref{soldim1} that $E(x,\cdot)$ is a solution to $\mathcal{L}u=0$ in $W^{1,2}(U)$ whose norm in the latter space is bounded independently of $x\in O$. Because $\bs{\Sigma}$ is {\it a fortiori}
                Dini continuous on $V$, it follows from  \cite[Theorem 1.1]{YLi}
                that $y\mapsto \nabla_y E(x,y)$  belongs to $C(U,\R^N)$,  with
                modulus of continuity controlled by $\omega_{\bs{\Sigma}_{|V}}$,
                $\mathrm{d}(\partial U,\partial V)$, 
                the ellipticity constants in \eqref{sell} and the dimension;
                in particular it is majorized independently of $x\in O$ by a positive continuous function vanishing at $0$, and since
                $\|\nabla E(x,\cdot)\|_{L^2(U)}$ is uniformly bounded with respect to $x\in O$ we deduce that  $\|\nabla E(x,\cdot)\|_{L^\infty(U)}$ is also
                bounded uniformly in $x\in O$. In particular the map
                $y\mapsto E(\cdot,y)$ is Lipschitz from $U$ to $L^2(O)$.
                Besides,
                shrinking $U$ slightly if necessary,
                \cite[Theorem 8.8]{GiTru} entails since  $\bs{\Sigma}_{|U}$ is {\it a fortiori} Lipschitz that $E(x,\cdot)_{|U}\in W^{2,2}(U)$
                with norm controlled by $\|\nabla E(x,\cdot)\|_{W^{1,2}(V)}$, the ellipticity constants and $\mathrm{d}(\partial U,\partial V)$.
                Thus,
                for $i\in\{1,\cdot,N\}$, the function $w_{x.i}:=\partial_{y_i}E(x,\cdot)_{|U}$ is a weak solution in $W^{1,2}(U)$ of 
\begin{equation}
  \label{regLGTou}
\mathrm{div}\,\left(\bs{\Sigma}\nabla w_{x,i}\right)=-\mathrm{div}\left(\left(\partial_i\bs{\Sigma}\right)
\nabla E(x,\cdot)\right),\qquad x\in O.
\end{equation}
Let $\{h_n\}\subset\R$
be  a sequence tending to $0$ with $0<h_n<\mathrm{dist}(\overline{O},\overline{U})$. For $(x,y)\in O\times U$
                and $j\in\{1,\cdots,N\}$,  let us form the sequence 
                $v_{n,j}(x,y):=\frac{\nabla_yE(x+h_n\bb{e}_j,y)-\nabla_yE(x,y)}{h_n}$. 
                From  the sequential continuity assertion in Theorem
                \ref{soldim1} (applied with  $S=\overline{U}$) we deduce that $\|\nabla_y E(\cdot,y)\|_{W^{1,2}(O,\R^N)}$ is bounded independently of $y\in U$
                and therefore, by classical properties of difference quotients \cite[Lemma 7.23]{GiTru}, we get that $\|v_{n,j}(\cdot,y)\|_{L^2(O,\R^N)}$ is bounded independently of $y\in U$ and $n\in\mathbb{N}$. If we fix $y\in U$ some subsequence $v_{m_n,j}(\cdot,y)$ converges weakly to some $f_{y,j}\in L^2(U,\R^N)$, and
              of necessity $f_{y,j}(x)=\partial_{x_j}\nabla_yE(x,y)$ for if $\varphi\in C_c^\infty(O)$ then
              \begin{align}\nonumber
                \lim_n\int_Ov_{m_n,j}(x,y)\varphi(x)dm(x)=\lim_n\int_O\nabla_yE(x,y)\frac{\varphi(x-h_{m_n}\bb{e}_j)-\varphi(x)}{h_{m_n}}dm(x)\\
                \label{derquotf}
                =-\int_O\nabla_yE(x,y)\partial_{x_j}\varphi(x)dm(x),
\end{align}
where we noticed that $\frac{\varphi(x-h_{m_n})-\varphi(x)}{h_{m_n}}$ converges uniformly to
$-\partial_{x_j}\varphi(x)$ on $O$. Consequently, $v_{n,j}(\cdot,y)$ converges weakly to $\partial_{x_j}\nabla_yE(x,y)$ in $L^2(U)$.

In another connection, letting $[v_{n,j}]_i(x,\cdot):=\frac{w_{x+h_n\bb{e}_j,i}-w_{x,i}}{h_n}$ denote the $i$-th component of the vector-valued  function $v_{n,j}(x,\cdot)$, 
  we deduce by linearity from \eqref{regLGTou} that
\begin{equation}
  \label{regLGToub}
\mathrm{div}\,\left(\bs{\Sigma}\nabla [v_{n,j}]_i\right)=-\mathrm{div}\left(\left(\partial_j\bs{\Sigma}\right)
v_{n,j}(x,\cdot)\right).
\end{equation}
Consequently, $v_{n,j}(x,\cdot)$ is a weak solution in $W^{1,2}(U)$  of the $N\times N$ \emph{system} of equations:
\begin{equation}
  \label{syseq}
\sum_{\alpha=1}^N\partial_\alpha\left(\sum_{\beta=1}^N\mathcal{A}^{\alpha,\beta}\partial_\beta  v_{n,j}(x,\cdot) +B^\alpha v_{n,j}(x,\cdot)\right)=0
\end{equation}
where $\mathcal{A}^{\alpha,\beta}:=\mathrm{diag}\{\bs{\Sigma}_{\alpha,\beta}\}$ and
$B^\alpha :=\partial_j[\bs{\Sigma}]_\alpha$ with $[\bs{\Sigma}]_\alpha$ to mean the $\alpha$-th row of $\bs{\Sigma}$.
This system is strictly elliptic, by \eqref{sell}, and it has Dini-continuous
coefficients by assumption. Hence, it is a special case of \cite[Theorem 1.1]{DongXu} (the case of a single domain) that $v_{n,j}(x,\cdot)$ is
Lipschitz continuous on every $C^{1,\alpha}$-regular open set $U^\prime$ with $S\subset\overline{U}^\prime\subset U$, where $\alpha$ is arbitrary in $(0,1)$.
Moreover, in view of \cite[Remark 1.4]{DongXu},
the Lipschitz constant can be made uniform with respect to $x\in O$,
because so does the $W^{1,2}(U)$-norm of $v_{n,j}(x,\cdot)$ since  it is controlled in terms of  $\|v_{n,j}(x,\cdot)\|_{W^{2,2}(U)}$ by  properties of difference quotients,
and the latter norm can be majorized uniformly with respect to $x$ as pointed out earlier in the proof.  Integrating over $x$, we get since $\overline{O}$ is compact that $\|v_{n,j}(\cdot,y)-v_{n,j}(\cdot,y^\prime)\|_{L^2(O,\R^N)}\leq K|y-y^\prime|$ for
$y, y^\prime\in U^\prime$ with some constant $K$,  and passing to the weak limit
(which cannot increase the norm) we get that
$\|\partial_j\nabla_yE(\cdot,y)-\partial_j\nabla_yE(\cdot,y^\prime)\|_{L^2(O,\R^N)}\leq K|y-y^\prime|$. As $j$ was arbitrary in $\{1,\cdots,N\}$ a similar inequality holds for $\|\nabla\nabla_yE(\cdot,y)-\nabla\nabla_yE(\cdot,y^\prime)\|_{L^2(O,\R^{N\times N})}$, and since we pointed out earlier in the proof that the map
 $y\mapsto E(\cdot,y)$ is Lipschitz from $U$ to $L^2(O)$
we deduce it is Lipschitz from $U^\prime$ to $W^{1,2}(O)$.
Moreover, observing from \eqref{regLGTou} that,  for fixed $x\in O$, the
function $w_x(y):=\nabla_yE(x,y)$ lies in $W^{1,2}(U^\prime,\R^N)$ and
satisfies the same system of equations \eqref{syseq} as $v_{n,j}$, 
we get from  \cite[Theorem 1.1]{DongXu} again that $w_x$ is Lipschitz on any
$C^{1,\alpha}$-regular  open set $U^{\prime\prime}\subset U^\prime$ containing $S$, and likewise the Lipschitz constant can be made independent of $x\in O$. 
Integrating over $x$ as before, we find that $\|\nabla_yE(\cdot,y)-\nabla_yE(\cdot,y^\prime)\|_{L^2(O,\R^N)}\leq K^\prime|y-y^\prime|$ for
$y, y^\prime\in U^{\prime\prime}$. Renaming $U^{\prime\prime}$ as $U$, we find
altogether that $y\mapsto \nabla_y E(\cdot,y))$ is Lipschitz continuous
                from $U$ into $W^{1,2}(O,\R^N)$ as announced.
              \end{proof}

                \section{Some unique continuation results}
                \label{UC}
                 Let  $\bs{\Sigma}:\R^N\to\R^{N\times N}$ satisfy \eqref{sell}
                and be Dini-continuous in a neighborhood of a 
                compact set $S\subset\R^N$. Recall from \eqref{genEllop}
                the differential operator  $\mathcal{L}$  
                and consider Equation \eqref{eqelf}
                where $\supp\bs{\mu}\subset S$, whose (very weak) solution $\Phi$ uniquely exists
                by Theorem  \ref{soldim1}. Below, we pick a closed set
                $Q\subset\R^N\setminus S$ and check that Scenarios (S.1) and (S.2), set out in Section \ref{CondEUsec},
                cause $\Phi$ to vanish on $\R^N\setminus S$ if the restriction $\Phi_{|Q}$ does;  and  similarly  for $\nabla\Phi_{|Q}$.

Keeping in mind the definitions made after conditions \ref{conE}
in Section \ref{Sect2.1}, let us first consider
Scenario (S.1) assuming that $\bs{\Sigma}$ is piecewise Lipschitz on
$\R^N\setminus S$, say  with pieces $\{\mathcal{O}_j\}_{1\leq j\leq\ell}$.
By Theorem \ref{soldim1} the function $\Phi$  belongs to $C^{0,\alpha}_{loc}(\R^N\setminus S)$, and reference
   \cite[Theorem 1.1]{LiVog} implies that  $\nabla\Phi$ extends
   Hölder-continuously to
   $\partial\mathcal{O}_j\setminus \mathcal{V}(S)$,  $j\in\{1,\cdots,\ell\}$, for  any 
   neighborhood $\mathcal{V}(S)$ of $S$. Hence, the
   tangential gradient $\nabla_T\Phi$ is continuous on $\partial\mathcal{O}_{j}\setminus S$, and the normal derivative
   $\partial\Phi/\partial n^\pm$ is continuous from each side although a jump may occur across $\partial\mathcal{O}_j$. Moreover, since
   $\Phi$ is a weak solution of $\mathcal{L}\Phi=0$ in $W^{1,2}_{loc}(\R^N\setminus S)$, it follows from the Green formula that
   \begin{equation}
     \label{transcon}
     \bs{\Sigma}^+(x)\left(\frac{\partial\Phi}{\partial n^+}(x)\right )n^+(x)=\bs{\Sigma}^-(x)\left(\frac{\partial\Phi}{\partial n^-}(x)\right)n^-(x),\qquad x\in  \partial\mathcal{O}_j\setminus S,
     \end{equation}
   where $\bs{\Sigma}^\pm$ denote the limiting values of $\bs{\Sigma}$ from each side of  $\partial\mathcal{O}_j\setminus S$. So,
   $\partial\Phi/\partial n^+(x)=0$ if and only if    $\partial\Phi/\partial n^-(x)=0$, and it makes sense to say that $\nabla\Phi(x)$ vanishes
   when  $x\in\R^N\setminus S$.
   

\begin{lemma}
  \label{S1}
  Let  $\bs{\Sigma}:\R^N\to\R^{N\times N}$ satisfy \eqref{sell} and be  Dini-continuous in a neighborhood of a slender compact set $S$, as well as  piecewise Lipschitz on 
  $\R^N\setminus S$.
   Let further  $\Phi$ 
  be the very weak solution to $\mathcal{L}u=\mathrm{div}\,\bs{\mu}$ on $\R^N$
   given by  Theorem \ref{soldim1}, with
   $\bs{\mu}\in \mathcal{M}(S)^N$.
   If  $Q$ is a compact set  with $Q\cap S=\emptyset$ that contains an interior point or a nonempty $C^{1,\alpha}$-smooth
   hypersurface, then 
   $\Phi_{|Q}\equiv0$ or   $\nabla\Phi_{|Q}\equiv0$ implies
   $\Phi_{|\R^N\setminus S}\equiv0$.
    \end{lemma}
    \begin{proof}
      Let $\mathcal{O}_1,\cdots ,\mathcal{O}_l$ be the pieces of the Lipschitz decomposition of $\bs{\Sigma}$ on  $\R^N\setminus S$, so that
      \begin{equation}\label{decsigmaLip}
        \R^N\setminus S=\left(\cup_j\overline{\mathcal{O}}_j\right)\setminus \partial\left(\R^N\setminus S\right)=\left(\cup_j\overline{\mathcal{O}}_j\right)\setminus S
      \end{equation}
      where the first equality holds by definition of piecwise Lipschitzianity
      and the second because $\partial\left(\R^N\setminus S\right)\subset S$
      but no element of $S$ can lie in the left hand side. Recall also from Theorem \ref{soldim1} that $\Phi$ is a weak solution to $\mathcal{L}\phi=0$ in
      $\R^N\setminus S$.

      If $Q$ contains an interior point it must contain a nonempty
      ball $B(x_0,r_0)\subset\mathcal{O}_{j_0}$ for some $j_0$. Thus, $\Phi$ is  constant on $B(x_0,r_0)$ and since $\bs{\Sigma}$ is locally Lipschitz
      in $\mathcal{O}_{j_0}$ by assumption, we conclude from \cite[Theorem 1]{AronKS} that $\Phi$ is constant throughout $\mathcal{O}_{j_0}$, say
      $\Phi_{|\mathcal{O}_{j_0}}\equiv c_0$ (obviously
      $c_0=0$ if we assume that $\Phi_{|Q}\equiv0$,
       but   we do not yet know that $c_0$ is zero 
       if we rather assume that $\nabla\Phi_{|Q}\equiv0$). Suppose first that $\partial\mathcal{O}_{j_0}\subset S$. Then we either have
       $\mathcal{O}_{j_0}=\R^N\setminus S$, whence
       $c_0=0$ since $\Phi$ vanishes at infinity (see Theorem \ref{soldim1})
       and there is nothing more to prove, or else $\R^N\setminus S$ 
       consists of  two disjoint nonempty open sets  namely $\mathcal{O}_{j_0}$ and $\R^N\setminus(\overline{\mathcal{O}}_{j_0}\cup S)$ ($=\R^N\setminus(\mathcal{O}_{j_0}\cup S))$, which contradicts the slenderness of $S$. Suppose now
       that $\partial\mathcal{O}_{j_0}\not\subset S$, so that
       we can find $y\in \partial\mathcal{O}_{j_0}$ with
       $B(y,\rho)\cap S=\emptyset$ for some $\rho>0$. From \eqref{decsigmaLip}
       we see that there is $\mathcal{O}_{j_1}$ with $(\partial\mathcal{O}_{j_0}\cap\partial\mathcal{O}_{j_1})\setminus S\neq\emptyset$ and thus, by definition of piecewise Lipschitzianity, $(\partial\mathcal{O}_{j_0}\cap\partial\mathcal{O}_{j_1})\setminus S$ must contain a connected manifold $\mathcal{G}$ which is
       isometric to the graph of a $C^{1,\alpha}$-function $\psi:B(0,r)\to\R^N$.
                Since the limit of $\nabla\Phi$ from $\mathcal{O}_{j_0}$ is zero on $\mathcal{G}$, so is the limit from $\mathcal{O}_{j_1}$; see the  discussion surrounding \eqref{transcon}. We can now construct a H\"older-smooth domain $\mathcal{O}\subset \mathcal{O}_{j_1}$ such that $\partial\mathcal{O}\supset \mathcal{G}$, and since H\"older-smooth domains  are  special cases
         of Dini domains as defined in \cite{AdolEsc}, we deduce from
         \cite[Theorem 0.5 \& Remark p. 959]{AdolEsc} that $\Phi$ is constant
         on $\mathcal{O}$, and of necessity the constant is $c_0$.
         Thus, as $\bs{\Sigma}$ is Lipschitz on $\mathcal{O}_{j_1}$, it is 
         identically equal to $c_0$ there, by \cite[Theorem 1]{AronKS} again.
         So, $\Phi\equiv c_0$ on $\mathcal{O}_{j_0}\cup\mathcal{O}_{j_1}$, and if
         $(\overline{\mathcal{O}}_{j_0}\cup\overline{\mathcal{O}}_{j_1})\setminus S=\R^N\setminus S$ we are done because either $\mathcal{O}_{j_0}$ or $\mathcal{O}_{j_1}$  is unbounded so that necessarily $c_0=0$. If on the contrary
         $(\overline{\mathcal{O}}_{j_0}\cup\overline{\mathcal{O}}_{j_1})\setminus S$ is a strict subset of $\R^N\setminus S$, then there is $\mathcal{O}_{j_2}$
         with $\partial\mathcal{O}_{j_2}\cap(\partial\mathcal{O}_{j_1}\cup\partial\mathcal{O}_{j_0})\setminus S\neq\emptyset$. Thus, by definition of piecewise Lipschitzianity, either $(\partial\mathcal{O}_{j_0}\cap\partial\mathcal{O}_{j_2})\setminus S$ or $(\partial\mathcal{O}_{j_1}\cap\partial\mathcal{O}_{j_2})\setminus S$ must contain a connected manifold $\mathcal{G}_1$ which is
         isometric to the graph of a $C^{1,\alpha}$-function, and we conclude as before that $\Phi\equiv c_0$ on $\mathcal{O}_{j_2}$. Continuing in this way we eventually obtain that $\Phi$ is constant on $\R^N\setminus S$ ($=\left(\cup_j\overline{\mathcal{O}}_j\right)\setminus S$) and that the constant is zero
         since $\Phi$ vanishes at infinity, as desired.                           
         We turn to the case where $Q$ contains a $C^{1,\alpha}$ hypersurface,
         say $\Upsilon$. If $\Upsilon\cap\mathcal{O}_j\neq\emptyset$ for some $j$, a subset $\mathcal{G}_1$ of $\Upsilon$ can be found such that
         $\overline{\mathcal{G}}_1\subset \subset\mathcal{O}_j$ and 
          $\mathcal{G}_1$ is (isometric to) the graph of a $C^{1,\alpha}$-function
          defined, say on $B(0,r)$  for some $r>0$. Proceeding as we did above with $\mathcal{G}$, we can imbed $\mathcal{G}$ in $\partial\mathcal{U}$ where
          $\mathcal{U}$ is Hölder-smooth  and $\overline{\mathcal{U}}\subset\mathcal{O}_j$. Since $\nabla\Phi$ vanishes on $\mathcal{G}$ by assumption, we get
          from \cite[Theorem 0.5 \& Remark p. 959]{AdolEsc} that $\Phi$ is constant on $\mathcal{U}$ and we are back to the previous case 
          where $Q$ contains an interior point. If on the contrary $\Upsilon\subset\partial\mathcal{O}_j$ for some $j$, we may find a Hölder-smooth local graph
          $\mathcal{G}_1\subset \partial\mathcal{O}_j\setminus S$, and again we are back to the previous case.
      \end{proof}
      We turn to scenario (S.2), where the injectivity of the measurement operator is guaranteed by the following lemma.
      \begin{lemma}
        \label{S2} Let $\bs{\Sigma}:\R^N\to\R^{N\times N}$ be real analytic
        and satisfy \eqref{sell}. For $S$ a compact slender set, denote by
        $\Phi$ the very weak solution to
        $\mathcal{L}u=\mathrm{div}\,\bs{\mu}$ on $\R^N$
   given by  Theorem \ref{soldim1}, with
   $\bs{\mu}\in \mathcal{M}(S)^N$. If $Q=\overline{V}$ where
   $V$ is  a nonempty bounded open subset of a closed, nonsingular
   real analytic hypersurface $\mathcal{A}$   such that $S\cap \mathcal{A}=\emptyset$, then   $\Phi_{|Q}\equiv0$ or
   $\partial\Phi/\partial n^{\pm}_{|Q}\equiv0$ (either $+$ or $-$)
   implies $\Phi_{|\R^N\setminus S}\equiv0$.
 \end{lemma}
 \begin{proof}
   The real analyticity of $\Sigma$ implies that $\Phi$ is real analytic
   in $R^N\setminus S$. This well-known fact ensues for instance
   from  \cite[Lemma 2.3 \& Section 5]{MoNi}, once it is observed that
   $\Phi$ is a $C^\infty$-solution in the strong sense to $\mathcal{L}\phi=0$
   on $R^N\setminus S$ because it is a weak solution in $W^{1,2}_{loc}(\R^N\setminus S)$, by Theorem \ref{soldim1}, and  we can use
   \cite[Corollary 8.11]{GiTru}. Restricting to a  component containing a connected component of $V$, we may assume that $\mathcal{A}$ is connected. 
   Then, since $\mathcal{A}$ is analytic and nonsingular,
   $\Phi_{|Q}=0$ (resp.  $\partial\Phi/\partial n^{\pm}_{|Q}\equiv0$) implies
   that  $\Phi\equiv0$ (resp.  $\partial\Phi/\partial n^{\pm}\equiv0$)  on $\mathcal{A}$. By the Jordan-Brouwer theorem\footnote{A version valid for smooth complete,  possibly unbounded surfaces may be found as \cite[Theorem A.1]{BVHNS};
     it carries over    {\it mutatis mutandis} to higher dimension.}
   $\R^N\setminus\mathcal{A}$ has exactly two connected components of which $\mathcal{A}$ is the boundary.
   One of them contains $S$ and the  other, say  $\mathcal{U}$,  is disjoint from $S$. Hence, $\Phi$ is a solution  to the Dirichlet (resp. Neumann) problem $\mathcal{L}\phi=0$ on $\mathcal{U}$,    with zero boundary data. In the first case, the uniqueness of solutions to the Dirichlet
   problem; {\it e.g.} \cite[Corollary 8.2]{GiTru} implies that
   $\Phi_{|\mathcal{U}}\equiv0$. In the second case, we will use that
   solutions of the Neumann
   problem are unique up to an additive  constant, thus
    $\Phi_{|\mathcal{U}}\equiv c$  for some $c\in\R$ whence  $\Phi_{|\R^N\setminus S}\equiv c$ by analyticity, and necessarily $c=0$ since $\Phi$ vanishes at infinity. However,  the standard argument for uniqueness-up-to-a-constant
    is based on the Green  formula:
    \[
\int_{\mathcal{U}}(\nabla\Phi)^t\bs{\Sigma}\nabla\Phi dm=\int_{\mathcal{A}}\Phi\langle \bs{\Sigma}\nabla\Phi,n^+\rangle d\sigma 
\]
where $n^+$ is the normal to $\partial\mathcal{U}=\mathcal{A}$ pointing
outwards \cite[Section 5.5.2]{Demengel},
and in order to invoke it  we  still need to show that $\nabla\Phi$ lies in
$L^2(\mathcal{U})$  even if  $\mathcal{A}$ is  unbounded.
For this, let $r_0$ be large enough that   $S\subset B(0,r_0/2)$,
and pick $R>r_0$. From \eqref{estfreeze} and \eqref{compNew} we get that
$|\nabla_y E(x,y)|\leq C_2(R-r_0)^{2-N}$ for $y\in S$, uniformly with respect to $x\notin B(0,R)$, where $C_2$ depends on $r_0$, the ellipticity constants
and $\|{\mathrm D} \bs{\Sigma}\|_{L^\infty(B(0,r_0),\R^{N\times N})}$, with $\mathrm{D}$
to mean the derivative identified with a matrix-valued function.
Consequently $\nabla_yE(\cdot,y)\in L^2(\R^N\setminus B(0,R))$,
uniformly with respect to $y\in S$ provided that $R>r_0$, and therefore
$\Phi\in L^2(\R^N\setminus B(0,R))$ for such $R$, by \eqref{vwsRv1p}.
This and the analyticity of $\Phi$ off $S$ together show that  $\Phi\in L^2(\Pi)$ for any 
half-plane $\Pi$ whose closure is disjoint from $S$.
Assuming for simplicity that $\Pi=\{\bs{x}:x_n>0\}$ and putting $y_{\rho,\eta}:=(0,\cdots,0,\rho+\eta)^t$ for fixed $\eta>0$, we deduce from Cacciopoli's inequality that
$\|\nabla\Phi\|_{L^2(B(y_{\rho,\eta},\rho),\R^N)}\leq C_3$ where $C_3$ is independent of $\rho$. Letting $\rho\to\infty$ we get that $\nabla\Phi\in L^2(\Pi_\eta,\R^N)$
with $\Pi_\eta=\{\bs{x}:x_n>\eta\}$, and  since $\eta>0$ is arbitrary $\nabla\Phi$ lies in $L^2(\Pi^\prime,\R^N)$ for any 
half-plane $\Pi^\prime$ whose closure is disjoint from $S$. It ensues that 
$\nabla\Phi\in L^2(\R^N\setminus \mathcal{C},\R^N)$ for $\mathcal{C}$ an open  cube containing $S$. Because  $\mathcal{C}\cap\mathcal{U}$ has compact closure
disjoint from $S$ we get that $\nabla\Phi$ is bounded there by analyticity, and finally $\nabla\Phi$  lies in
$L^2(\mathcal{U},\R^N)$ as desired.
   \end{proof}

   \section{Gradients and quasi-gradients on a compact set}
   \label{Rgrad}
   We account in this appendix for Remark \ref{onrange} in Section
   \ref{Sect2.1}.
   For  $S\subset\R^N$, recall from Section \ref{preln} the space
   $\mathcal{G}(S)$ of continuous gradients on $S$ comprised of restrictions $\nabla\varphi_{|_S}$ with $\varphi\in C^{1}(\R^N)$, and its uniform closure in $C(S,\R^N)$ which is the space $\mathcal{Q}(S)$ of quasi-gradients on $S$;
   see \cite[Section 3.1.1]{SmirnovKhavin1998} for their density properties
   in $C(S,\R^N)$ when $S$ is compact.
   In the latter case $\mathcal{M}(S)^N$ is dual to $C(S,\R^N)$
   of which $\mathcal{Q}(S)$ is a closed subspace, and clearly a vector-valued
   measure in
   $\mathcal{M}(S)^N$ is orthogonal to $\mathcal{Q}(S)$ if and only if it is divergence free in the distributional sense on $\R^N$. Thus, by the Hahn-Banach theorem,
   $\mathcal{Q}(S)$ is comprised exactly of those members of
   $C(S,\R^N)$ that integrate to zero against all divergence-free measures in
   $\mathcal{M}(S)^N$. 
   Under rather mild assumptions on $S$, the lemma below implies that 
   quasi gradients are in fact gradients. More precisely, when all rectifiably connected components of $S$ are quasi-convex, Lemma \ref{lemma|constructg}
   implies that every ${\bs h}\in\mathcal{Q}(S)$  is a gradient because
   it integrates to zero against measures of the form $\bs{\pi}(\bs{\gamma})$  with
   $\bs{\gamma}$  a closed rectifiable curve in $S$
   (see definition \eqref{eq|Rgamma|old}), since such measures  are divergence free.
   
   We start with a preliminary observation: if $(E_i)_{i\in I}$ are
   the rectifiably connected components  of a compact set $S$, then  $E_i$ is closed if  it is quasi convex Hence, when  all $E_i$ are quasi convex, they are finite in number,
   for they partition $S$ into disjoint compact sets.
   
\begin{lemma}
	\label{lemma|constructg}
	Let   $S\subset\R^\N$ be a compact set, all rectifiably connected components of which are quasi-convex. 
	Suppose  that $h\in C(S,\R^N)$ integrates to zero against
    all measures    of the form $\bs{\pi}(\bs{\gamma})$  for
    $\bs{\gamma}$  a closed rectifiable curve in $S$,  where $\bs{\pi}(\bs{\gamma})$ is defined in \eqref{eq|Rgamma|old}.
    \\[-.1in]
          
    Then,  there is  a $C^1$-map $g:\R^\N\to\R$ such that $\grad g_{|S}=
    h$.
\end{lemma}
\begin{proof}
  Since $S$ is a finite union of disjoint quasi convex compact sets by a
  preliminary observation, it is enough by a partition of unity argument to establish the result for each rectifiably connected component. Thus, we assume that $S$ itself is quasi convex. 
    Fix $z_0$ in $S$ and, for  $z\in S$, let $\bs{\gamma}_z:[a,b]\to S$ be
    a rectifiable curve
    such that $\bs{\gamma}_z(a)=z_0$ and $\bs{\gamma}_z(b)=z$.
    Define a member of
    $\mathcal{M}(S)^\N$ by
    $\bs{\mu}_z:=\bs{\pi}(\bs{\gamma}_z)$,
    and put
    \[
		g(z):=\int_a^b {{\bs h}(\bs{\gamma}_z(t))}\cdot\bs{\gamma}_z^\prime(t)\,dt=\langle {\bs h},\bs{\mu}_z\rangle.
    \]
	Note that $g(z)$ does not depend on the choice of $\bs{\gamma}_z$ joining
	$z_0$ to $z$, by our assumption that $\int{\bs h}\cdot d\bs{\pi}(\widetilde{\bs{\gamma}})=0$ whenever $\widetilde{\bs{\gamma}}$ is a closed curve in $S$.
Now, for $z_1,z_2\in S$, we can write
\[
  g(z_2)=g(z_1)+{\bs h}(z_1)\cdot(z_2-z_1)+R(z_1,z_2)\]
where
\[
  R(z_1,z_2):=\int_{t_1}^{t_2}\left({{\bs h}(\bs{\gamma}(t))} -{{\bs h}(z_1)}\right)\cdot\bs{\gamma}^\prime(t)\,dt  \]
with $\gamma:[t_1,t_2]\to S$ an arc of length at most $C|z_1-z_2|$ joining $z_1$ to $z_2$; such an arc indeed exists, by the quasi-convexity of $S$.
From the continuity of ${\bs h}$ and the bound on the length of $\bs{\gamma}$,
it follows that $R(z_1,z_2)=o(|z_1-z_2|)$ as $z_1,z_2\to z\in S$ and therefore,
by the Whitney extension theorem \cite[Appendix A]{AbRo}, $g$ extends to a $C^1$
function on $\R^\N$ whose gradient on $S$ coincides with $\bs h$.
\end{proof}
We point out that Lemma \ref{lemma|constructg} does not hold for arbitrary  compact  $S\subset\R^N$. Consider indeed a plane topologist's sine curve:
\[\Gamma:=\{(x_1,x_2,0,\cdots,0)^t:x_1\in(0,1],\ x_2=\sin\frac{1}{x_1}\}\subset\R^N,\]
and define $S:=\Gamma\cup T$ where $T:=\{(0,x_2,0,\cdots,0)^t:\ x_2\in[0,1]\}$ is the segment joining $0$ to $(0,1,0\cdots,0)^t$ along the $x_2$ axis. Then $S$ is compact and has two rectifiably connected components, one of which ($\bs{\Gamma})$ is not quasi-convex. Now,  a divergence-free measure on $\R^N$ supported on $S$ must be divergence free also when restricted to the plane $P:=\{(x_1,x_2,0,\cdots,0)^t\}$, hence  it can be represented as a superposition of measures $\bs{\pi}(J)$ with $J\subset P$ a Jordan curve; see  \cite[p. 843]{Smi94} and  \cite{BG} or \cite[Theorem 2.8]{BVH} for details. Consequently $S$ carries no nonzero  divergence-free measure, because a Jordan curve supported on
$\Gamma$ or $T$ must roll back over itself and cancel out
the integral defining $\bs{\pi}( J)$.  In particular, \emph{any} ${\bs h}\in C(S,\R^N)$ integrates to zero against divergence-free measures in $\mathcal{M}(S)^N$. However,  if we pick ${\bs h}(\bs{x})=0$ when $\bs{x}\in T$ and
${\bs h}(\bs{x})=x_1{\bs v}(\bs{x})$ with ${\bs v}$ the unit tangent to
$\bs{\Gamma}$ oriented towards decreasing $x_1$ when  $\bs{x}\in \bs{\Gamma}$,
then ${\bs h}$ is continuous  and tangent to $\bs{\Gamma}$
but its integral along the latter is infinite. Thus, $\bs h$ 
cannot be the restriction to $S$  of the gradient of a  $C^1(\R^N)$-function.

\section{Declarations}
\subsection{Funding}
This research was funded in part by the Austrian Science Fund (FWF) SFB 10.55776/F68 ``Tomography Across the Scales". For open access purposes, the author has applied a CC BY public copyright license to any author-accepted manuscript version arising from this submission.
\subsection{Data Availability Statement}
No datasets were generated or analysed during the current study

\end{appendices}

\bibliography{C2D2}

\end{document}